\documentclass{amsart}
\usepackage[margin=1in]{geometry}
\usepackage{amsmath}
\usepackage{amsfonts}
\usepackage{amssymb}
\usepackage{amsthm}
\usepackage{stmaryrd}
\usepackage{graphicx}
\usepackage{multicol}
\usepackage{caption}
\usepackage{subcaption}
\usepackage{bm}
\usepackage{algorithm}
\usepackage{algpseudocode}

\usepackage{tikz, comment}

\usepackage{import}

\usepackage{hyperref}
\usepackage{cleveref}
\usepackage{enumitem}
\usepackage{empheq}

\newtheorem{theorem}{Theorem}

\newtheorem{lemma}[theorem]{Lemma}
\newtheorem{conj}[theorem]{Conjecture}

\newtheorem{cor}[theorem]{Corollary}
\newtheorem{proposition}[theorem]{Proposition}

\theoremstyle{definition}
\newtheorem{definition}[theorem]{Definition}
\newtheorem{remark}[theorem]{Remark}
\newtheorem{example}[theorem]{Example}

\numberwithin{theorem}{section}

\numberwithin{equation}{section}

\newcommand{\mb}[1]{\mathbf{#1}}

\newcommand{\Z}{\mathbb{Z}}
\newcommand{\C}{\mathbb{C}}

\newcommand{\ps}{\mathsf{ps}}
\newcommand{\st}{\mathsf{St}}

\renewcommand{\L}{\mathcal{L}}

\newcommand{\lan}{\langle}
\newcommand{\ran}{\rangle}

\newcommand{\tcr}[1]{\textcolor{red}{#1}}
\DeclareMathOperator{\mindeg}{mindeg}
\DeclareMathOperator{\maxdeg}{maxdeg}

\newcommand{\psgn}[2]{\ps_{#2}(Z_{#1}(\mathbf{x};t))}
\DeclareMathOperator{\PC}{\mathcal{PC}}

\DeclarePairedDelimiter{\floor}{\lfloor}{\rfloor}
\DeclarePairedDelimiter{\ab}{\lvert}{\rvert}

\DeclarePairedDelimiter{\ang}{\langle}{\rangle}

\newcommand{\computeClassOdd}[1]{%
  \gdef\coCountOdd{0}%
  \gdef\coCountEven{0}%
  \gdef\coSumEven{0}%
  \foreach \v [count=\i] in {#1} {%
    \ifodd\i
      \pgfmathtruncatemacro{\coCountOdd}{\coCountOdd+1}%
      \xdef\coCountOdd{\coCountOdd}%
    \else
      \pgfmathtruncatemacro{\coCountEven}{\coCountEven+1}%
      \xdef\coCountEven{\coCountEven}%
      \pgfmathtruncatemacro{\coSumEven}{\coSumEven+\v}%
      \xdef\coSumEven{\coSumEven}%
    \fi
  }%
  \pgfmathtruncatemacro{\classOdd}{\coCountOdd-\coCountEven+\coSumEven}%
  \xdef\classOdd{\classOdd}%
}

\newcommand{\normalizeTuple}[1]{%
  \foreach \v [count=\i] in {#1} {%
    \ifnum\i=1
      \xdef\ntFirst{\v}%
    \fi
    \xdef\ntLast{\v}%
  }%
  \pgfmathtruncatemacro{\ntNeedRev}{\ntFirst<\ntLast}%
  \ifnum\ntNeedRev=1
    \gdef\normList{}%
    \foreach \v in {#1} {%
      \ifx\normList\empty
        \xdef\normList{\v}%
      \else
        \xdef\normList{\v,\normList}%
      \fi
    }%
  \else
    \xdef\normList{#1}%
  \fi
}

\newcommand{\computePositions}[1]{%
  \gdef\ppPrevX{0}%
  \gdef\ppPrevL{0}%
  \gdef\ppFirst{1}%
  \foreach \v [count=\i] in {#1} {%
    \pgfmathtruncatemacro{\ppCurL}{\v-1}%
    \ifnum\ppFirst=1
      \pgfmathsetmacro{\ppCurX}{\dxBase}%
      \gdef\ppFirst{0}%
    \else
      \pgfmathsetmacro{\ppCurX}{\ppPrevX+\dxBase+\gapScale*(\ppPrevL+\ppCurL)}%
    \fi
    \expandafter\xdef\csname Xpos\i\endcsname{\ppCurX}%
    \xdef\ppPrevX{\ppCurX}%
    \xdef\ppPrevL{\ppCurL}%
  }%
}

\newcommand{\boldtuple}[1]{%
  \normalizeTuple{#1}%
  $($\def\sep{}%
  \foreach \v in \normList {%
    \sep
    \v
    \xdef\sep{,\,}%
  }$)$%
}

\newcommand{\caterpillar}[1]{%
  \normalizeTuple{#1}%
  \expandafter\computeClassOdd\expandafter{\normList}%
  \def\dxBase{0.75}
  \def\gapScale{0.12}
  \expandafter\computePositions\expandafter{\normList}%
  \begin{tikzpicture}[yscale=-1,scale=0.725, transform shape, baseline=-2.5pt,
      catnode/.style={circle,draw,inner sep=0pt,minimum size=5pt},
      every path/.style={line width=0.5pt}]
    \def\rad{0.75}  
    \foreach \v [count=\i] in \normList {%
      \coordinate (S\i) at (\csname Xpos\i\endcsname,0);
      \pgfmathtruncatemacro{\parityOdd}{mod(\i,2)}%
      \pgfmathtruncatemacro{\spineRed}{(\parityOdd==1)*(\classOdd==5)+(\parityOdd==0)*(\classOdd==3)}%
      \pgfmathtruncatemacro{\leafRed}{1-\spineRed}%
      \ifnum\spineRed=1
        \node[catnode,fill=black] (S\i) at (S\i) {};
      \else
        \node[catnode,fill=white] (S\i) at (S\i) {};
      \fi
      \ifnum\i>1
        \pgfmathtruncatemacro{\prev}{\i-1}
        \draw (S\prev) -- (S\i);
      \fi
      \pgfmathtruncatemacro{\L}{\v-1}
      \ifnum\L>0
        \foreach \j in {1,...,\L} {
          \ifnum\L=1
            \pgfmathsetmacro{\ang}{90}
          \else
            \pgfmathsetmacro{\ang}{60+(\j-1)*60/(\L-1)}
          \fi
          \pgfmathsetmacro{\offx}{\csname Xpos\i\endcsname+\rad*cos(\ang)}
          \pgfmathsetmacro{\offy}{\rad*sin(\ang)}
           \draw (S\i) -- (\offx,\offy);
          \ifnum\leafRed=1
            \node[catnode,fill=black] at (\offx,\offy) {};
          \else
            \node[catnode,fill=white] at (\offx,\offy) {};
          \fi

        }
      \fi
    }
  \end{tikzpicture}%
}

\newcommand{\plaincaterpillar}[1]{%
  \normalizeTuple{#1}%
  \expandafter\computeClassOdd\expandafter{\normList}%
  \def\dxBase{0.95}%
  \def\gapScale{0.12}%
  \expandafter\computePositions\expandafter{\normList}%
  \begin{tikzpicture}[yscale=-1,scale=0.725, transform shape, baseline=-2.5pt,
      catnode/.style={circle,draw,inner sep=0pt,minimum size=5pt},
      every path/.style={line width=0.5pt}]
    \def\rad{0.75}  
    \foreach \v [count=\i] in \normList {%
      \coordinate (S\i) at (\csname Xpos\i\endcsname,0);
      \node[catnode,fill=black] (S\i) at (S\i) {};
      \ifnum\i>1
        \pgfmathtruncatemacro{\prev}{\i-1}
        \draw (S\prev) -- (S\i);
      \fi
      \pgfmathtruncatemacro{\L}{\v-1}
      \ifnum\L>0
        \foreach \j in {1,...,\L} {
          \ifnum\L=1
            \pgfmathsetmacro{\ang}{90}
          \else
            \pgfmathsetmacro{\ang}{60+(\j-1)*60/(\L-1)}
          \fi
          \pgfmathsetmacro{\offx}{\csname Xpos\i\endcsname+\rad*cos(\ang)}
          \pgfmathsetmacro{\offy}{\rad*sin(\ang)}
           \draw (S\i) -- (\offx,\offy);
           \node[catnode,fill=black] at (\offx,\offy) {};
        }
      \fi
    }
  \end{tikzpicture}%
}

\usepackage[colorinlistoftodos]{todonotes}

\newcommand{\LINECOMMENT}[1]{\State \textit{\# #1}}

\renewcommand{\models}{\vDash} 

\title{Principal specializations of the {\MakeLowercase{t}}-chromatic symmetric function}

\author[A.~Bingham]{Aram Bingham}\address[A.~Bingham]{Departamento de Matem\'aticas, Universidad de Chile, 3425 Las Palmeras, Ñuñoa, Santiago, CL}
\email{\textcolor{blue}{\href{mailto:aram@matmor.unam.mx}{aram@matmor.unam.mx}}}

\author[K.~Celano]{Kyle Celano}\address[K.~Celano]{Department of Mathematics, Wake Forest University, NC}
\email{\textcolor{blue}{\url{celanok@wfu.edu}}}

\author[P.~Commins]{Patricia Commins}
\address[P.~Commins]{LACIM, D\'{e}partement de Math\'{e}matiques, Université du Québec \`a Montr\'{e}al, Canada}\email{\textcolor{blue}{\href{mailto:commins.patricia@uqam.ca}{commins.patricia@uqam.ca}}}

\author[R.~Orellana]{Rosa Orellana}
\address[R.~Orellana]{Mathematics Department, Dartmouth College,
Hanover, NH, USA}
\email{Rosa.C.Orellana@dartmouth.edu}
\urladdr{\href{https://math.dartmouth.edu/~orellana/}{https://math.dartmouth.edu/~orellana/}}

\author[J.~Pan]{Jianping Pan}
\address[J.~Pan]{School of Mathematical and Statistical Sciences, Arizona State University, Tempe, AZ, USA}
\email{\textcolor{blue}{\href{mailto:jianping.pan@asu.edu}{jianping.pan@asu.edu}}}

\author[M.C.~Simone]{Mary Claire Simone}
\address[M.C.~Simone]{Mathematics Department, Skagit Valley College, Mount Vernon, WA, USA}
\email{\textcolor{blue}{\href{mailto:MaryClaire.Simone@skagit.edu}{MaryClaire.Simone@skagit.edu}}}

\date{\today}

\keywords{chromatic symmetric function, principal specialization, caterpillar}
\subjclass{05E05, 05C60, 05C31, 05C05}

\begin{document}

\begin{abstract}
We introduce the $t$-chromatic symmetric function $Z_G(\mathbf{x};t)$, a variant of Stanley's chromatic symmetric function in which the variable $t$ records the number of distinct colors used in a proper coloring. Although $Z_G(\mathbf{x};t)$ is equivalent to the chromatic symmetric function as a graph invariant, its principal specializations retain substantially more information than those of the chromatic symmetric function. 
This construction is motivated by the work of Bajo--Beck--Vindas-Meléndez on the $q$-chromatic polynomial and the conjecture of Loehr--Warrington that the $n$th principal specialization of the chromatic symmetric function distinguishes trees on at most $n$ vertices, both of which stem from Stanley's longstanding question of whether the chromatic symmetric function distinguishes trees.
We determine the classes of caterpillars resulting from the second principal specialization  of $Z_G(\mathbf{x};t)$, showing that the second specialization already captures some of the structure of trees.
We then prove that the third principal specialization distinguishes and explicitly reconstructs caterpillars with three spine vertices as well as elongated bistars. Finally,  we conjecture, based on computational data, that the fourth principal specialization distinguishes all caterpillars.
\end{abstract}

\maketitle


\section{Introduction}
The chromatic symmetric function of a graph was introduced by Stanley in the mid-1990s as a symmetric-function analogue of the chromatic polynomial  \cite{stanley_1995}. Let $\mathbf{x}=(x_1, x_2, \ldots) $ be a sequence of commuting variables. For a simple graph $G=(V,E)$, the chromatic symmetric function of $G$ is
\[
X_G(\mathbf{x})=\sum_{\kappa\in\PC_G} \mathbf{x}^{\kappa},
\]
where $\PC_G$ denotes the set of proper colorings of the vertices of $G$ with colors from $\mathbb{N}$, and $\mathbf{x}^\kappa = \prod_{v\in V} x_{\kappa(v)}$. Although the chromatic polynomial, $\chi_G(k)$, records only the number of proper colorings using a set of $k$ colors, the chromatic symmetric function retains information about how the colors are distributed among the vertices. In particular, two graphs with the same chromatic polynomial can have different chromatic symmetric functions.

Since its introduction, several generalizations and specializations of the chromatic symmetric function have been developed, in part motivated by the question of how much information about a graph is encoded by its proper colorings. Gebhard and Sagan introduced a chromatic symmetric function in noncommuting variables, $\mathbf{Y}_G(\mathbf{y})$,  that retains information about which vertices receive which colors and satisfies a deletion-contraction relation, providing a refinement of Stanley's invariant \cite{GebhardSagan2001}. Aliste-Prieto, de Mier, and Zamora \cite{APdMZ2021} studied the $U$-polynomial of Noble and Welsh \cite{noble_weighted_1999}, which is equivalent to the chromatic symmetric function as an invariant of trees.   Loehr and Warrington introduced a rooted version of the chromatic symmetric function \cite{LoehrWarrington2024}. Together, these developments highlight the broader question of how much information about a graph can be recovered from its proper colorings and how refinements of the chromatic symmetric function can improve its ability to distinguish non-isomorphic graphs.

One of the central questions surrounding the chromatic symmetric function is whether it distinguishes trees. Despite the fact that all trees on the same number of vertices have the same chromatic polynomial, Stanley was unable to find two non-isomorphic trees with the same chromatic symmetric function \cite[Problem 7.47(b)]{stanley1999enumerative}. This led to what is now known as the Tree Isomorphism Conjecture: if $T$ and $T'$ are non-isomorphic trees, then
\[
X_T(\mathbf{x})\neq X_{T'}(\mathbf{x}).
\]
The conjecture has been verified for all trees with at most 29 vertices by Heil and Ji \cite{HeilJi18}, and several important families of trees are known to be distinguished by their chromatic symmetric functions \cite{AWvW24,ALISTEPRIETOZAMORA, Crew2022, arunkumar2023properqcaterpillarsdistinguishedchromatic, gonzalez2024chromatic, HurynChmutov20, LoeblSereni, martin2007distinguishingtreeschromaticsymmetric,WYZ2024, zeng2026multiplicity}. Nevertheless, the general conjecture remains open.

A natural approach to Stanley's question is to simplify the chromatic symmetric function while retaining enough information to distinguish trees. One such simplification is obtained through principal specialization. For a symmetric function $f(x_1, x_2, \ldots)$, its $n$th {\it principal specialization} is
\[
\ps_n(f)=f(1,q,q^2,\ldots,q^{n-1},0,0,\ldots).
\]
For the chromatic symmetric function, the coefficient of $q^i$ in $\ps_n(X_G)$ counts proper $n$-colorings of $G$ using colors $\{0, 1, 2, \ldots, n-1\}$ having a total chromatic sum of $i$. Thus, principal specialization turns a symmetric function in infinitely many variables into a univariate polynomial encoding a weighted enumeration of proper colorings.

Principal specializations have been studied in connection with the Tree Isomorphism Conjecture. Loehr and Warrington conjectured that $\ps_n(X_T)$ distinguishes trees with at most $n$ vertices \cite{LoehrWarrington2024}, and verified the conjecture for trees with at most 17 vertices. More recently, Bajo, Beck, and Vindas-Meléndez introduced the {\it $q$-chromatic polynomial} and showed that its leading coefficient is equivalent to the stable principal specialization of the chromatic symmetric function \cite{BBV}. They conjectured that this stable principal specialization distinguishes trees. These developments suggest that principal specialization may provide a useful way to extract structural information from the chromatic symmetric function and a way to approach the Tree Isomorphism Conjecture.

In \cite{shareshian_chromatic_2016}, Shareshian and Wachs introduced a $t$-analogue of the chromatic symmetric function. This resulted in a quasisymmetric function that was useful to address other conjectures related to the chromatic symmetric function. In this paper, we introduce a new $t$-analogue of the chromatic symmetric function
\[Z_G(\mathbf{x};t)=\sum_{\kappa\in \PC_G}t^{\ell(\kappa)}\mathbf{x}^\kappa,\]
where $\ell(\kappa)$ is the number of distinct colors used by $\kappa$. Our $t$-analogue is symmetric and we believe it is better suited to help with the Tree Isomorphism Conjecture. 

We focus on principal specializations of this function, 
\[
\ps_n(Z_G(\mathbf{x}; t)):= Z_G(1,q,\ldots,q^{n-1},0,\ldots;t),
\]
which is a bivariate polynomial. The coefficient of $q^i t^j$ in this polynomial counts proper $n$-colorings of $G$ using colors $\{0, 1, 2, \ldots, n-1\}$ having a chromatic sum of $i$ that use exactly $j$ distinct colors. 
Thus, unlike the principal specialization of $X_G$, the specialization of $Z_G$ retains information about the number of colors used by a coloring. In particular, although $Z_G$ and $X_G$ are equivalent as graph invariants, their principal specializations need not be equivalent. Indeed, there are graphs $G$ and $H$ for which
$\ps_n(X_G)=\ps_n(X_H)$
but
$\ps_n(Z_G)\neq\ps_n(Z_H)$ (see Examples \ref{ex:345 354} and \ref{ex:same-PSX--different-PSZ}).

The additional information contained in $\ps_n(Z_T(\mathbf{x};t))$ is particularly useful for recovering structural properties of trees. In Figure \ref{intro}, we show how our new function and its specializations relate to objects in the literature. 
\begin{figure}[htbp]
    \begin{center}
        \begin{tikzpicture}[xscale=3,yscale=2,scale=.75]
        \tikzset{every node/.style={draw}}
            \node (Xt) at (0,0) {$X_G(\mathbf{x})$};
            \node (psnXt) at (1,-2+1) {$\ps_{n}(X_G(\mathbf{x}))$};
            \node (psXt) at (-1,-2+1) {$\ps(X_G(\mathbf{x}))$};
            \draw[->] (Xt)-- (psXt);
            \draw[->] (Xt)-- (psnXt);
            \draw[dashed,->,transform canvas={yshift=3pt}] (psXt)--(psnXt);
            \draw[dashed,->,transform canvas={yshift=-3pt}] (psnXt)--(psXt);
            \node (Yt) at (0,1) {$Y_G(\mathbf{y})$};
            \node (chiT) at (0,-3+1) {$\chi_G(k)$};
            \draw[->] (Yt) -- (Xt) ;
            \draw[->] (Xt) -- (chiT) ;
            \node (Zt) at (2.5,0) {\tcr{$Z_G(\mathbf{x};t)$}};
            \node (psnZt) at (2.5,-2+1) {\tcr{$\ps_n(Z_G(\mathbf{x};t))$}};
            \draw[->,transform canvas={yshift=3pt}] (Xt)-- node[draw=none,above] {$\mb{x}^{\alpha}\mapsto t^{\ell(\alpha)}\mb{x}^{\alpha}$} (Zt);
             \draw[->,transform canvas={yshift=-3pt}]  (Zt)--node[draw=none,below] {$1\mapsfrom t$} (Xt);
            \draw[->] (Zt)--(psnZt);
            \draw[->,transform canvas={yshift=3pt}] (psnZt)--node[above,draw=none] {$1\mapsfrom t$}(psnXt);
            \draw[->,transform canvas={yshift=-3pt},dashed] (psnXt)-- (psnZt);
        \end{tikzpicture}
    \end{center}
\caption{A diagram situating various invariants mentioned in this paper. The solid arrows indicate that the target can be obtained from the source (usually via specialization). The dashed arrows indicate that there is a pair of graphs with the same source invariant, but different target invariants; see Examples \ref{ex:345 354} and \ref{ex:same-PSX--different-PSZ}}
            \label{intro}
\end{figure}

Our primary focus is on caterpillars. A \textit{caterpillar} is a tree obtained by attaching leaves to the vertices of a path, called the \textit{spine} of the caterpillar. A caterpillar can naturally be represented by a composition 
$\alpha=(\alpha_1,\ldots,\alpha_\ell)$,
where $\ell$ is the number of vertices in the spine, $\alpha_1, \alpha_\ell\geq 2$, and $\alpha_i-1$ is the number of leaves attached to the $i$th spine vertex. Two such caterpillars are isomorphic precisely when their corresponding compositions are equal up to reversal. A \textit{bistar} is a caterpillar with $\ell =2$ and an \textit{elongated bistar} is a caterpillar with corresponding composition $(a,1^r,b)$ such that $a\geq b\geq 2$ and $\ell \geq 0$.

Our first results describe the information about caterpillars that is already visible in the second principal specialization. In \Cref{prop:basic-properties-psn(Z_G)}, we show that for a connected bipartite graph, the coefficient of $t^2$ in $\ps_n(Z_G(\mathbf{x};t))$ records the unordered pair of bipartition sizes. Consequently, $\ps_2(Z_G(\mathbf{x};t))$ allows us to recover the sizes of the two bipartition classes of a caterpillar. We use this observation to characterize the $\ps_2$-classes of caterpillars in \Cref{prop:ps2-classes-caterpillars}. In particular, we describe all caterpillars having a prescribed second principal specialization and determine the elongated bistars (\Cref{cor:ps2-class-el-bistars}) and three-spine-vertex caterpillars (\Cref{cor:ps2-class-3cat}) that occur in these classes.

The second principal specialization does not, however, distinguish caterpillars in general. This leads us to consider the third principal specialization. Our main results show that $\ps_3(Z_G(\mathbf{x};t))$ contains enough information to reconstruct two infinite families of caterpillars. We give two explicit deterministic algorithms, \Cref{alg:elongated-bistars} and \Cref{alg:recover-abc}, that recover elongated bistars
$(a,1^r,b)$ and caterpillars with three spine vertices indexed by the composition $(a,b,c)$, respectively, from $\ps_3(Z_G(\mathbf{x};t))$.
These results suggest that principal specializations of the $t$-chromatic symmetric function $Z_G(\mathbf{x};t)$ at relatively small $n$ can capture substantial structural information about trees.
SageMath~\cite{sagemath} implementations of our three reconstruction Algorithms are available in the accompanying repository.\footnote{\url{https://github.com/jianping-pan/t-chromatic-caterpillars.git}}

The third principal specialization of $Z_G$ does not distinguish all caterpillars; for example, the pair of compositions (2,1,3,3) and (3,1,2,3) index caterpillars with the same third principal specialization of $Z_G$. Nevertheless, we have found no pair of nonisomorphic trees with the same fourth principal specialization among trees with at most 25 vertices. This motivates our conjecture that
$\ps_4(Z_T(\mathbf{x};t))$
distinguishes all trees. 

The remainder of the paper is organized as follows. In \Cref{sec:background}, we review the necessary background on compositions, graphs, colorings, symmetric functions, and the chromatic symmetric function. In \Cref{sec:t-chromatic-symmetric-function}, we introduce the $t$-chromatic symmetric function $Z_G(\mathbf{x};t)$ and its principal specializations, establish their basic properties, and characterize the second principal specialization classes of caterpillars. In \Cref{sec:bistars}, we give an explicit reconstruction algorithm for elongated bistars from their third principal specializations. In \Cref{sec:3-cat}, we prove the corresponding reconstruction theorem for caterpillars with three spine vertices. Finally, in \Cref{sec:future-directions}, we discuss limitations of the third principal specialization and formulate directions for future work.

\section{Background and preliminaries}\label{sec:background}

We state some basic notation used throughout the work, much of which is consistent with \cite{stanleyvol1enumerative,stanley1999enumerative,West96}. For a positive integer $n$, let $[n]:=\{1,\dots,n\}$ and $[0] := \varnothing$.

\subsection{Compositions and partitions}
Let $n$ be a positive integer. A \textbf{composition} of size $n$ and length $\ell$ is a sequence $\alpha=(\alpha_1,\dots,\alpha_\ell)$ of positive integers such that $\alpha_1+\cdots+\alpha_\ell=n$. We write $\alpha\vDash n$ if $\alpha$ is a composition of size $n$ (of any length) and we use $\ell(\alpha)$ to denote the length of $\alpha$. A \textbf{partition} of length $\ell$ is a composition $\lambda=(\lambda_1,\dots,\lambda_\ell)$ such that $\lambda_1\geq \lambda_2\geq \cdots \geq \lambda_\ell.$ We write $\lambda\vdash n$ if $\lambda$ is a partition of size $n$ (of any length). We sometimes use {multiplicity notation} for repeated parts, meaning we write $(\ldots, a^k, \ldots) $ in place of $(\ldots, \underbrace{a, a, \ldots, a}_{k \text{ times}}, \ldots)$. For example, the compositions $(5, 1^3, 2)$ and $ (5, 1, 1, 1, 2)$ are the same.

\subsection{Graphs and colorings}
A \textbf{graph} $G$ is a pair $(V,E)$ consisting of a finite set $V=V(G)$, the \textbf{vertex set}, and a finite set $E=E(G)$, the \textbf{edge set}, consisting of 2-element subsets of $V$. All graphs in this paper are \textit{simple}, meaning they contain no loops or multiple edges. We use standard language in graph theory; see for example~\cite{West96}. We will be working with a class of trees called caterpillars, formed by attaching leaves to a path graph. 
Formally, we define caterpillars as follows.  
\begin{definition}
    Let $\alpha=(\alpha_1,\dots,\alpha_\ell)$ be a composition of $d\geq 2$ and length $\ell$ such that $\alpha_1,\alpha_\ell\geq 2$, or $\alpha = (1)$. The \textbf{caterpillar of type $\alpha$}, denoted $C_\alpha=(V,E)$, is the graph with 
    \[V=\{v_1,\dots,v_\ell\}\cup \{w_{i,j}:i\in [\ell],\ j\in [\alpha_i-1]\},\]
    \[E=\{\{v_i,v_{i+1}\}:i\in [\ell-1]\}\cup \{\{v_i,w_{i,j}\}:i\in [\ell],\ j\in [\alpha_i-1]\}.\]
    See Figure~\ref{fig:caterpillar-564} for an example. The set of vertices $\{v_1,\dots,v_\ell\}$ is called the \textbf{spine} of the caterpillar. 
\end{definition}

Reversing the composition induces  an isomorphism of two caterpillars.
\begin{lemma}\label{lem:caterpillars-isomorphic}
    Let $\alpha=(\alpha_1,\dots,\alpha_\ell)$ and $\beta=(\beta_1,\dots,\beta_{\ell'})$ be compositions of sizes $d,d'\geq 4$ and lengths $\ell,\ell'\geq 2$, with $\alpha_1, \alpha_\ell, \beta_1, \beta_{\ell'} \geq 2$. Then the caterpillars $C_\alpha$ and $C_{\beta}$ are isomorphic if and only if $\beta=\alpha$ or $\beta=(\alpha_\ell,\dots,\alpha_1)$.
\end{lemma}

\begin{figure}[htbp]
    \centering
      \begin{tikzpicture}
            \tikzset{every node/.style={draw,circle,fill,inner sep=2pt}}
            \node (a1) at (0,0) {};
            \node (a2) at (.5,0) {};
            \node (a3) at (1,0) {};
            \node (a4) at (1.5,0) {};
            \node (a5) at (2,0) {};
            \node (a) at (1,1) {};

            \node (b1) at (2.5,0) {};
            \node (b2) at (3,0) {};
            \node (b3) at (3.5,0) {};
            \node (b4) at (4,0) {};
            \node (b5) at (4.5,0) {};
            \node (b6) at (5,0) {};
            \node (b) at (4,1) {};

            \node (cb) at (6,1) {};

            \node (c1) at (5.5+1.5,0) {};
            \node (c2) at (6+1.5,0) {};
            \node (c3) at (6.5+1.5,0) {};
            \node (c4) at (7+1.5,0) {};
            \node (c) at (6.5+1.5,1) {};

            \draw (a1)--(a);
            \draw (a2)--(a);
            \draw (a3)--(a);
            \draw (a4)--(a);
            \draw (a5)--(a);

            \draw (b1)--(b);
            \draw (b2)--(b);
            \draw (b3)--(b);
            \draw (b4)--(b);
            \draw (b5)--(b);
            \draw (b6)--(b);

            \draw (c1)--(c);
            \draw (c2)--(c);
            \draw (c3)--(c);
            \draw (c4)--(c);

            \draw (a)--(b)--(cb)--(c);
        \end{tikzpicture}
    \caption{The caterpillar $C_{(6,7,1,5)}$}
    \label{fig:caterpillar-564}
\end{figure}
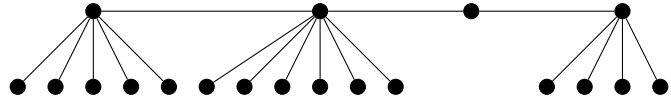
We will sometimes denote the caterpillars $C_\alpha$ and $C_\beta$ simply as $\alpha$ and $\beta$ if the context is clear. We consider a special family of caterpillars: elongated bistars.
\begin{definition}
     A caterpillar of type $\alpha$ is a \textbf{bistar} if $\alpha$ has length $2.$ More generally, a caterpillar of type $\alpha$ is an \textbf{elongated bistar} if $\alpha$ is of the form $\alpha = (a, 1^r, b)$ for $a \geq b \geq 2$ and $r \geq 0.$ See \Cref{fig:elongated-bistar}.
\end{definition}

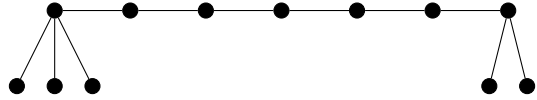
\begin{figure}[h]
\centering
      \begin{tikzpicture}
            \tikzset{every node/.style={draw,circle,fill,inner sep=2pt}}
            \node (a1) at (0,0) {};
            \node (a2) at (.5,0) {};
            \node (a3) at (1,0) {};
            \node (a) at (.5,1) {};

            \node (b) at (1.5,1) {};

            \node (c) at (2.5, 1) {};

            \node (d) at (3.5, 1) {};

            \node (e) at (4.5, 1) {};

            \node (f) at (5.5, 1) {};

           \node (g) at (6.5, 1) {};
           \node (g1) at (6.25, 0) {};
           \node (g2) at (6.75,0) {};

            \draw (a1)--(a);
            \draw (a2)--(a);
            \draw (a3)--(a);
            
            \draw (g1)--(g);
            \draw (g2)--(g);

            \draw (a)--(b)--(c)--(d)--(e)--(f)--(g);
        \end{tikzpicture}

\caption{The elongated bistar of type $(4, 1^5, 3).$}\label{fig:elongated-bistar}
\end{figure}

We are primarily interested in graph colorings.

\begin{definition}
    A \textbf{coloring} of a graph $G=(V,E)$ is a function $\kappa:V\to \Z_{>0}$. For $v\in V$, we call $\kappa(v)$ the \textbf{color} of vertex $v$. The \textbf{type} of a coloring is the infinite sequence $\alpha=(|\kappa^{-1}(1)|,|\kappa^{-1}(2)|,\dots)$. A coloring $\kappa$ is \textbf{proper} provided $\kappa(u)\neq \kappa(v)$ whenever $\{u,v\}\in E$. The \textbf{length} of a coloring $\kappa$, denoted $\ell(\kappa)$, is the number of distinct colors used, i.e. the size of its image. The set of proper colorings of $G$ is denoted $\PC_G.$ For a positive integer $n$, we say that $\kappa\in \PC_G$ is a \textbf{proper $n$-coloring} provided $\kappa(V) \subseteq [n]$. The set of proper $n$-colorings is denoted $ \PC_G(n)$. The \textbf{chromatic polynomial} $\chi_G:\Z_{>0}\to \Z_{\geq 0}$ is defined by 
    \[\chi_G(n)=|\PC_G(n)|.\]
\end{definition}
\begin{example}
    For any tree $T$ on $d$ vertices, $\chi_T(n)=n(n-1)^{d-1}$ for all $n\geq 1$.
\end{example}
\subsection{Symmetric functions and chromatic symmetric functions}
Let $\mathbf{x}=(x_1,x_2,\dots)$ be an infinite sequence of commuting variables. The space of symmetric functions of homogeneous degree $d$ is denoted $\Lambda_d(\mathbf{x})$ or just $\Lambda_d.$ For background on symmetric functions, see~\cite[Chapter~7]{stanley1999enumerative}.

\begin{definition}
The \textbf{monomial symmetric function} corresponding to a partition $\lambda$ is
\[
m_\lambda=\sum_{\alpha} \mathbf{x}^{\alpha}
\]
where the sum is over all distinct rearrangements $\alpha = (\alpha_1,\alpha_2, \dots)$ of $\lambda$ (possibly with internal and trailing zeros) and $\mathbf{x}^{\alpha}=\prod_{i \geq 1}x_{i}^{\alpha_i}$.
\end{definition}

It is well-known that $\{m_\lambda:\lambda\vdash d\}$ forms a basis of the vector space $\Lambda_d$.

\begin{definition}[{\cite{stanley_1995}}]
    The \textbf{chromatic symmetric function} of a graph $G=(V,E)$ is 
\[X_G(\mathbf{x})=\sum_{\kappa\in \PC_G} \mathbf{x}^{\kappa},\]
where $\mathbf{x}^{\kappa}=\prod_{v\in V}x_{\kappa(v)}.$
\end{definition}
This is a generalization of the chromatic polynomial in the sense that for any positive integer $n$,
\[X_G(\underbrace{1,1,\dots,1}_{n},0,0,\dots)=\chi_G(n).\]
Despite two trees of the same number of vertices always having the same chromatic polynomial, Stanley was unable to find two trees with the same chromatic symmetric function.

\begin{conj}[Stanley's Tree Isomorphism Conjecture\footnote{This is not phrased as a conjecture in either of \cite{stanley_1995} or \cite{stanley1999enumerative}, but it is convenient to refer to it as such here.} 
{\cite{stanley_1995}, \cite[Exercise~7.47(b)]{stanley1999enumerative}}]\label{question:tree-isom}
If $T$ and $T'$ are non-isomorphic trees, then $X_T(\mathbf{x})\neq X_{T'}(\mathbf{x})$.
\end{conj}

Recall that values of the chromatic polynomial $\chi_G(n)$ are obtained by setting the first $n$ variables of $X_G(\mathbf{x})$ to 1 and the rest to 0. While this specialization weakens distinguishing power, specialization of graph invariants often allows certain properties to be more easily read. This idea leads to the following paradigm: \textit{Specialize variables in the chromatic symmetric function to simplify as much as possible while still retaining the (conjectured) distinguishing property.} A symmetric function has a natural specialization.

\begin{definition}\label{def:nprincipal}
    Let $n$ be a positive integer and $f\in \Lambda_d(\mathbf{x})$. The \textbf{$n$th principal specialization of $f$} is
    \[\ps_n(f):=f(1,q,q^2,\dots,q^{n-1},0,0,\dots),\]
which is obtained by setting the variable $x_i = q^{i-1}$, for $1\leq i\leq n$ and $x_j = 0$, for all $j>n$.    
\end{definition}
From Definition \ref{def:nprincipal}, we can see that $\ps_n(f)$ is a polynomial in $\C[q]$. This polynomial has the interpretation that for a graph $G=(V,E)$, if
\begin{equation}\label{eq:psnXG-interpretation}\ps_n(X_G(\mathbf{x}))=\sum_{\substack{\kappa:V\to \{0,1,\dots,n-1\}\\\text{$\kappa$ proper}}}q^{\sum_{v\in V}\kappa(v)}=\sum_{i\geq 0}c_iq^i,\end{equation}
then $c_i$ is the number of proper colorings $\kappa: V\to \{0,1,\dots,n-1\}$ with $\sum_{v\in V}\kappa(v)=i$. Notice that under this interpretation of $\ps_n(X_G)$, we now include 0 as a color for notational convenience. In this setting, we call the quantity
\[\sum_{v\in V}\kappa(v)\]
the \textbf{weight} or  \textbf{chromatic sum} of the coloring $\kappa.$ This specialization of the chromatic symmetric function has been considered by Loehr and Warrington \cite{LoehrWarrington2024}  in the context of Stanley's Tree Isomorphism Conjecture.

\begin{conj}[{\cite[Conjecture 21]{LoehrWarrington2024}}]\label{conj:LW}
    For all $n\geq 1$, if $T$ and $T'$ are trees with at most $n$ vertices, then $T$ and $T'$ are isomorphic if and only if $\ps_n(X_T(\mathbf{x}))=\ps_n(X_{T'}(\mathbf{x})).$ In other words, $\ps_n(X_T(\mathbf{x}))$ distinguishes trees with at most $n$ vertices. 
\end{conj}
Loehr and Warrington confirmed this conjecture on all trees with at most $17$ vertices. If true, this would imply a positive answer to Stanley's conjecture. The polynomial $\ps_n(X_T(\mathbf{x}))$ has also been considered by Loebl \cite{Loebl07}.

Bajo, Beck, and Vindas-Mel\'{e}ndez \cite{BBV} define the \textbf{$q$-chromatic polynomial}\footnote{In the notation of \cite{BBV}, what we present here is $\chi^{\mathbf{1}}_{G}(q,y)$. They define more generally $\chi^{\lambda}_{G}(q,y)$ where $\lambda$ weights the vertices of $G$. This satisfies a similar specialization correspondence with the weighted chromatic symmetric function of Crew--Spirkl \cite{CrewSpirkl20}. } $\tilde{\chi}^q_G(y)$ to be the unique polynomial in $y$ with coefficients as rational functions in $q$ such that 
   \[\tilde{\chi}^q_G(1+q+\cdots+q^{n-1})=\ps_n(X_G(\mathbf{x}))\]
   for all $n\geq 1$.
   It turns out that the leading coefficient of $\tilde{\chi}^q_G(y)$ is equivalent information
   to evaluating $\tilde{\chi}_G^q(y)$ at $y=(1-q)^{-1}.$ Further, the leading coefficient of $\tilde{\chi}_G^q(y)$ can be viewed as a specialization of the chromatic symmetric function. The  \textbf{stable principal specialization} of a symmetric function $f$ is the limit of the $n$th principal specializations
    \[\ps(f)=\lim_{n\to \infty}\ps_n(f)=f(1,q,q^2,\dots)\]
   in the sense that the coefficient of $q^{i}$ in the sequence $\ps_1(f),\ps_2(f),\dots$ eventually becomes a constant for each $i\geq 0$. Then, up to a scaling factor, the leading coefficient of $\tilde{\chi}_G^q(y)$ is equivalent information to $\ps(X_G(\mathbf{x}))$ \cite[Corollary 18]{BBV}. 
   Bajo, Beck, and Vindas-Mel\'{e}ndez conjecture that the leading coefficient of $\tilde{\chi}_G^q(y)$ distinguishes trees.
   \begin{conj}\cite[Conjecture 1, Corollary 18]{BBV}\label{conj:BBV}
       $\ps(X_T(\mathbf{x}))$ distinguishes trees.
   \end{conj} This conjecture has also been confirmed on all trees up to $17$ vertices. However, \Cref{conj:LW} and \Cref{conj:BBV} seem independent from one other \textit{a priori}. To see this, consider the following:
   \begin{enumerate}
       \item The symmetric function defined in the \textit{Schur basis} (see \cite[Chapter 7]{stanley1999enumerative} for definitions)  \begin{align*}
           f&=s_{2,1,1,1,1,1,1} - s_{2,2,1,1,1,1} + s_{2,2,2,1,1} - 2 s_{2,2,2,2} + 2 s_{3,2,1,1,1}  - 3 s_{3,2,2,1}\\
           &\quad  + 3 s_{3,3,1,1}- s_{3,3,2} - 2 s_{4,1,1,1,1} + s_{4,2,1,1} - 4 s_{4,2,2} + 4 s_{4,3,1} + 2 s_{4,4} \\&\quad - 2 s_{5,2,1} - s_{5,3} + s_{6,1,1}
       \end{align*}
 of degree 8 has $\ps_8(f)=0$ but $\ps(f)\neq 0$.

\item The symmetric function \[g=s_{6,3,2,2,2}-s_{5,5,2,1,1,1}\] of degree 15 has $\ps(g)=0$ but $\ps_{15}(g)\neq 0$.
   \end{enumerate}
These examples (computed in SageMath~\cite{sagemath}) show that, in the space of symmetric functions, neither finite principal specialization nor stable principal specialization determines the other.
In particular, this means that $\ps_n(X_G)$ could potentially distinguish two graphs that $\ps(X_G)$ could not tell apart, and vice versa.

Implicit in \cite{BBV}, there is a conjecture that is weaker than both \Cref{conj:LW} and \Cref{conj:BBV}, but still stronger than \Cref{question:tree-isom}.
\begin{conj}[{\cite{BBV}}]\label{conj:q-chrom-polynomial-distinguishes}
    The $q$-chromatic polynomial distinguishes trees. Equivalently, the set of values $\{\ps_n(X_T(\mathbf{x})):n\geq 1\}$ distinguishes trees. 
\end{conj}
Note that the computational verifications of Conjectures 
\ref{conj:LW} and \ref{conj:BBV} imply that \Cref{conj:q-chrom-polynomial-distinguishes} holds for trees with at most $17$ vertices.

\section{The $t$-chromatic symmetric function}\label{sec:t-chromatic-symmetric-function}
In this section, we introduce a new $t$-analog of the chromatic symmetric function that tracks the number of distinct colors used. We give some basic properties and compare this new function and its specialization to the classical chromatic symmetric function.

\subsection{Definition and basic properties}
Recall that for a proper coloring $\kappa:V(G)\to \Z_{>0}$, the quantity $\ell(\kappa)$ is the number of distinct colors used by $\kappa.$
\begin{definition}
For a graph $G$ with $d$ vertices, the\textbf{ $t$-chromatic symmetric function} is
\[Z_G(\mathbf{x};t)=\sum_{\kappa\in \PC_G}t^{\ell(\kappa)}\mathbf{x}^\kappa.\]
\end{definition}
Observe that $Z_G(\mathbf{x};1)=X_G(\mathbf{x})$. The $t$-chromatic symmetric function naturally lives in the ring $\C[t][[\mathbf{x}]]$ of formal power series in $\mathbf{x}$ with coefficients in $\C[t]$. Let $\Lambda(\mathbf{x};t)$ be the corresponding ring of symmetric functions and let $\Lambda_d(\mathbf{x};t)$ be the $\C$-vector space of symmetric functions in $\mathbf{x}$ of homogeneous degree $d$ with coefficients in $\C[t]$  (with the usual addition and scalar multiplication). Note that $\Lambda_d(\mathbf{x};t)$ is an infinite dimensional vector space with basis \[\{t^im_\lambda:i\in \Z_{\geq 0},\ \lambda\vdash d\}.\]

\begin{example}\label{ex:Z-of-st5}
    Let $G$ be the path graph on four vertices, or equivalently, the caterpillar $C_{(2,2)}$. This has $t$-chromatic symmetric function  $Z_G(\mathbf{x}; t) = 2t^2m_{2,2}  + 6t^3m_{2, 1, 1} + 24t^4m_{1, 1, 1, 1}.$ Indeed, there are five types of colorings of $G$ up to symmetry; see Figure~\ref{fig:coloring4path}. The type in the top row contributes $2t^2m_{2, 2}$ since we can swap the red (1) and blue (2) colors. Similarly, each of the three coloring-types in the second row contributes $2t^3m_{2, 1, 1}$ and the coloring-type in the third row contributes $24t^4m_{1, 1, 1, 1}$.
\end{example}

\definecolor{cbred}{RGB}{187, 85, 102}   
\definecolor{cbblue}{RGB}{0, 68, 136}  
\definecolor{cbyellow}{RGB}{221, 170, 51}  
\definecolor{cbsky}{RGB}{136, 204, 238}
\begin{figure}[h]
\centering
      \begin{tikzpicture}[yscale=1]

             \node at (-9,1.25) {$\lambda = (2, 2)$};
             \node at (-9,0) {$\lambda =(2, 1, 1)$};
             \node at (-9,-1.25) {$\lambda =(1, 1, 1, 1)$};

             \tikzset{every node/.style={draw,circle,fill,inner sep=2pt}}
             
            \node[color=cbred, label=above:1] (a) at (-1.5,0) {};
            \node (b)[color=cbblue, label=above:2] at (-.5,0) {};
            \node (c)[color=cbsky, label=above:3] at (.5,0) {};
            \node (d)[color=cbred, label=above:1] at (1.5,0) {};

            \node (a')[color=cbred, label=above:1] at (-6.5,0) {};
            \node (b')[color=cbblue, label=above:2] at (-5.5,0) {};
            \node (c')[color=cbred, label=above:1] at (-4.5,0) {};
            \node (d')[color=cbsky, label=above:3] at (-3.5,0) {};

            \node (a'')[color=cbred, label=above:1] at (6.5,0) {};
            \node (b'')[color=cbblue, label=above:2] at (5.5,0) {};
            \node (c'')[color=cbred, label=above:1] at (4.5,0) {};
            \node (d'')[color=cbsky, label=above:3] at (3.5,0) {};

            \node (a1)[color=cbred, label=above:1] at (-1.5,1.25) {};
            \node (b1)[color=cbblue, label=above:2] at (-.5,1.25) {};
            \node (c1)[color=cbred, label=above:1] at (.5,1.25) {};
            \node (d1)[color=cbblue, label=above:2] at (1.5,1.25) {};

            \node (a3)[color=cbred, label=above:1] at (-1.5,-1.25) {};
            \node (b3)[color=cbblue, label=above:2] at (-.5,-1.25) {};
            \node (c3)[color=cbsky, label=above:3] at (.5,-1.25) {};
            \node (d3)[color=cbyellow, label=above:4] at (1.5,-1.25) {};
           
            \draw (a)--(b)--(c)--(d);
            \draw (a')--(b')--(c')--(d');
            \draw (a'')--(b'')--(c'')--(d'');
            \draw (a1)--(b1)--(c1)--(d1);
            \draw (a3)--(b3)--(c3)--(d3);

        \end{tikzpicture}

\caption{The five types of colorings of $C_{(2, 2)}$.}
\label{fig:coloring4path}
\end{figure}
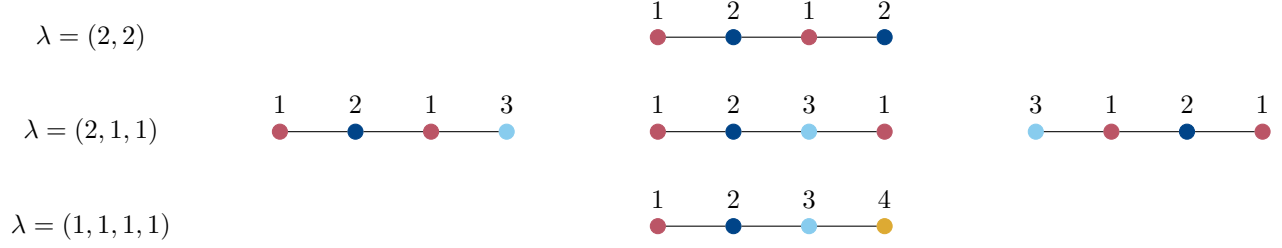\
Our $t$-analog shares a number of properties with the classical chromatic symmetric function.

\begin{proposition}\label{prop:basic-properties-of-zg}
    Let $G$ be a graph with $d$ vertices. 
    \begin{enumerate}
        \item $Z_G(\mathbf{x};t)\in \Lambda_d(\mathbf{x};t)$.
        \item If $X_G(\mathbf{x}) = \sum_{\lambda \vdash d}c_\lambda m_\lambda$, then $Z_G(\mathbf{x}; t) = \sum_{\lambda \vdash d}t^{\ell(\lambda)}c_\lambda m_\lambda$.
        \item For all $n$, we have \[Z_G(\underbrace{1,\dots,1}_{n},0,0,\dots; 1)=\chi_G(n).\]
    \end{enumerate}
\end{proposition}
\begin{proof}
    For Part (1), observe that swapping two colors does not change the total number of distinct colors used. Parts (2) and (3) follow immediately from the definitions.
\end{proof}
\begin{remark}\label{remark:bad-multiplication-for-ZG}
    For graphs $G$ and $H$, let $G\sqcup H$ denote their disjoint union. An important property of the chromatic symmetric function is that it is multiplicative across disjoint unions: $X_{G\sqcup H}=X_GX_H$ \cite[Proposition 2.3]{stanley_1995}. This is not in general true for $Z_G(\mathbf{x};t).$ For example, if $G$ and $H$ are each a  single vertex, then $Z_G(\mathbf{x};t)=Z_H(\mathbf{x};t)=tm_1$ and so $Z_G(\mathbf{x};t)Z_H(\mathbf{x};t)=(tm_1)^2=t^2(m_{2}+2m_{1,1})$. On the other hand,
    \[Z_{G\sqcup H}(\mathbf{x};t)=tm_2 + 2t^2m_{1,1}\neq t^2(m_{2}+2m_{1,1})=Z_G(\mathbf{x};t)Z_H(\mathbf{x};t).\]
    Hence, additional care must be taken when computing the $t$-chromatic symmetric function for disjoint graphs.
\end{remark}

Define a linear map $L_d:\Lambda_d(\mathbf{x})\to \Lambda_d(\mathbf{x};t)$ by setting $L_d(m_\lambda)=t^{\ell(\lambda)}m_\lambda$ and extending by linearity. Then, $L_d$ is injective and satisfies $L_d(X_G(\mathbf{x}))=Z_G(\mathbf{x};t)$ by Proposition~\ref{prop:basic-properties-of-zg}.  Let $\mathcal{G}_d$ be the set of graphs of $d$ vertices. Since $L_d$ is an injective linear map, an equation of the form
\[\sum_{G\in \mathcal{G}_d}c_G X_G(\mathbf{x})=0\]
for some constants $c_G\in \C$ holds if and only if\[\sum_{G\in \mathcal{G}_d}c_G Z_G(\mathbf{x};t)=0\] also holds.
This implies that $Z_G(\mathbf{x};t)$ obeys the following recurrence relations.\footnote{In light of \Cref{remark:bad-multiplication-for-ZG}, extra care must be used in applying these recurrence relations to compute $Z_G(\mathbf{x};t).$ In particular, one must compute directly the $t$-chromatic symmetric function for forests of paths and forests of stars, respectively, for applying triple deletion and deletion-near-contraction, respectively. }
\begin{proposition}[Triple deletion \cite{ORELLANA20141}/Modular Law \cite{guay-paquet_modular_2013}]
    Let $e_1,e_2,e_3$ form a triangle in $G$. Then
        \[Z_G(\mathbf{x};t)=Z_{G-e_1}(\mathbf{x};t)+Z_{G-e_2}(\mathbf{x};t)-Z_{G-\{e_1,e_2\}}(\mathbf{x};t),\]
        where for a subset of edges $S$, $G-S$ denotes the deletion of those edges from $G$.
\end{proposition}

For an edge $e$ of $G$, let $G\odot e$ be $G/e$ (contraction) with an additional leaf added to the contracted vertex of $e$. Let $(G\odot e)-\ell_e$ be $G\odot e$ with the new leaf-edge removed.
\begin{proposition}[Deletion-near-contraction \cite{AMOZ23}]
    Let $e$ be an edge of $G$. Then, we have 
        \[Z_G(\mathbf{x};t)=Z_{(G-e)}(\mathbf{x};t)-Z_{(G\odot e)-\ell_e}(\mathbf{x};t)+Z_{G\odot e}(\mathbf{x};t).\]
        \end{proposition}

Moreover, the injectivity of $L_d$ implies that $Z_G(\mathbf{x};t)$ has equal strength as $X_G(\mathbf{x})$ in distinguishing graphs.
\begin{proposition}
    Let $G$ and $H$ be graphs. Then $Z_G(\mathbf{x};t)=Z_{H}(\mathbf{x};t)$ if and only if $X_G(\mathbf{x})=X_{H}(\mathbf{x})$.
\end{proposition}

In particular, one proves Stanley's Tree Isomorphism Conjecture if and only if one proves the corresponding conjecture for the $t$-chromatic symmetric function.

\subsection{Specialization}
In this paper, we are interested in the principal specialization of $Z_G(\mathbf{x};t)$ which we define analogously to principal specialization of symmetric functions. 
\begin{definition}
   Let $n$ be a positive integer. The \textbf{$n$th-principal specialization} of $f=f(\mathbf{x};t)\in \Lambda_d(\mathbf{x};t)$ is defined to be
    \[\ps_n(f(\mathbf{x};t))=f(1,q,\dots,q^{n-1},0,0,\dots; t).\]
\end{definition}
This is a polynomial in $\C[q,t]$. This polynomial has the interpretation that for a graph $G=(V,E)$, if
\begin{equation}\label{eq:psnZG-interpretation}
\ps_n(Z_G(\mathbf{x};t))=\sum_{\substack{\kappa:V\to \{0,1,\dots,n-1\}\\\text{$\kappa$ proper}}}q^{\sum_{v\in V}\kappa(v)}t^{\ell(\kappa)}=\sum_{\substack{0\leq i,\\ 1 \leq j \leq n}}c_{i,j}q^it^j,
\end{equation}
then $c_{i,j}$ is the number of proper colorings $\kappa: V\to \{0,1,\dots,n-1\}$ with weight $\sum_{v\in V}\kappa(v)=i$ and length $\ell(\kappa)=j$. Since $Z_G(\mathbf{x};1)=X_G(\mathbf{x}),$ we therefore have the following result.

\begin{proposition}\label{prop:psnzg-equal-implies-psnxg-equal}
    Let $G$ and $H$ be graphs and $n$ a positive integer. If $\ps_n(Z_G(\mathbf{x};t))=\ps_n(Z_H(\mathbf{x};t))$, then $\ps_n(X_G(\mathbf{x}))=\ps_n(X_H(\mathbf{x}))$. 
\end{proposition}

While $X_G(\mathbf{x})$ and $Z_G(\mathbf{x};t)$ have equal strength in distinguishing graphs, their principal specializations do not. We show that it is strictly harder for two graphs to have the same principal specialization of $Z_G$ than it is for two graphs to have the same principal specialization of $X_G.$

\begin{example}\label{ex:345 354}
The smallest $n$ for which two graphs $G,H$ have $\ps_n(X_G(\mathbf{x}))=\ps_n(X_H(\mathbf{x}))$ but $\ps_n(Z_G(\mathbf{x};t))\neq\ps_n(Z_H(\mathbf{x};t))$ is $n=3$. As one relevant example, consider the caterpillars $C_{(5,4,3)}$ and $C_{(4,5,3)}$. Direct computation (found with SageMath \cite{sagemath}) shows that 
$\ps_3(X_{C_{(5,4,3)}}(\mathbf{x}))=\ps_3(X_{C_{(4,5,3)}}(\mathbf{x}))$ and their common value is 
\begin{align*}&q^{ 20} + 7 q^{19} + 31 q^{18} + 94 q^{17} + 224 q^{16} + 444 q^{15}+729 q^{14} + 991 q^{13} + 1102 q^{12}\\
 &\quad+ 991 q^{11} + 729 q^{10} + 444 q^{9} + 224 q^{8} + 94 q^{7} + 31 q^{6} + 7 q^{5} + q^{4}\end{align*}
but $\ps_3(Z_{C_{(5,4,3)}}(\mathbf{x};t))\neq \ps_3(Z_{{C_{(4,5,3)}}}(\mathbf{x};t))$ since
\begin{align*}
    \ps_3(Z_{C_{(5,4,3)}}(\mathbf{x};t))&=t^2(q^{19} + q^{17} + q^{14} + q^{10} + q^{7} + q^{5})\\
    &\quad+t^3(q^{20} + 6 q^{19} + 31 q^{18} + 93 q^{17} + 224 q^{16} + 444 q^{15} + 728 q^{14}\\
    &\quad+ 991 q^{13} + 1102 q^{12} + 991 q^{11} + 728 q^{10} + 444 q^{9} + 224 q^{8} + 93 q^{7}\\ &\quad+ 31 q^{6} + 6 q^{5} + q^{4})
\end{align*}
whereas
\begin{align*}
    \ps_3(Z_{C_{(4,5,3)}}(\mathbf{x};t))&=t^2(2 q^{18} + 2 q^{12} + 2 q^{6})\\
    &\quad+t^3( q^{20} + 7 q^{19} + 29 q^{18} + 94 q^{17} + 224 q^{16} + 444 q^{15} + 729 q^{14}\\ 
    &\quad+ 991 q^{13} + 1100 q^{12} + 991 q^{11} + 729 q^{10} + 444 q^{9} + 224 q^{8} + 94 q^{7}\\ &\quad+ 29 q^{6} + 7 q^{5} + q^{4})
\end{align*}
One can see a difference in the $t^2$ coefficients. 
As we will see in \Cref{cor:recover-sizes-bipartition-caterpillar}, this coefficient is telling us the sizes in the (unique) bipartition of the two respective caterpillars.
\end{example}

\begin{remark}[A peculiar family of caterpillars]
    In addition to their use in the previous example, we noticed another, seemingly unusual, property about the caterpillars $C_{(5, 4, 3)}$ and $C_{(4, 5, 3)}$. In particular, by direct computation in SageMath, one can check that even though $\ps_3(X_{C_{(5, 4, 3)}}(\mathbf{x})) = \ps_3(X_{C_{(4, 5, 3)}}(\mathbf{x}))$, we actually have   \[\ps_2(X_{C_{(5, 4, 3)}}(\mathbf{x})) \neq \ps_2(X_{C_{(4, 5, 3)}}(\mathbf{x})). \] This goes against the intuition that ``more colors'' necessarily means ``stronger distinguishing power,''  illustrating the subtleties of working with principal specializations. 
    
    This seems to be a rare property. In fact, we verified computationally that there are only three pairs of compositions $\alpha, \beta$ with $|\alpha| = |\beta| \leq 16$ satisfying $\ps_2(X_{C_{\alpha}}(\mathbf{x})) \neq \ps_2(X_{C_{\beta}}(\mathbf{x}))$ but $\ps_3(X_{C_{\alpha}}(\mathbf{x})) = \ps_3(X_{C_{\beta}}(\mathbf{x}))$: 
\begin{itemize}
    \item $(5, 4, 3)$ and $(4, 5, 3)$,
 \item $(5, 1, 1, 4, 3)$ and $(4, 1, 1, 5, 3)$, and
  \item $(5, 2, 2, 4, 3)$ and $(4, 2, 2, 5, 3)$.
\end{itemize}
Additionally, we verified that for $j \leq 7$, each pair $(5, j, j, 4, 3)$ and $(4, j, j, 5, 3)$ exhibits this phenomenon, raising the following questions. Do $(5, j, j, 4, 3)$ and $(4, j, j, 5, 3)$ exhibit this phenomenon for all $j$? And if so, do they continue to be the only pairs of compositions which do?
\end{remark}

    \begin{example}\label{ex:same-PSX--different-PSZ}
    The smallest example of two graphs $G,H$ having $\ps_n(X_{G}(\mathbf{x})) = \ps_n(X_{H}(\mathbf{x}))$ and $\ps_n(Z_{G}(\mathbf{x};t)) \neq \ps_n(Z_{H}(\mathbf{x};t))$ for some $n$ can be found on $6$ vertices. As one example, let $G$ and $H$ be as in \Cref{fig:same_ps4_xg_different_ps4_zg}. Then, $\ps_4(X_G(\mathbf{x}))=\ps_4(X_H(\mathbf{x}))$ and their common value (as found using SageMath~\cite{sagemath}) is 
    \[q^{14} + 3 q^{13} + 13 q^{12} + 24 q^{11} + 34 q^{10} + 42 q^{9} + 34 q^{8} + 24 q^{7} + 13 q^{6} + 3 q^{5} + q^{4},\]
    but $\ps_4(Z_G(\mathbf{x};t))\neq \ps_4(Z_H(\mathbf{x};t))$ since 

    \begin{align*}
         \ps_4(Z_{G}(\mathbf{x};t))&=t^3(q^{14} + 3 q^{13} + q^{12} + 4 q^{11} + 2 q^{10} + 2 q^{9} + 2 q^{8} + 4 q^{7} + q^{6} + 3 q^{5} + q^{4})\\
         &\quad+t^4(12 q^{12} + 20 q^{11} + 32 q^{10} + 40 q^{9} + 32 q^{8} + 20 q^{7} + 12 q^{6})
    \end{align*}
    whereas
    \begin{align*}
        \ps_4(Z_{H}(\mathbf{x};t))&=t^3(q^{14} + 3 q^{13} + 7 q^{12} + 4 q^{11} + 8 q^{10} + 2 q^{9} + 8 q^{8} + 4 q^{7} + 7 q^{6} + 3 q^{5} + q^{4})\\
        &\quad+t^4(6 q^{12} + 20 q^{11} + 26 q^{10} + 40 q^{9} + 26 q^{8} + 20 q^{7} + 6 q^{6})
    \end{align*}
   
    Note that $q^{12}t^3$ and $q^{12}t^4$ coefficients differ in both $t$-chromatic symmetric function specializations, but add to the same thing in the classical case! 
    We further note for posterity that $G$ and $H$ have non-equal $(t)$-chromatic symmetric functions since \[X_G(\mathbf{x})=720 m_{1,1,1,1,1,1} + 144 m_{2,1,1,1,1} + 20 m_{2,2,1,1} + 12 m_{3,1,1,1} + m_{3,2,1}\]
    whereas
    \[X_H(\mathbf{x})=720 m_{1,1,1,1,1,1} + 120 m_{2,1,1,1,1} + 20 m_{2,2,1,1} + 6 m_{2,2,2} + 6 m_{3,1,1,1} + m_{3,2,1}.\]
    \end{example}
   
    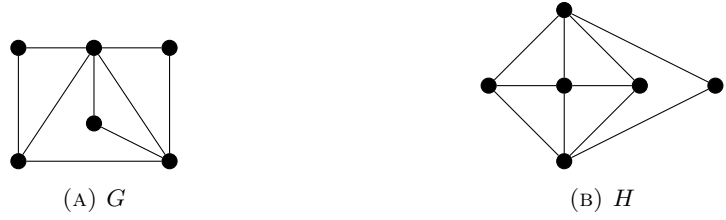
\begin{figure}
        \centering
        \begin{subfigure}[b]{.4\textwidth}
        \centering
         \begin{tikzpicture}
       \tikzset{every node/.style={draw,circle,fill=black,inner sep=2pt}}
       \node (1) at (0,1) {};
       \node (2) at (0,-.5) {};
       \node (3) at (1,1) {};
       \node (4) at (2,-.5) {};
       \node (5) at (1,0) {};
       \node (6) at (2,1) {};
        \draw (3)--(1)--(2)--(4)--(5)--(3)--(6)--(4)--(3);
        \draw (3)--(2);
       \end{tikzpicture}
             \caption{$G$}
        \label{fig:same_ps4_xg_different_ps4_zg_graph1}
        \end{subfigure}
        \begin{subfigure}[b]{.4\textwidth}
        \centering
         \begin{tikzpicture}
         \tikzset{every node/.style={draw,circle,fill=black,inner sep=2pt}}
         \node (1) at (0,0) {};
       \node (2) at (1,1) {};
       \node (3) at (1,0) {};
       \node (4) at (1,-1) {};
       \node (5) at (2,0) {};
       \node (6) at (3,0) {};
        \draw (1)--(2)--(6)--(4)--(1)--(3)--(2)--(5)--(4)--(3)--(5);
       \end{tikzpicture}
             \caption{$H$}
        \label{fig:same_ps4_xg_different_ps4_zg_graph2}
        \end{subfigure}
      
        \caption{$\ps_4(X_G)=\ps_4(X_H)$ but $\ps_4(Z_G(\mathbf{x};t))\neq \ps_4(Z_H(\mathbf{x};t))$}
        \label{fig:same_ps4_xg_different_ps4_zg}
    \end{figure}
    
\begin{conj}\label{conj:LW-for-ZG}
    For all $n$, $\ps_n(Z_G(\mathbf{x};t))$ distinguishes trees with at most $n$ vertices. 
\end{conj}

From \Cref{prop:psnzg-equal-implies-psnxg-equal} and Examples \ref{ex:345 354} and \ref{ex:same-PSX--different-PSZ}, we can see that this conjecture is strictly weaker than the conjecture of Loehr and Warrington. Therefore it is true for $n\le 17$.
\begin{cor}\label{cor.lw.implies}
    If Conjecture~\ref{conj:LW} is true, then \Cref{conj:LW-for-ZG} is true. 
\end{cor}

\begin{proof}
    Follows from Proposition~\ref{prop:psnzg-equal-implies-psnxg-equal}.
\end{proof}

By Corollary~\ref{cor.lw.implies}, Conjecture~\ref{conj:LW-for-ZG} is true for $n\le 17$. While weaker, \Cref{conj:LW-for-ZG} is nevertheless still enough to imply a positive answer to Stanley's Tree Isomorphism Conjecture.
\begin{proposition}
    If \Cref{conj:LW-for-ZG} is true, then $X_G(\mathbf{x})$ distinguishes trees.
\end{proposition}
\begin{proof}
    If $G$ and $H$ are two trees on at most $n$ vertices (for some positive $n$) and $X_G(\mathbf{x})=X_H(\mathbf{x})$, then $Z_G(\mathbf{x};t)=Z_H(\mathbf{x};t)$. Hence,  $\ps_n(Z_G(\mathbf{x};t))=\ps_n(Z_H(\mathbf{x};t))$ and so \Cref{conj:LW-for-ZG}, if true, implies that $G\cong H$.
\end{proof}

The following proposition records some basic properties of $\ps_n(Z_G(\mathbf{x};t))$ that we will make use of in the later sections.

\begin{proposition}\label{prop:basic-properties-psn(Z_G)}
Let $n\geq 2$ and $G$ be a graph with at least 2 vertices such that $\chi_G(n) > 0$.  
Suppose the minimum and maximum degrees of $q$ appearing in $\ps_n(Z_G(\mathbf{x}; t))$ are $d$ and $D$, respectively.
\begin{enumerate}
\item The number of vertices in $G$ is equal to $\frac{d+D}{n-1}$.
\item Given any $j \in \{d,d+1,\dots,D\}$, the coefficient of $q^j$ is equal to that of $q^{d+D-j}$.
\item Suppose $G$ is a connected bipartite graph. Moreover, let $c_1$ (resp. $c_2$) be the size of the smallest (resp. largest) part of the unique bipartition 
of $G$. Then,
\[[t^2]\ps_n(Z_G(\mathbf{x};t)) = \sum_{\substack{j_1,j_2\in \{0,1,\dots,n-1\}\\ j_1\neq j_2}} q^{c_1j_1+c_2j_2}.\]
\end{enumerate}
\end{proposition}
\begin{proof}
Let $\PC_G(n)$ be the set of proper colorings $\kappa:V\to [n]$. First, observe that $d$ and $D$ are the minimum and maximum weights, respectively, of all colorings $\kappa\in \PC_G(n)$. Second, define a map $\psi:\PC_G(n)\to \PC_G(n)$ as follows: for $\kappa\in \PC_G(n)$, let $\rho:=\psi(\kappa)$ be the coloring defined by $\rho(v)=n-\kappa(v)+1$ for $v\in V$. It is straightforward to check that $\rho$ is a proper coloring. Moreover, if $\kappa$ has weight $w$, then the weight of $\rho=\psi(\kappa)$ is

\[\sum_{v\in  V}(\rho(v)-1)=\sum_{v\in  V}n-\kappa(v)+1-1=(n-1)|V|-w.\]

(1) Let $\kappa\in \PC_G(n)$ be a proper coloring of $G$ of minimum weight $d$. Then, $\rho=\psi(\kappa)$ has weight $(n-1)|V|-d$. Suppose for sake of contradiction that $(n-1)|V|-d < D$. Then if $\tau\in \PC_G(n)$ is a coloring of weight $D$ and $\sigma=\psi(\tau)$, then $\sigma$ has weight $(n-1)|V|-D$. Since, $(n-1)|V|-d<D$, we get that $(n-1)|V|-D<d$ which means that $\sigma$ has strictly smaller weight than $\kappa$ which is a contradiction of the minimality of $\kappa.$ Thus, $\rho$ must achieve maximal coloring and hence $(n-1)|V|-d=D$, which gives the desired result.

(2) The coefficient of $q^j$ is
\[\sum_{i\geq 0}c_{i,j}t^{i},\] 
where $c_{i,j}$ is the number of proper colorings of  length $i$ and weight $j$. It is straightforward to check that $\psi$ is an involution and that if $\kappa$ uses $i$ distinct colors, then $\psi(\kappa)$ also uses $i$ distinct colors. Thus, $c_{i,j}=c_{i,d+D-j}$ for all $i,j$ and so the claim follows since $(n-1)|V|=d+D$ by part (1).

(3) Let $G$ be a connected bipartite graph with at least 2 vertices. Then $G$ has a unique bipartition into sets $C_1$ and $C_2$. Without loss of generality, assume that $|C_i|=c_i$ for $i\in \{1,2\}$. Hence, any $\kappa\in\PC_G(n)$ that uses exactly $2$ distinct colors is obtained by picking distinct $k_1, k_2\in [n]$ and then coloring all of the vertices of $C_1$ as $k_1$ and all of the vertices of $C_2$ as $k_2$. Thus, the claim follows.\qedhere

\end{proof}
As an immediate consequence, as well as preview of our results below, we can use $\ps_n(Z_G(\mathbf{x};t))$ to recover some information about caterpillars.
\begin{cor}\label{cor:recover-sizes-bipartition-caterpillar}
   Let $n\geq 2$. Consider a caterpillar of type $\alpha=(\alpha_1,\dots,\alpha_k)$. For $i\in [k]$, let $\overline{\alpha_i}=\alpha_i-1$.
   \begin{enumerate}
       \item If $k$ is odd, then we can recover the following quantities from $\ps_n(Z_{\alpha}(\mathbf{x};t))$:
   \[\begin{cases}
       \overline{\alpha_1}+\overline{\alpha_3}+\cdots+\overline{\alpha_k}+\lfloor k/2\rfloor, \\
       \overline{\alpha_2}+\overline{\alpha_4}+\cdots+\overline{\alpha_{k-1}}+\lceil k/2\rceil.
   \end{cases}\]
   \item  If $k$ is even, then we can recover the following quantities from $\ps_n(Z_{\alpha}(\mathbf{x};t))$:
   \[\begin{cases}
       \overline{\alpha_1}+\overline{\alpha_3}+\cdots+\overline{\alpha_{k-1}}+ k/2, \\
       \overline{\alpha_2}+\overline{\alpha_4}+\cdots+\overline{\alpha_k}+k/2.
   \end{cases}\]
   \end{enumerate}
\end{cor}
\begin{proof}The unique bipartition (with sizes $c_1\leq c_2$) of a caterpillar alternates membership along its spine. 
From this observation, it is straightforward to see that parts of the bipartition have sizes as in the claim (in their respective cases). Now let $m$ (resp. $M$) be the minimum (resp. maximum) $q$-degree of $[t^2]\ps_n(Z_G(\mathbf{x};t))$. From \Cref{prop:basic-properties-psn(Z_G)} part (3), we have that $c_1=m$ and $c_2=\frac{M-(n-2)m}{n-1}$, and so the claim follows.
\end{proof}
\begin{remark}
    The minimum $q$-degree of $\ps_n(Z_G(\mathbf{x};t))$ is the same as that of $\ps_n(X_G(\mathbf{x}))$. For a tree, this common value need \textit{not} be equal to the size of the smallest part of the unique bipartition. In other words, a coloring of the tree of minimal weight (which contributes to the minimal $q$-degree) is not always achieved by a $2$-coloring. The problem of finding a coloring of minimal weight of a graph, called the \textit{chromatic sum problem}, was first investigated by Kubicka \cite{Kubicka1989} who showed it to be an NP-complete problem for general graphs. For trees, she constructed for each positive integer $k$ a family of trees whose chromatic sum requires at least $k$ colors to be achieved. As a consequence, the bipartition of a tree may be much harder to recover from $\ps_n(X_G(\mathbf{x}))$ than it is to recover from $\ps_n(Z_G(\mathbf{x};t))$.
\end{remark}

\subsection{$\ps_2$ classes}

Recall from \Cref{prop:basic-properties-psn(Z_G)} that if $G$ is a connected bipartite graph, then 
\[\ps_2(Z_G(\mathbf{x};t))=t^2(q^{c_1}+q^{c_2})\]
where $c_1\leq c_2$ are the sizes of the parts in the unique bipartition of $G$ into independent sets. The \textbf{star on $k$ vertices} $\st_k$ is the only tree with bipartition sets of sizes $1$ and $k-1$, and therefore $\st_k$ is the unique tree in the $\ps_2$ class for $t^2(q+q^{k-1})$.

 We next study the $\ps_2$ classes of caterpillars. Since the unique bipartition of a caterpillar is created by alternating membership along its spine as explained in the proof of Corollary~\ref{cor:recover-sizes-bipartition-caterpillar}, we characterize all types of caterpillars in the same $\ps_2$ class in the following proposition. 

\begin{proposition}\label{prop:ps2-classes-caterpillars}

 All caterpillars with $k\geq 4$ vertices and $\ell \geq 2$ vertices in the spine in the $\ps_2$ class for $t^2(q^{c_1} + q^{c_2})$ with $2\leq c_1 \leq c_2$ are of the following possible types, up to reversal of the composition. 
 \begin{itemize}
     \item If $\ell$ is even, then $(a_1, b_1, a_2, b_2, \ldots, a_{\ell/2}, b_{\ell/2}) \models k$ where $\sum_{i=1}^{\ell/2} b_i = c_1$.
     
     \item If $\ell$ is odd, then either 
     \begin{itemize}
         \item $(a_1, b_1, a_2, b_2, \ldots, b_{\lfloor \ell/2\rfloor}, a_{\lceil \ell/2 \rceil}) \models k$ where $\sum_{i=1}^{\lfloor \ell/2 \rfloor} b_i = c_1-1$, or
         \item $(b_1, a_1, b_2, a_2, \ldots, a_{\lfloor \ell/2\rfloor}, b_{\lceil \ell/2 \rceil}) \models k$ where $\sum_{i=1}^{\lceil \ell/2 \rceil} b_i = c_1+1$.
     \end{itemize}
 \end{itemize}

\end{proposition}

 \begin{example} The list in Table~\ref{tab:ps2-class-example} contains all caterpillars with $k=8$ vertices in the $\ps_2$ class for  $t^2(q^3 + q^5)$, listed by number of vertices in the spine, $\ell$. Note that the maximum number of vertices in the spine is $\ell = 2(3)-1 = 5$.

\begin{table}[]
    \centering
    \scalebox{0.9}{\begin{tabular}{ccc c ccc}
\textbf{$\ell$} & \textbf{$\alpha$} & \textbf{caterpillar} & \hspace{3em} &
\textbf{$\ell$} & \textbf{$\alpha$} & \textbf{caterpillar}\\
\cline{1-3}
\cline{5-7} \rule{0pt}{1.5em}
$2$ & \boldtuple{3,5}         & \caterpillar{3,5} & \hspace{3em} &
$4$ & \boldtuple{2,1,3,2}     & \caterpillar{2,1,3,2} \\[1.7em]
\cline{1-3}     \rule{0pt}{1.5em}
$3$ & \boldtuple{2,2,4}       & \caterpillar{2,2,4} & \hspace{3em} &
    & \boldtuple{3,1,2,2}     & \caterpillar{3,1,2,2} \\[1.7em]
    & \boldtuple{3,2,3}       & \caterpillar{3,2,3} & \hspace{3em} &
    & \boldtuple{4,1,1,2}     & \caterpillar{4,1,1,2} \\[1.7em]
\cline{5-7}     \rule{0pt}{1.5em}
    & \boldtuple{2,4,2}       & \caterpillar{2,4,2} & \hspace{3em} &
$5$ & \boldtuple{2,1,1,1,3}   & \caterpillar{2,1,1,1,3} \\[1.7em]
\cline{1-3}     \rule{0pt}{1.5em}
    &                         &                     & \hspace{3em} &
    & \boldtuple{2,1,2,1,2}   & \caterpillar{2,1,2,1,2} \\[1.7em]
\cline{5-7}     \rule{0pt}{1.5em}

\end{tabular}}
    \caption{The $\ps_2$ class of caterpillars for $t^2(q^3+q^5)$ listed by number of vertices in the spine, $\ell$. The bipartition is represented by the filled and unfilled vertices.
    }
    \label{tab:ps2-class-example}
\end{table}

 \end{example}

The description given in Proposition~\ref{prop:ps2-classes-caterpillars} shows that there is at most one {bistar} within each $\ps_2$ class.
\begin{cor}
    Fix a number of vertices $k$. For all $2 \leq c \leq \lfloor \frac{k}{2} \rfloor$, there is a unique bistar $C_{(c, k-c)}$ in the $\ps_2$ class for $t^2(q^c + q^{k-c})$. 
\end{cor}
The $\ps_2$ classes of caterpillars may contain multiple {elongated bistars}.

\begin{cor}\label{cor:ps2-class-el-bistars}
    Fix a number of vertices $k$ and values $c_1 \leq c_2$ such that $c_1+c_2=k$. The elongated bistars in the $\ps_2$ class for $t^2(q^{c_1}+q^{c_2})$ are of type 
     \begin{itemize}
        \item $(c_2 - \frac{r}{2}, 1^r, c_1-\frac{r}{2})$ with $r$ even, $0\leq r \leq 2c_1-4$, and
        \item $(a, 1^{2c_1-3}, b)$ with $2\leq b\leq a $ satisfying $a+b=c_2-c_1+3$. 
    \end{itemize}
 
\end{cor}

\begin{proof}
  
Suppose the elongated bistar ${(a, 1^r, b)}$ with $a\geq b$ has $$\psgn{{(a, 1^r, b)}}{2} = t^2(q^{c_1} + q^{c_2}).$$ 
    Let $(V_1, V_2)$ be the bipartition of ${(a, 1^r, b)}$ with $|V_1| = c_1, |V_2| = c_2$. 
    Since $a\geq b$, the $a-1$ leaves on the left-most spine vertex are in $V_2$ while the spine vertex itself is in $V_1$. 
     Note that the spine vertices must alternate between $V_1$ and $V_2$. 
     
     If $r$ is even, the right-most spine vertex is in $V_2$ and the $b - 1$ leaves on the last spine vertex are in $V_1$; therefore, $c_2=(a-1)+r/2+1$ so $a=c_2-r/2$, and $c_1=1+r/2+(b-1)$ so $b=c_1-r/2$ as claimed.
     If $r$ is odd, the right-most spine vertex is also in $V_1$ and the $b -1$ leaves on the last spine vertex are in $V_2$. Therefore,  $c_2=a+b-2+\frac{r+1}{2}$ while $c_1=\frac{r-1}{2}+2$. This implies that $r=2c_1-3$ and $a+b=c_2-c_1+3$ as claimed.
\end{proof}

The $\ps_2$ classes of caterpillars also contain multiple caterpillars of type $(a,b,c)$ for some $a\geq c\geq 2$ and $b \geq 1$. We summarize the consequences of Proposition~\ref{prop:ps2-classes-caterpillars} for caterpillars with three internal vertices in the following corollary. 

\begin{cor} \label{cor:ps2-class-3cat}
   Fix a number of vertices $k$ and values $c_1 \leq c_2$ such that $c_1 + c_2 = k$. The caterpillars with three internal vertices in the $\ps_2$ class for $t^2(q^{c_1} + q^{c_2})$ are of type $(a, c_1-1, c_2+1-a)$ or $(a, c_2 - 1, 
c_1+1-a)$.
\end{cor}
\begin{proof}
  Let $c_1$ be the smaller of the two sizes of the bipartition of $C_{(a,b,c)}$. We can see that
    \[c_1=\min\{a+c-1,b+1\}.\]
    If $c_1=a+c-1$, then $k=a+b+c=c_1+b+1$, so $b=k-c_1-1$. If $c_1=b+1$, then $b=c_1-1$, as desired.
\end{proof}

\section{Reconstructing elongated bistars}
\label{sec:bistars}
Throughout, we will often write $(a, 1^r, b)$ as shorthand for the elongated bistar $C_{(a, 1^r, b)}$.

To prove that $\psgn{(a, 1^r, b)}{3}$ distinguishes elongated bistars, we provide an explicit, deterministic algorithm recovering $(a, r, b)$ from $\psgn{(a, 1^r, b)}{3}$: Algorithm~\ref{alg:elongated-bistars}. The remainder of this section proves that this algorithm works.
\begin{algorithm}
    \caption{Recover $(a, r, b)$ from the polynomial $f = \ps_3(Z_{{(a, 1^r, b)}}(\mathbf{x}; t))$}\label{alg:elongated-bistars}
\begin{algorithmic}[1]
\Require A $(q,t)$-polynomial of the form $f = \ps_3(Z_{{(a, 1^r, b)}}(\mathbf{x}; t))$
\Ensure The triple $(a,r,b)$

\State $d_2 \gets  \mindeg ([t^2]f)$ and $D_2 \gets  \maxdeg ([t^2]f)$
\State $N \gets \frac{1}{2}(d_2 + D_2)$ \Comment{$N = $ total number of vertices}\\
 \LINECOMMENT{Gather special $q$-degrees from $f$}

\State $d_3 \gets  \mindeg([t^3]f)$\\ 

    \LINECOMMENT{From these degrees, recover $r$ (and thus also $z:= a + b$) and  $b$.}
    \If{$d_3 - d_2 \neq 1$ or $[q^{d_3}t^3]f > d_3 - 1$} \Comment{$r$ is even}
        \State $r \gets 2d_3 - 6$ and  $z \gets N - r$ 
        \State $b \gets 3 - (d_3 - d_2)$\\
    
    \Else\Comment{$r$ is odd}

        \State $r \gets 2d_3 - 5$ and $z \gets N - r$
        \If{$z = 4$}
        \State $b \gets 2$
        \Else \Comment{$a > 2$}
        \If{$[q^{d_3 + 1}t^3]f = 2N-3+\delta_{r\geq 2} + \binom{\lfloor r/2 + 1 \rfloor}{2}$}
        \State $b \gets 2$
        \Else
        \If{$z = 6$}
        \State $b \gets 3$
        \Else 
        \State $m \gets [q^{d_3 + 2}t^3]f - F(z - 2, r)$\Comment{See def. of $F$ in Equation \eqref{eqn:F-def}}
        \If{$m > \binom{z - 4}{2} + 1$} 
        \State $b \gets 3$
        \Else 
        \State $b \gets \min\left(\left\{x: \binom{z - x - 1}{2} + \binom{x - 1}{2} - m = 0\right\}\right)$
        \EndIf
        \EndIf
        \EndIf
        \EndIf\\
    \EndIf\\
    \LINECOMMENT{Determine $a$ and return the triple}
    \State $a \gets z - b$\\
    \Return{$(a, r, b)$}
\end{algorithmic}
\end{algorithm}

    \begin{theorem}\label{thm:recovering-elongated-bistars}
        Elongated bistars ${(a, 1^r, b)}$ can be recovered from $\ps_3 (Z_{{(a,1^r,b)}}(\mathbf{x};t))$ by \Cref{alg:elongated-bistars}.
    \end{theorem}

To prove \Cref{thm:recovering-elongated-bistars}, we will need several lemmas, some of which are in the appendix. We review some notation used to state them.

Recall that we write $[q^it^j]$ to mean the coefficient $[q^it^j]\psgn{G}{n}$ when $G$ and $n$ are clear from the context. We will also write $\mindeg f(q)$ to denote the smallest exponent $i$ such that $q^i$ appears in polynomial $f(q)$ with non-zero coefficient. 

For an elongated bistar $C_{(a,1^r,b)}$, we let its spine be the path $v_A-s_1-\cdots-s_r-v_B$. The set of leaves of $v_A$ ($v_B$, respectively) is denoted $A$ ($B$, respectively). 

First, we recover $\floor{\frac{r}{2}}$.
\begin{lemma}\label{lem:EB-t3-mindeg}
    For all $r\geq 0$ and $a\geq b\geq 2$, we have
\[\mindeg [t^3]\ps_3(Z_{{(a,1^r,b)}}(\mathbf{x}; t))=\left\lfloor \frac{r}{2}\right\rfloor+3.\]
Thus, the quantity $\lfloor \frac{r}{2}\rfloor$ is recoverable from $\psgn{(a,1^r,b)}{3}$.
\end{lemma}
In proofs, we use bold numbers to represent colors for clarity: $\mb{0,1,2}.$
\begin{proof}
    Consider the proper coloring that assigns $\mb{2}$ to $v_A$, assigns $\mb{1}$ to $v_B$, $\mb{0}$ to the leaves of $A$ and $B$, and alternates colors on the $r$-spine between $\mb{0}$ and $\mb{1}$ such that $s_r$ is colored $\mb{0}$. This forms a 3-coloring of weight $2+1+\lfloor r/2\rfloor.$ We claim this is minimal among 3-colorings. 
    
    Since we must use a $\mb{2}$ at least once, the only way to get a smaller weight coloring is to use fewer $\mb{1}$'s. However, the $r$-spine requires a minimum of $\lfloor r/2\rfloor$  non-$\mb{0}$ colors; say this $r$-spine coloring has $t$ $\mb{2}$'s  and $\lfloor r/2\rfloor-t$ $\mb{1}$'s, (yielding $r$-spine weight $\floor{r/2}+t$). We then consider compatible colorings of $v_A$ and $v_B$ and the vertices of $A$ and $B$. The rest of the graph needs at least two non-$\mb{0}$'s. If $t=0$, then the rest of the graph needs a $\mb{2}$ in particular, so this part of the coloring has weight at least 3. Hence, the total weight is at least $\lfloor r/2\rfloor +3$. If $t>0$, then the rest of the graph has weight at least 2, so the total weight is at least $\lfloor r/2\rfloor+t+2\geq \floor{r/2}+3.$ As this results in total weight at least $\lfloor{r/2}\rfloor+3$ in either case, the claim is proved.
\end{proof}

Next, we prove a lemma that is a step towards recovering $r$.
\begin{lemma}\label{lem:EB-t2-mindeg}
    For all $r\geq 0$ and $a\geq b\geq 2$, we have 
    \[\mindeg [t^2]\ps_3(Z_{(a,1^r,b)}(\mathbf{x};t))=\begin{cases}
        \lfloor\frac{r}{2}\rfloor + b&\text{if $r$ is even},\\
        \lfloor\frac{r}{2}\rfloor + 2&\text{if $r$ is odd}.\\
    \end{cases}\]
\end{lemma}
\begin{proof}
    By \Cref{prop:basic-properties-psn(Z_G)}(3), $\mindeg [t^2]\ps_3(Z_{{(a,1^r,b)}}(\mathbf{x};t))$ is $c_1$,
     the size of the smaller part of the bipartition of the vertices of ${(a,1^r,b)}$. The two parts of the bipartition have sizes
\[\begin{cases}1+\lfloor r/2\rfloor +(b-1)\quad \text{and}\quad  a-1+\lfloor r/2\rfloor +1 &\text{if $r$ even,}\\
    1+\lfloor r/2\rfloor +1\quad \text{and}\quad  a-1+\lceil r/2\rceil + b-1 &\text{if $r$ odd.}
    \end{cases}\]
    Since $a\geq b \geq 2$, the claim follows.
\end{proof} 
 
 Let $d_2=\mindeg [t^2]\ps_3(Z_{{(a,1^r,b)}}(\mathbf{x};t))$ and $d_3=\mindeg [t^3]\ps_3(Z_{{(a,1^r,b)}}(\mathbf{x};t))$ for an unknown elongated bistar $(a,1^r,b)$. Since we can recover the value of $\lfloor r / 2\rfloor$, it is enough to  recover the parity of $r$. The proof will rely on a computation (\Cref{lem:r_by_2_plus_3}) found in \Cref{sec:Append.A}.
\begin{lemma}\label{lem:recover-r-parity}
    If $d_3-d_2\neq 1$ or $[q^{d_3}t^{3}]>d_3-1$, then $r$ is even. Otherwise, $r$ is odd.
\end{lemma}

\begin{proof} Comparing \Cref{lem:EB-t3-mindeg} and \Cref{lem:EB-t2-mindeg}, we have $d_3-d_2\leq 1$ in general. If $d_3-d_2 < 1$, then $r$ is even from \Cref{lem:EB-t2-mindeg}. Suppose $d_3-d_2=1$. From \Cref{lem:r_by_2_plus_3}, we have 
    \[[q^{d_3}t^3]=\begin{cases}
        4\floor{r/2}+6 & \text{if $r$ is even and $a=b=2$,}\\
        3\floor{r/2}+4 & \text{if $r$ is even and $a>b=2$,}\\
        \floor{r/2}+2 & \text{if $r$ is odd.}
    \end{cases}\]
Hence, if $[q^{d_3}t^3]>d_3-1=\lfloor r/2\rfloor +2$, then $r$ must be even. Otherwise, $r$ is odd.
    \end{proof}

\begin{cor} \label{lem:recover-r}
For any elongated bistar ${(a,1^r,b)}$, we can determine the value of $r$ from $\ps_3 (Z_{{(a,1^r,b)}}(\mb{x};t))$.
\end{cor}
\begin{proof}
Use \Cref{lem:recover-r-parity} to recover the parity of $r$. Then use \Cref{lem:EB-t3-mindeg} to recover $r$: if $r$ is even, then $r=2d_3-6$; if $r$ is odd, then $r=2d_3-5.$
    \end{proof}

In Appendix~\ref{sec:Append.A}, we make a few more computations of higher coefficients of $\ps_3(Z_{{(a, 1^r, b)}}(\mathbf{x}; t))$ that help us to recover the values of $a$ and $b$ and thus prove \Cref{thm:recovering-elongated-bistars}.

    \begin{proof}[Proof of \Cref{thm:recovering-elongated-bistars}]
    We verify  our proposed recovery algorithm (Algorithm~\ref{alg:elongated-bistars}). When relevant, we include the line numbers from the algorithm.

    First, observe the total number of vertices $N$ (line 2) is correct by \Cref{prop:basic-properties-psn(Z_G)}.

      Lines 8 and 9 (determining precisely when $r$ is even and recovering its value in this case) follow from \Cref{lem:recover-r} and \Cref{lem:EB-t3-mindeg}. Line 10 (recovering $b$ when $r$ is even) then follows from \Cref{lem:EB-t2-mindeg}.

      Lines 12 and 13 (concluding that $r$ is odd and determining its value) also follow from \Cref{lem:recover-r} and \Cref{lem:EB-t3-mindeg}. When $r$ is odd, determining $b$ can be done as follows.

      If $z := a + b =  4$, it must be that $b = 2$, since $a \geq b$, justifying lines 14-15. Otherwise, we move to line 16.
      
      Lines 17 and 18 (determining that $b = 2$ from $[q^{d_3 + 1}t^3]f$) follows from  \Cref{lem:elongated-bistar-c1-odd}. Otherwise, the same lemma tells us $b > 2$ and we move to line 19.

      Lines 20-21 are similar to 14-15. Finally, the ``else'' case (lines 22 - 27) is proved within the proof of \Cref{lem:recovering-a-b-r-odd}.

      In all cases, by the end of line 33, we know the values of $r$, $b$, and $z$. Since $a = z - b$, we have recovered the elongated bistar.
      \end{proof}

\section{Reconstructing caterpillars with 3 internal vertices}
\label{sec:3-cat}

In this section, we show that $\ps_3(Z_G(\mathbf{x};t))$ distinguishes caterpillars with 3 internal vertices, which are caterpillars of type $(a,b,c)$ for some $a\geq c\geq 2$ and $b \geq 1$. We do so by providing an explicit reconstruction algorithm (Algorithm~\ref{alg:recover-abc}). Throughout this section, we will often write $(a, b, c)$ as a shorthand for the caterpillar $C_{(a, b, c)}.$

\begin{algorithm}
\caption{Recover $(a,b,c)$ from the polynomial $f= \psgn{(a,b,c)}{3}$}
\label{alg:recover-abc}
\begin{algorithmic}
\Require A polynomial $f = \psgn{(a,b,c)}{3}$
\Ensure The triple $(a,b,c)$
\State $d_2 \gets  \mindeg ([t^2]f)$ and $D_2 \gets  \maxdeg ([t^2]f)$
\State $N \gets \frac{1}{2}(d_2 + D_2)$ \Comment{$N = $ total number of vertices}
\State multiset $M=\{\{a+c-1,\;b+1\}\} \gets \{\{d_2, N - d_2\}\}$.
\State $c_{33}\gets [q^3t^3]f$,  $c_{43}\gets [q^4t^3]f$ and $c_{53}\gets [q^5t^3]f$.

\If{$c_{33}\neq 0$}
    \Comment{$b=1$}
    \State Recover $(a,c)$ from the formulas in Proposition~\ref{prop:b1 recovery} using $c_{53}$.
    \State \Return $(a,1,c)$
\EndIf

\Comment{$b>1$}
\If{$c_{43}>1$}
    \Comment{$\min\{a,b,c\}=2$}
    \State Recover $(a,b,c)$ using the case $\min\{a,b,c\}=2$ in Corollary~\ref{prop:recovering_from_m2}.
    \State \Return $(a,b,c)$
\Else
    \Comment{$\min\{a,b,c\}\ge 3$}
    \State $(b,s)\gets$  output of Algorithm~\ref{alg:midends} with input ${f,M}$ \Comment{$s=a+c$}
    \State Determine $m=\min\{a,b,c\}$ and multiplicity using Remark~\ref{rem.min}.
    \If{$m$ occurs with multiplicity $3$}
        \State \Return $(m,m,m)$
    \ElsIf{$m$ occurs with multiplicity $2$}
        \If{$b = m$}
            \State \Return $(N-2m,m,m)$
        \Else
            \State \Return $(m,N-2m,m)$
        \EndIf
    \Else \Comment{multiplicity = 1}
        \If{$b > m$}
        \State \Return $(N-b-m,b,m)$
        \Else
        \State Use Proposition~\ref{prop:recover from m3} to recover $(a,b,c)$
        \EndIf
    \EndIf
\EndIf
\end{algorithmic}
\end{algorithm}

\begin{theorem}\label{thm:MainTheorem}
$\ps_3(Z_G(\mathbf{x};t))$ distinguishes caterpillars with 3 internal vertices. 
Moreover, the values $(a, b, c)$ are recoverable from $\psgn{(a,b,c)}{3}$ via \Cref{alg:recover-abc}.
\end{theorem}

As a consequence, we obtain a considerable strengthening of \cite{ALISTEPRIETOZAMORA,LoeblSereni} in the case when the caterpillar has $3$ spine vertices.

\begin{cor}
The chromatic symmetric function on three variables $X_G(x_1,x_2,x_3)$ distinguishes and recovers caterpillars with 3 spine vertices.
\end{cor}

 Our general strategy is to identify certain coefficients of   $\psgn{(a,b,c)}{3}$ to reveal information about $a,b,c$. Throughout, assume that $a\geq c\geq 2$ and $b \geq 1$, and we set $\Bar{a}=a-1$, $\Bar{b}=b-1$ and $\Bar{c}=c-1$. We start by examining the coefficient of $t^2$: this tells us the sizes of the two independent sets in the unique bipartition of the tree.

\begin{lemma} \label{lem:ac1b2}
The multiset of values $\{\{\bar{a}+\bar{c}+1,\bar{b}+2\}\}$ is determinable from $\psgn{(a,b,c)}{3}.$
\end{lemma}
\begin{proof}
This follows immediately from \Cref{cor:recover-sizes-bipartition-caterpillar}.
\end{proof}
The remainder of the information necessary to recover $(a,b,c)$ comes from various $q$-degrees in the coefficient $[t^3]\psgn{(a,b,c)}{3}$.  Our strategy is to inductively assume a lower bound for $a,b,c$ and then find a degree $s$ such that  $[q^st^3]\psgn{(a,b,c)}{3}$ verifies that one of these values meets the lower bound. With this minimum value  $m:=\min\{a,b,c\}$ determined, the other values and their placement in the composition $(a,b,c)$ can be recovered from determined coefficients.

For the remainder of this section, we will understand $(a,b,c)$ as fixed and write $[q^rt^s]\psgn{(a,b,c)}{3}$ as just $[q^rt^s]$, and $[t^s]\psgn{(a,b,c)}{3}$ as just $[t^s]$ for brevity when there is no possibility of confusion. We will state formulas for various $[q^rt^s]$, the proofs of which can be found in Appendix~\ref{appendix:ps-qrt3}. Finally, throughout the remainder of the section, we will make extensive use of the Kronecker delta function $\delta_P$, which takes value 1 if proposition $P$ holds and 0 otherwise.

\begin{lemma}\label{lem:coeff-of-q3-t3}
   For all caterpillars of type $(a,b,c)$, we have
  \[ 
  [q^3t^3]=2\delta_{b = 1}=\begin{cases}
  2&\text{if $b=1$},\\
   0 &\text{if $b>1$.}
   \end{cases}
   \]
\end{lemma}
In other words, \Cref{lem:coeff-of-q3-t3} says that we can determine whether $b = 1$ from $\ps_3(Z_{(a, b, c)}(\mathbf{x};t))$.  The proof of Lemma~\ref{lem:coeff-of-q3-t3} is in Appendix~\ref{appendix:coeff-of-q3-t3}.

\begin{lemma}\label{cor:coeff-t3-q4-list} 
      For all caterpillars of type $(a,b,c)$, if $[q^3t^3] = 0$, then 
     \begin{align*}
         [q^4t^3] =1 + 2\delta_{c = 2} + 3\delta_{b = 2} + 5 \delta_{a = 2}\delta_{c = 2}= \begin{cases}
             11 & \text{ if } a = b = c = 2,\\
             8 & \text{ if } a = c =  2, b \geq 3,\\
             6 & \text{ if } b = c = 2, a \geq 3,\\
             4 & \text{ if } b = 2, a \geq 3, c \geq 3,\\
             3 & \text{ if } c = 2, a \geq 3, b \geq 3,\\
             1 & \text{ if } a, b, c \geq 3.
         \end{cases}
     \end{align*}
     \end{lemma}
The proof of Lemma~\ref{cor:coeff-t3-q4-list} is in Appendix~\ref{appendix:coeff-of-q4-t3}. From this, it follows that we can recover the following information about $\min\{a,b,c\}$ from $\psgn{(a,b,c)}{3}$ just from two coefficients.
\begin{itemize}
\item If $[q^3t^3]=2$, then $\min\{a,b,c\}=1$.
\item If $[q^3t^3]=0$ and $[q^4t^3]\neq 1$, then $\min\{a,b,c\}=2$.
\item If $[q^3t^3]=0$ and $[q^4t^3]=1$, then $\min\{a,b,c\}\geq 3$.
\end{itemize}
Now we can break up the recovery based on the values of $[q^3t^3]$ and $[q^4t^3]$, and hence based on the value of $m=\min\{a,b,c\}$. In all cases, the coefficient $[q^{m+2}t^3]$ will determine the position of any part of the composition $(a,b,c)$ that is equal to $m$. Once we have determined the position of any $m$, the coefficient $[q^{m+4}t^3]$ will determine the values of the other parts. Another key part of the argument is the following basic (``anti-'')unimodality lemma about binomial coefficients.
\begin{lemma} \label{lem:antiunimodal binomials}
    For any integers $n\geq 1$ and $0\leq r \leq n$, the quantity
\[\binom{n-r}{2}+\binom{r}{2}\text{ is }\begin{cases}
    \text{ strictly decreasing for $0\leq r \leq \frac{n}{2} $, and} \\
    \text{ strictly increasing for $\frac{n}{2} \leq r \leq n$.}
\end{cases}\]

\end{lemma}
\begin{proof}
   Recall that  for $0\leq r\leq n$ we have
  \[\binom{n-r}{2}+\binom{r}{2}+r(n-r)=\binom{n}{2},\]
  or in other words
   \[\binom{n-r}{2}+\binom{r}{2}=r^2-nr+\binom{n}{2}.\]
   Then, observe that the right hand side is a monic quadratic polynomial in $r$ with minimum at $r=n/2$. 
\end{proof}

\subsection{Min = 1}
Assume $[q^3t^3]=2$ and hence that $\min\{a,b,c\}=1=b$. Since we have used the coefficient of $q^{m+2}t^3=q^{3}t^3$ to determine the position of $m=1$, we then use the coefficient of $q^{m+4}t^3=q^5t^3$ to determine the other values $a$ and $c$ of the caterpillar $(a,1,c).$
\begin{proposition}
{\label{prop:unified_q5t3}}
For all caterpillars of type $(a,b,c)$, we have
\begin{align*}
    [q^5t^3] &= b + \left[ 2(\bar{a}+\bar{c}) + \binom{\bar{a}}{2} + \binom{\bar{c}}{2}\right]\delta_{b=1} + (2\bar{a}+2\bar{c}+3)\delta_{b=2} + 4\delta_{b=3} \\
    &\quad+3(b+1)\delta_{a=2} + (a+b)\delta_{c=2} + 4\delta_{a=3}\delta_{c=2} + 2\delta_{a=3} + 2\delta_{c=3}.
\end{align*}
\end{proposition}

The proof of Proposition~\ref{prop:unified_q5t3} is in Appendix~\ref{appendix:coeff-of-q5-t3}. Simplifying the formula of \Cref{prop:unified_q5t3} according to cases, we have the following expression for $[q^5t^3]$. 

\begin{cor} \label{cor:q5t3b1simp}
For all caterpillars of type $(a,1,c)$, we have
 \[[q^5t^3]=\begin{cases}
 14&\text{if $a=c=2$,}\\
  18&\text{if $a=3$,\ $c=2$,}\\
   5+3\bar{a}+\binom{\bar{a}}{2}&\text{if $a\geq 4$,\ $c=2$,}\\
        15 & \text{if $a=3$,\ $c=3$,}\\
        \binom{\bar{a}}{2}+2\bar{a}+8 & \text{if $a\geq 4$,\ $c=3$,}\\
        2(\bar{a}+\bar{c})+\binom{\bar{a}}{2}+\binom{\bar{c}}{2}+1 & \text{if $a\geq c \geq 4$}.
    \end{cases}\]
\end{cor}

\begin{proposition}\label{prop:b1 recovery}
    All caterpillars of type $(a,1,c)$ are recoverable from $\ps_3(Z_{(a,b,c)}(\mathbf{x};t)).$
\end{proposition} 
\begin{proof}
    \Cref{lem:coeff-of-q3-t3} allows us to determine whether the caterpillar is of type $(a, 1, c)$. When $b=1$, the multiset $\{\{\bar{a}+\bar{c}+1,\bar{b}+2\}\}$ of \Cref{lem:ac1b2} has the form $\{\{\bar{a}+\bar{c}+1,2\}\}$ with $\bar{a}+\bar{c}+1\geq 3$, so we can also recover the value $z:=\bar{a}+\bar{c}$. This means that we only need to distinguish between caterpillars with the same $z$-value.
    
    Labelling the cases of \Cref{cor:q5t3b1simp} from 1 to 6, we can see that the $z$-value is different for cases $1$ and $2$ than in all other cases, so the caterpillar can be recovered without even considering the value of $[q^5t^3]$. If case 3 and case 4 have the same $z$ value, then we must have $a=4$ in the 3rd case. But then they cannot have the same $[q^5t^3]$ value since plugging in $a=4$ into the 3rd case gives 17.

    Thus, all that remains is to distinguish the 3rd, 5th, and 6th cases from each other. For this, \Cref{lem:antiunimodal binomials} implies that for all $z\geq 2$ and $3\leq \bar{c}\leq z/2$, we have
\begin{align*}
    3z+2+\binom{z-1}{2}+\binom{1}{2}&>2z+3+\binom{z-1}{2}+\binom{1}{2}\\
    &>2z+3+\binom{z-2}{2}+\binom{2}{2}\\
    &>2z+1+\binom{z-\bar{c}}{2}+\binom{\bar{c}}{2}
\end{align*}
with the 1st inequality by the fact that $3z+2>2z+3$ for $z>1$. For $z=\bar{a}+\bar{c}$, the 1st expression is exactly the 3rd case, the 3rd is the 5th, and 4th is the 6th.

Having distinguished the 6 cases, we now recover the caterpillars. One first determines $b=1$ from \Cref{lem:coeff-of-q3-t3} and thus also $z=\bar{a}+\bar{c}.$ If $z=2$ or $z=3$, then we know the caterpillar is either $(2,1,2)$ or $(3,1,2)$, respectively. Otherwise, we examine $[q^5t^3].$ If $[q^5t^3]=15$, the caterpillar is $(3,1,3)$. Next, we check if $[q^5t^3]$ is in the set \[  \left\{3z+2+\binom{z-1}{2}+\binom{1}{2},2z+3+\binom{z-2}{2}+\binom{2}{2}\right \}.\] We know there are two distinct elements in the set from the above argument. If we get the 1st value, then we know that $z=\bar{a}+1$ and thus the caterpillar is $(z,1,2)$. If we get the 2nd value, then we know that $z=\bar{a}+2$ and thus the caterpillar is $(z-1,1,3)$. If $[q^5t^3]$ is not one of those two values, it must be of the form
\[[q^5t^3]=2z+1+\binom{z-\bar{c}}{2}+\binom{\bar{c}}{2}\]
We then solve this quadratic equation for $\bar{c}$. The two solutions are the values $\bar{a}$ and $\bar{c}$.
\end{proof}

\subsection{Min = 2}
Assume that $[q^3t^3]=0$ and $[q^4t^3]\neq 1$, so that $m=2$. The arguments will mimic the case of $m=1$, but are best handled separately from $m\geq 3$ due to the particular form of the relevant coefficients in this case. First since the values of $[q^4t^3]$ are all distinct as shown in \Cref{cor:coeff-t3-q4-list}, $[q^4t^3]$  tells us which values of $(a,b,c)$ are 2, if any. Thus, additional work only needs to be done to recover the caterpillar when $[q^4t^3]=4$ (so $a,c>b=2$) and $[q^4t^3]=3$ (so $a,b>c=2$). 
\begin{proposition}
    \label{prop:coeff-of-q6t3}
    If $b \geq 2$, then the coefficient $\ps_3(Z_{(a, b, c)}(\mathbf{x}; t))$ is 
    \begin{align*}
        [q^6t^3]&=\left({\bar{b} \choose 2} + 2(\bar{a}+\bar{c})\right) + 
        \left(4(\bar{a}+\bar{c}) + {\bar{a} \choose 2} + {\bar{c} \choose 2}\right)\delta_{b = 2}\\  &\quad+ 
        \left(2(\bar{a}+ \bar{c}) + 6\right)\delta_{b = 3} +
         5\delta_{b = 4} + \left(2{\bar{b} \choose 2} + 6\bar{b} +2\right)\delta_{a = 2}\delta_{c = 2} +\left(3\bar{b} + 11\right)\delta_{a = 3}\delta_{c = 2}\\
          &\quad+\left(\bar{b} + 11\right)\delta_{a = 3}\delta_{c = 3}  
        + 5\delta_{a = 4}\delta_{c = 2} + 
        2\delta_{a = 4}+ 
        \left(2(\bar{a}+\bar{b}) + {\bar{a} \choose 2} + {\bar{b} \choose 2}\right)\delta_{c = 2} \\
         &\quad+ (\bar{a}+\bar{b} + 4)\delta_{c = 3} + 2 \delta_{c = 4}.
    \end{align*}
\end{proposition}

The proof of Proposition~\ref{prop:coeff-of-q6t3} is in Appendix~\ref{appendix:coeff-of-q6-t3}. We use the formulas for $[q^5t^3]$ and $[q^6t^3]$ to determine the caterpillars in the leftover cases from \Cref{cor:coeff-t3-q4-list}.

\begin{lemma}\label{lem:subcase4iii}
If $[q^3t^3]=0$ and $[q^4t^3] = 4$, then we can recover $(a,b,c)$ from $\ps_3(Z_{(a, b, c)}(\mathbf{x}; t)).$
\end{lemma}

\begin{proof}
By \Cref{cor:coeff-t3-q4-list}, we know that $b = 2$ and $a, c \geq 3$. Hence, our known multiset is of the form $\{\{ a + c - 1, 3\}\}$, so we can recover the sum $z:=\bar{a} + \bar{c}$. Next, we check the coefficient $[q^5t^3].$ By Proposition~\ref{prop:unified_q5t3}, when $b = 2$ and $a,c, \geq 3$, we have 
\[[q^5t^3] = 2z +5  + 2\delta_{a = 3}\delta_{c = 3} + 2\delta_{c = 3}.\]
    We now have three cases based on the (currently unknown) values of $a$ and $c$:

    \begin{enumerate}
        \item[(i)] If $a = c = 3$, then $[q^5t^3]$ satisfies $[q^5t^3] = 2z +9$.
        \item[(ii)] If $a \geq 4 > c = 3$, then $[q^5t^3]$ satisfies  $[q^5t^3] = 2z + 7$.
        \item[(iii)]  If $a \geq c \geq 4$, then $[q^5t^3]$ satisfies $[q^5t^3] = 2z + 5$.
    \end{enumerate}
  Since we know $z$ and the coefficient $[q^5t^3]$,  we are able to distinguish among the cases (i), (ii), and (iii). In case (i), the caterpillar is $(3, 2, 3).$ In case (ii), the caterpillar is $((a + c) - 3, 2, 3)$. It remains to show we can recover the caterpillar in case (iii).

  Assume we are in case (iii), so that $a \geq c \geq 4.$ Then, \Cref{prop:coeff-of-q6t3} reduces to
    \begin{align*}
        [q^6t^3]= 6z + {\bar{a} \choose 2} + {\bar{c} \choose 2} + 2\delta_{a = 4}\delta_{c = 4} + 2\delta_{c = 4}.
    \end{align*}
    We now have three cases, based on the values of $a$ and $c$:
    \begin{enumerate}
        \item If $a = c = 4$, then $[q^6t^3] = 46$.
        \item If $a \geq 5 > c = 4$, then $[q^6t^3] = 6\bar{a}+23 + \binom{\bar{a}}{2}$.
        \item If $a, c \geq 5$, then $[q^6t^3] =  6z  + {\bar{a} \choose 2}+ \binom{\bar{c}}{2}.$
    \end{enumerate}
    Observe that in case (2), since $a \geq 5$, the coefficient $[q^6t^3] \geq 53$. Similarly, in case (3), since $a, c \geq 5$, the coefficient $[q^6t^3] \geq 60.$ Hence, if $[q^6t^3] = 46$, we can immediately conclude that $a = c = 4.$

    Assume now that $[q^6t^3] > 46$, so that we are in case (2) or (3). Observe that the formula in case (2) can be rewritten as: \begin{equation}\label{eqn:case-2-rewritten}[q^6t^3] = 6z+2+\binom{3}{2} + \binom{z- 3}{2}.\end{equation} 
   
    On the other hand, the formula in case (3) can be written as \begin{equation}\label{eqn:case-3-rewritten}[q^6t^3] = 6z + \binom{z-\bar{c}}{2} + \binom{\bar{c}}{2}.\end{equation} 

    Applying \Cref{lem:antiunimodal binomials}, we have that 
    \begin{equation}\label{eqn:final-to-show} 6z+2+\binom{3}{2} + \binom{z- 3}{2} >6z+\binom{\bar{c}}{2} + \binom{z- \bar{c}}{2},\end{equation}
    since $3<4\leq \bar{c}\leq z/2$. Thus, $[q^6t^3]$ distinguishes case (2) and case (3) for a fixed $z$, and we can recover $(a,2,c)$ as follows: if $[q^6t^3]$ satisfies \Cref{eqn:case-2-rewritten}, then we know that $(a,2,c)=(z-2,2,4)$. Otherwise, we know that $[q^6t^3]$ satisfies \Cref{eqn:case-3-rewritten}. We then solve this quadratic equation for $\bar{c}$. The two solutions are the values $\bar{a}$ and $\bar{c}$, and we know that the (weakly) smaller of the two is $\bar{c}.$
\end{proof}
\begin{lemma}\label{lem:subcase5}
Assume that $[q^3t^3]=0$ and that $[q^4t^3] = 3$. 
    Then, we can recover the individual values of $a$ and $c$ from $\ps_3(Z_{(a, b, c)}(\mathbf{x}; t)).$
\end{lemma}

\begin{proof}
    By \Cref{cor:coeff-t3-q4-list}, we can conclude that $c = 2$ and $a, b \geq 3.$ Hence, our known multiset is $\{\{ a + 1, b + 1\}\}$, and we also know the multiset $\{\{ a , b \}\}$. We are done if $\{\{a, b\}\} = \{\{3, 3\}\}$, since then the caterpillar is $(3, 3, 2).$ So, we assume this is not the case and we check the coefficient $[q^5t^3]$,  which by Proposition~\ref{prop:unified_q5t3} is
     \[[q^5t^3] = a + 2b + 4\delta_{b = 3} + 6\delta_{a = 3}  =\begin{cases}
         19&a=3,b=3,c=2,\\
         2b+9&a=3,b>3,c=2,\\
         a+10&a>3,b=3,c=2,\\
         a+2b&a>3,b>3,c=2\\
     \end{cases}.\] 
    By assumption, we are not in the first case, and so $\{\{a, b\}\}$ has either no $3$s or exactly one $3$.
    
    First, assume $\{\{a, b\}\}$ contains exactly one $3$, meaning that $\{\{a, b\}\} = \{\{3, y\}\}$ for some $y > 3$. Then either:
    \begin{enumerate}
        \item $a = 3$ and $b = y$, in which case $[q^5t^3 ] = 2y + 9$, or
         \item $a = y$ and $b = 3$, in which case $[q^5t^3] = y + 10$.
    \end{enumerate}
    There are no solutions for $2y + 9 = y + 10$ with $y > 3$, and so we are able to know in which case we are by checking which equation is satisfied by $y$ and $[q^5t^3].$ Once we know in which case we are, we know the caterpillar: in the first subcase, we have the caterpillar $(3, y, 2)$; in the second, we have $(y, 3, 2).$
    
    Finally, assume the multiset $\{\{a, b\}\}$ has no 3's and hence $a,b>3$. Then, $c = 2$ and $[q^5t^3] = a + 2b = (a + b) + b.$ Since we know the value of $a + b$ and the coefficient $[q^5t^3]$, we can solve for $b$ (and then $a$). Hence, we will be able to recover the caterpillar, completing the proof.
\end{proof}

\begin{cor}
\label{prop:recovering_from_m2}
   If  $\min \{a, b, c\} =  2$, then we can determine the values  $(a, b, c)$ from $\ps_3(Z_{(a, b, c)}(\mathbf{x}; t)).$
\end{cor} 

\begin{proof}
We collect the various statements proven so far.

By assumption, we know that $\{a, b, c\}$ contains a $2$ and that $b \neq 1$. 

 By \Cref{cor:coeff-t3-q4-list}, there are five possibilities for the coefficient $[q^4t^3]$ given our assumptions: $11, 8, 6, 4$ or $3$. Recall that from Lemma~\ref{lem:ac1b2}, we also know the multiset of values $\{\{ a +c -1, b + 1\}\}.$ We show that we can recover $(a, b, c)$ from the given multiset and the coefficient $[q^4t^3]$ by examining each of the five cases below.

\begin{enumerate}
    \item If $[q^4t^3] = 11$, we know that $(a, b, c) = (2, 2, 2).$
    \item If $[q^4t^3] = 8$, we know that $a = c = 2$ and $b \geq 3.$ Hence, the given multiset is of the form $\{\{3, b + 1\}\}$ allowing us to recover $b$ as well. 
    \item  If $[q^4t^3] = 6$, then $b = c = 2$ and $a \geq 3.$ Hence, the given multiset is of the form $\{\{a + 1, 3\}.$ Since $a + 1 > 3$, we can uniquely recover $a.$
    \item If $[q^4t^3] = 4$, then $(a, b, c)$ is recoverable by \Cref{lem:subcase4iii}.
    \item  If $[q^4t^3] = 3$, then $(a, b, c)$ is recoverable by \Cref{lem:subcase5}.
\end{enumerate} 
\end{proof}
\subsection{Min \texorpdfstring{$\geq$}{>=} 3}

In this section, we assume $[q^3t^3] = 0$ and $[q^4t^3] = 1$, which by Lemma~\ref{cor:coeff-t3-q4-list} means we have $a,b,c \ge 3$.
We proceed in three stages. First, Proposition~\ref{prop:multiset vals} recovers the two values $b$ and $a+c$ and thus shows the correctness of \Cref{alg:midends}. Next, Proposition~\ref{prop:q_m_plus_2} determines the minimum value among $a,b,c$ as well as its multiplicity. This allows us to recover the values of $a$ and $c$ as long as $b$ is not the unique minimum. Finally, Proposition~\ref{prop:recover from m3} determines $a$ and $c$ when $b$ is the unique minimum.

\begin{algorithm}
\caption{Recover $b$ and $a+c$ when $\min\{a,b,c\}\ge3$}
\label{alg:midends}
\begin{algorithmic}
\Require $f$, multiset $M$
\Ensure $(b,a+c)$

\State $c_{53}\gets[q^5t^3]f$

\If{$4\in M$}
    \State $b\gets3$
\ElsIf{$\max(M) = \min(M)$}
    \State $b \gets \min(M) - 1$
\ElsIf{$5\in M$}
    \If{$c_{53} < 7$}
        \State $b \gets 4$
    \Else
        \State $b \gets c_{53} - 4$
    \EndIf
\ElsIf{$6\in M$}
    \If{$c_{53} \le 7$}
        \State $b \gets 5$
    \Else
        \State $b \gets c_{53} - 2$
    \EndIf
\Else
    \If{$c_{53} = \max(M) + 1$}
        \State $b \gets c_{53} - 2$
    \ElsIf{$c_{53} = \min(M) - 1$}
        \State $b \gets c_{53}$
    \ElsIf{$\max(M) - \min(M) \neq 2$}
        \If{$c_{53} = \max(M) - 1$}
            \State $b \gets c_{53}$
        \Else
            \State $b \gets c_{53} - 2$
        \EndIf
    \Else
        \State $c_{63}\gets[q^6t^3]f$
        \State $x \gets \frac{1}{2}(\sum M - 4)$
        \Comment{$\sum M$ means the sum of all elements in M}
        \If{$c_{63} = \binom{x-1}{2}+ 2(x+2) + (2x+3)$}
            \State $b \gets x$
        \Else
            \State $b\gets \max(M)-1$
        \EndIf
    \EndIf
\EndIf

\State \Return $(b,\sum M-b)$
\end{algorithmic}
\end{algorithm}

To start, we give a specialized version of \Cref{prop:coeff-of-q6t3}.
\begin{cor} \label{prop:q6t3}
Suppose $a \ge 4$, $b > 5$,  and $a\geq c \geq 3$, then
    \[[q^6t^3]=\binom{\bar{b}}{2}+2(\bar{a} + \bar{c})+(\bar{a}+\bar{b}+4)\delta_{c=3}+2\delta_{a=4} + 2\delta_{c=4}.\] 
\end{cor}

\begin{lemma}
\label{lem:q5 recovers b}
When $\min\{a,b,c\}\geq 3$, we have
\begin{equation}
\label{eq.q5t3}
[q^5t^3]=
\begin{cases}
    11 & \text{ if $a=b=c=3$,}\\
    9 & \text{ if $a > b=c=3$,}\\
    7 & \text{ if $a,c>3 = b$,} \\
    b+4 & \text{ if $b > 3 = a = c$,}\\
    b+2 & \text{ if $a,b>3 = c$,}\\
    b & \text{ if $a,b,c\geq 4$.}
\end{cases}
\end{equation}
Furthermore, we can distinguish between each of these cases using the multiset $\{\{m_1,m_2\}\}=\{\{a+c-1,b+1\}\}$, $[q^5t^3]$, and $[q^6t^3]$.
\end{lemma}

\begin{proof}
The values come directly from simplifying the formula of \Cref{prop:unified_q5t3} in each case. Next we explain how to distinguish between these cases. 

Denote the multiset $\{\{m_1,m_2\}\}=\{\{a+c-1,b+1\}\}$ for some $m_1\geq m_2$. Since $a,b,c\geq 3$, we always have $m_1\geq 5$ as $a+c-1 \ge 3 + 3 -1$, and we get $m_2=4$ if and only if $b=3$.  Thus, by inspecting the coefficient $[t^2]$, we can discern whether we are in one of the first three cases, and the value of $[q^5t^3]$ immediately tells us which.

It remains to discern between the cases with $m_2\geq 5$.  If $m_1=m_2$, then we can recover the values of $b$ and $a+c$. The value of $[q^5t^3] - b$ will be $4, 2$ or $0$, and they each correspond to the fourth, fifth and sixth case of (\ref{eq.q5t3}). In the rest of the proof, we assume $m_1 > m_2$.

Notice that $a + c - 1$ equals 5 if and only if $a = c = 3$. Supposing $m_2=5$, we have three possibilities:
\[[q^5t^3]=
\begin{cases}
b+4 & \text{ if } a=c=3 \text{ and } b \ge 4 \ (m_2 = a + c - 1),\\
6 &  \text{ if } a > 3 = c \text{ and } b = 4 \  (m_2 = b+1),\\
4 & \text{ if } a,c>3 \text{ and } b = 4 \ (m_2=b+1).
\end{cases}
\]
Note $b+4\geq 7$. If $[q^5t^3]\geq 7$, we can conclude that we are in the first possibility, or the fourth case in (\ref{eq.q5t3}); otherwise, the value of $[q^5t^3]$ (which is 6 or 4) tells us if we are in the fifth or the sixth case in (\ref{eq.q5t3}). 
    
Finally, we consider the case $m_1 > m_2 > 5$. We must distinguish between the final two cases in (\ref{eq.q5t3}). 
Supposing first that $m_2=6$, we have three possibilities:
\[[q^5t^3]=\begin{cases}
    7 & \text{if $c=3$, $a>4$ and $b=5$,}\\
    b+2 & \text{if $c=3$, $a=4$ and $b>5$,}\\
    5 & \text{if $a \geq c>3$ and $b=5$.}
\end{cases}\]
If $[q^5t^3] \ge 7$, we must have one of the first two possibilities (which can be distinguished by $b+2>7$ for $b>5)$, corresponding to the fifth case in (\ref{eq.q5t3}); otherwise $[q^5t^3] = 5$ and we are in the sixth case in (\ref{eq.q5t3}).

From now on, we assume $m_1 > m_2 > 6$. Together with $a\ge c$, we have $a \ge 4,\ b\ge 6$ and $c\ge 3$. Based on whether $m_1 = b + 1$ or $m_2 = b + 1$, we split each of the last two cases in (\ref{eq.q5t3}) and get 4 possibilities:
\begin{equation}
\label{eq.q5t3.last.two}
[q^5t^3]=
\begin{cases}
    b+2=m_2+1 & \text{ if } c = 3 \text{ and } b < a+1\ (m_2 = b + 1),\\
    b+2=m_1+1 &  \text{ if } c = 3 \text{ and } b > a+1\ (m_1 = b + 1), \\
    b=m_1-1 &  \text{ if } c > 3 \text{ and } b > a+c-2\ (m_1 = b + 1),\\
    b=m_2-1 & \text{ if } c > 3 \text{ and } b < a + c - 2\ (m_2 = b + 1).
\end{cases}
\end{equation}
Note that among $m_1 + 1, m_1 - 1, m_2 + 1$ and $m_2 - 1$, the value $m_1 + 1$ is strictly the largest and $m_2 - 1$ is strictly the smallest. This means if $[q^5t^3] = m_1 + 1$, then we must be in the second possibility in (\ref{eq.q5t3.last.two}) (that is, the fifth case in (\ref{eq.q5t3})), and if $[q^5t^3] = m_2 - 1$, then we must be in the fourth possibility in (\ref{eq.q5t3.last.two}) (that is, the sixth case in (\ref{eq.q5t3})). Furthermore, if we also have $m_1 \neq m_2 + 2$, then we have $m_1 - 1 \neq m_2 + 1$. This enables us to distinguish between the first and third possibilities in (\ref{eq.q5t3.last.two}), corresponding to the fifth and sixth cases in (\ref{eq.q5t3}) respectively.

To finish the proof, we may now assume that $[q^5t^3] = m_2 + 1 = m_1 - 1$. Recall the formula of \Cref{prop:q6t3}:
\[
[q^6t^3]=\binom{\bar{b}}{2}+2(\bar{a} + \bar{c})+(\bar{a}+\bar{b}+4)\delta_{c=3}+2\delta_{a=4} + 2\delta_{c=4}.
\]

For a given total number of vertices, $a + b + c = m_1 + m_2$, there is a unique caterpillar with 
\[
\text{$c=3$\quad and \quad}[q^5t^3]=b+2=m_2+1=m_1-1=a+c-2
\] 
that is of the form $(a,b,c) = (x+1,x,3)$. We argue that the coefficient $[q^6t^3]$ for this caterpillar is strictly larger than any other caterpillars with the same $m_1,m_2$ and $[q^5t^3]$ values. Note we can write any other such caterpillar in the following form, since we must have that $c > 3$ (and so $m_1 = b + 1$):
\[
(x+2-c, x+2, c)=(m_2+1-c,m_1-1,c),
\]
where $x:=\frac{1}{2}(m_1+m_2-4)$. Thus, we must be in the third possibility in (\ref{eq.q5t3.last.two}). 
Since $b = x = a - 1 \ge 6$, the caterpillar $(x+1,x,3)$ has
\begin{equation}
\label{eq:q6t3 max}
[q^6t^3]\psgn{(x+1,x,3)}{3} =\binom{x-1}{2}+2(x+2)+(2x+3),
\end{equation}
while the caterpillar $\alpha:=(x + 2 - c,x + 2,c)$ has
\[
[q^6t^3]\ps_3 Z_{C_{\alpha}}(\mathbf{x};t) = \binom{x+1}{2} + 2x + 2\delta_{x + 2 - c = 4} + 2\delta_{c=4} \le \binom{x+1}{2}+2x+4.
\]
Hence, it suffices to show 
\[\binom{x-1}{2}+2(x+2)+(2x+3)>\binom{x+1}{2}+2x+4,\]
which follows from
\begin{align*}
    \frac{(x-1)(x-2)}{2}+4x+7 &= \frac{x^2-3x+2+4x-2}{2}+2x+8\\
    &=\binom{x+1}{2}+2x+8.
\end{align*}
Therefore, if $[q^6t^3]$ is strictly less than the value in (\ref{eq:q6t3 max}), we conclude that $c>3$; otherwise $c = 3$. This concludes the proof.
\end{proof}
\begin{example}
There are caterpillars for which it is necessary to examine $[q^6t^3]$ in order to recover the composition. For example, the caterpillars $(7,6,3)$ and $(4,8,4)$ have the same multiset $\{\{9,7\}\}$ and coefficient $[q^5t^3]$. In fact, the coefficient $[q^6t^3]$ is the lowest degree coefficient where they disagree:
\begin{align*}
\psgn{(7,6,3)}{3}&= \left(\dots+ 41 q^{6} + 8 q^{5} + q^{4}\right) t^{3} + \left(q^{25} + q^{23} + q^{18} + q^{14} + q^{9} + q^{7}\right) t^{2},\\
\psgn{(4,8,4)}{3}&=\left(\dots + 37 q^{6} + 8 q^{5} + q^{4}\right) t^{3} + \left(q^{25} + q^{23} + q^{18} + q^{14} + q^{9} + q^{7}\right) t^{2}.
\end{align*}
\end{example}

\begin{proposition}
\label{prop:multiset vals}
Algorithm~\ref{alg:midends} recovers the values $b$ and $a+c$ from $\ps_3(Z_{(a, b, c)}(\mathbf{x}; t))$ for all caterpillars of type $(a,b,c)$ with $\min\{a,b,c\}\geq 3$.
\end{proposition}
\begin{proof}
The cases where $m= \min\{a,b,c\}$ equals 1 or 2 are handled in the previous two subsections. For $m\geq 3$, \Cref{lem:q5 recovers b} allows us to recover $b$ by distinguishing among the six possible cases for the value of $[q^5t^3]$. Then, comparing the value of $b$ with the multiset $\{\{b+1,a+c-1\}\}$, we can immediately obtain $a+c$.
\end{proof}
\begin{remark}
In Lemma~\ref{lem:q5 recovers b}, if we are in the 1st, 2nd or the 4th cases, we can recover the values of $a$ and $c$ as well. However, if we are in the 3rd, 5th or 6th cases, we need  more data to recover the caterpillar. For example, the caterpillars $(10,3,5), (9,3,6),$ and $(8,3,7)$ are all in the 3rd case, and they agree on their multiset consisting of $[q^3t^3],\ [q^4t^3],\ [q^5t^3]$ and $[q^6t^3]$.
\end{remark}

\begin{proposition}
\label{prop:q_m_plus_2}
Suppose $a,b,c\geq m$ for some $m\geq 3$. By Proposition~\ref{prop:multiset vals} the values of $a+c$ and $b$ are known. Define a function $F:\Z^2\to \Z$ by
\[F_m(\sigma,\tau)=\tau \delta_{\tau+1=m}+\sum_{j=0}^{\frac{m-2}{2}}\binom{\sigma}{j}\binom{\tau}{m-2-2j}+\sum_{j=0}^{\frac{m-3}{2}}\binom{\tau}{j}\binom{\sigma}{m-3-2j}.
\]
Then the number of values in $a,b,c$ that are equal to $m$ is 
\[
\frac{1}{2}\left([q^{m+2}t^3]- F_m(\bar{a}+\bar{c},\bar{b})\right).
\]      

\end{proposition}
\begin{proof}
    Subtracting $F_m(\bar{a} + \bar{c}, \bar{b})$ from the formula for $[q^{m+2}t^3]$ given in Appendix~\ref{proof.q_m_plus_2} reveals that
    \begin{equation}[q^{m+2}t^3] - F_m(\bar{a}+\bar{c},\bar{b})= 2(\delta_{a = m} + \delta_{b = m} + \delta_{c = m}) = : \chi(a,b,c;m) = \begin{cases}
        6&a=b=c=m,\\
        4&b=c=m,\ a>c,\\
        4&a=c=m,\ b>m,\\
        2&b=m,\ a,c>m,\\
         2&c=m,\ a,b>m,\\
        0&a,b,c>m.
    \end{cases}
    \end{equation}
        Since the quantities $\bar{a}+\bar{c}$ and $\bar{b}$ are known from \Cref{prop:multiset vals} when $a,b,c\geq m\geq 3$, we know the value of $F_m(\bar{a}+\bar{c},\bar{b})$ and thus we also know the quantity $\chi(a,b,c;m).$ Since $b$ is in particular known, we may conclude

         \[(a,b,c)=\begin{cases}
        (m,m,m)&\chi(a,b,c;m)=6,\\
        (a,m,m)&\chi(a,b,c;m)=4,\ b=m,\\
        (m,b,m)&\chi(a,b,c;m)=4,\ b>m,\\
         (a,m,c)&\chi(a,b,c;m)=2,\ b=m,\\
        (a,b,m)&\chi(a,b,c;m)=2,\ b>m,\\
        (a,b,c)&\chi(a,b,c;m)=0,
    \end{cases}\]
    where in the above, if we do not have $m$ in a coordinate position, then that coordinate is strictly larger than $m$. Thus, we have determined the locations of the $m$'s in $(a,b,c)$, if any.
\end{proof}

\begin{remark}
\label{rem.min}
Proposition~\ref{prop:q_m_plus_2} allows us to determine the value $\min\{a,b,c\}$ in the following way: initialize $m = 3$; if $[q^{m+2}t^3] - F_m(\bar{a}+\bar{c},\bar{b}) = 0$, set $m = m+1$ and repeat the process until the first time this value becomes positive. Once the value of $m$ is determined, we can further determine its multiplicity in $(a,b,c)$ as follows.
 
\begin{enumerate}
\item \label{case.triple.double.min} If $[q^{m+2}t^3] - F_m(\bar{a}+\bar{c},\bar{b})$ is equal to $6$ or $4$, the multiplicty of $m$ is 3 or 2, and we have recovered $a$ and $c$.
\item \label{case.single.min.not.b} If $[q^{m+2}t^3] - F_m(\bar{a}+\bar{c},\bar{b})$ is equal to $2$ and $b > m$, the multiplicity of $m$ is 1, and we have also recovered values $a$ and $c$.
\item \label{case.single.min.b}When $[q^{m+2}t^3] - F_m(\bar{a}+\bar{c},\bar{b}) = 2$ and $b = m$, the multiplicity of $m$ is 1 but we still don't know the values for $a$ and $c$ separately.
\end{enumerate}
The rest of the section will deal with the final case where the minimum value  $\min\{a,b,c\}$ occurs exactly once and at $b$. Recovering the caterpillars in this case requires us to look at a higher degree coefficient, $[q^{b+4}t^3]$.
\end{remark}

\begin{proposition}
\label{prop:q_m_plus_4}
Suppose $a \geq c > b \geq 3$. Define three functions: 
\begin{align}\label{eq:Gqm4}
    G(\bar{a}+\bar{c},\bar{b}) &= \binom{\bar{b}}{3} +\bar{b}\left(\bar{a} + \bar{c}\right)  + (\bar{a} + \bar{b}+\bar{c})
    +\sum_{j=1}^{\frac{\bar{b}+1}{2}}\binom{\bar{a}+\bar{c}}{j}\binom{\bar{b}}{\bar{b}+1-2j}+\sum_{j=0}^{\frac{\bar{b}}{2}}\binom{\bar{b}}{j}\binom{\bar{a}+\bar{c}}{\bar{b}-2j}
\end{align}

\begin{equation}\label{eq:psiqm4}
\psi(a,c;b)=
   \begin{cases}
  8b-2  &    a = c = b + 1,\\
  4b+2  &    a = b + 2, c = b + 1,\\
  a+3b-2 &  a > b + 2, c = b + 1,\\
   4    &   a = c = b + 2,\\
   2    &   a > b + 2 = c,\\
   0    &   a,c > b+2,
\end{cases}
\end{equation}
and 
\begin{equation}\label{eq:rhoqm4}
     \rho(a,c;b)= \bar{a}+\bar{c}+\binom{\bar{a}+\bar{b}}{2}+\binom{\bar{c}+\bar{b}}{2}.\end{equation}
Then 
\[
[q^{b+4}t^3]=G(\bar{a}+\bar{c},\bar{b})+\psi(a,c;b)+ \rho(a,c;b).
\]
\end{proposition}
The proof of Proposition~\ref{prop:q_m_plus_4} is in  Appendix~\ref{proof.q_m_plus_4}.

\begin{proposition}
\label{prop:recover from m3}
If $a,b,c \geq 3$, then we can recover the caterpillar $(a,b,c)$ from $\ps_3(Z_{(a,b,c)}(\mathbf{x};t))$.
\end{proposition}

\begin{proof}
First, we recover the values of $b$ and $a+c$ according to \Cref{prop:multiset vals}. Then, we follow Remark~\ref{rem.min} to get $\min\{a,b,c\}$ using Proposition~\ref{prop:q_m_plus_2}. If we are in the cases (\ref{case.triple.double.min}) or (\ref{case.single.min.not.b}) of \Cref{rem.min}, we have recovered the values of $a$ and $c$. We now discuss the case (\ref{case.single.min.b}), i.e, $a,c > b \ge 3$, using Proposition~\ref{prop:q_m_plus_4}.

First from \Cref{prop:q_m_plus_4}, we have 
\begin{equation}
\label{eq.m_plus_4_recover}   
\psi(a,c;b)+\binom{\bar{a}+\bar{b}}{2}+\binom{\bar{c}+\bar{b}}{2}=[q^{b+4}t^3]-G(\bar{a}+\bar{c},\bar{b})-\bar{a}-\bar{c}.
\end{equation}
We can compute the value, $R$, of the right-hand side of Equation~(\ref{eq.m_plus_4_recover}) as we know $[q^{b+4}t^3]$, $b$, and $a + c$. 

Next, we simplify the left-hand side of Equation~(\ref{eq.m_plus_4_recover}) in each case using the definition of $\psi$, and get
\[
R=
\begin{cases}
4\bar{b}^2+10\bar{b}+6&a=c=b+1,\\
4\bar{b}^2+8\bar{b}+7& a=b+2, c=b+1,\\
2\bar{b}^2+4\bar{b}+2+\bar{a}+\binom{\bar{a}+\bar{b}}{2}&a>b+2, c=b+1,\\
4\bar{b}^2+6\bar{b}+6&a=c=b+2,\\
2\bar{b^2}+3\bar{b}+3+\binom{\bar{a}+\bar{b}}{2}& a>b+2 = c,\\
\binom{\bar{a}+\bar{b}}{2}+\binom{\bar{c}+\bar{b}}{2}&a,c>b+2.
\end{cases}
\] 
Let $z=\bar{a}+\bar{c}$ and these quantities become
\begin{equation}
\label{eq.final.six}
R=
\begin{cases}
4\bar{b}^2+10\bar{b}+6&a=c=b+1,\\
4\bar{b}^2+8\bar{b}+7&a=b+2, c=b+1,\\
2\bar{b}^2+3\bar{b}+2+\binom{z}{2}&a>b+2, c=b+1,\\
4\bar{b}^2+6\bar{b}+6&a=c=b+2,\\
2\bar{b^2}+3\bar{b}+3+\binom{z-2}{2}& a>b+2 = c,\\
\binom{\bar{c}+\bar{b}}{2}+\binom{z-\bar{c}+\bar{b}}{2}&a,c>b+2.
\end{cases}
\end{equation}

To distinguish among the above six cases, we discuss in terms of the values $(a+c) + b$.
\begin{enumerate}
\item If $(a+c) + b = 3b+2$, then we must be in the first case in~\ref{eq.final.six} and we have recovered the caterpillar. 
\item If $(a+c) + b = 3b+3$, then we must be in the second case in~\ref{eq.final.six} and we have recovered the caterpillar. 
\item If $(a+c) + b = 3b+4$, we need to determine whether it is the third or the four case. Notice when $a > b+2$ and $c = b+1$, we have
\[
2\bar{b}^2+3\bar{b}+2+\binom{z}{2} \ge 2\bar{b}^2+3\bar{b}+2 + \binom{2\bar{b}+4}{2} = (4\bar{b}^2 + 6\bar{b} + 6) + 4\bar{b} + 2.
\]
Therefore, if $R > 4\bar{b}^2 + 6\bar{b} + 6$,
then we are in the third case. If $R = 4\bar{b}^2 + 6\bar{b} + 6$ then we are in the fourth case. In either case, we have recovered the caterpillar.
\item If $(a+c) + b \ge 3b + 5$, we necessarily have $a \ge b + 3 > b + 2$ and we need to determine if we are in the third, fifth or the sixth case in~\ref{eq.final.six}. Note that we have the following strict inequalities since $z = \bar{a} + \bar{c} \ge 2\bar{b} + 5 \ge 9$:
\begin{align*}
2\bar{b}^2+3\bar{b}+2+\binom{z}{2}
&= 2\bar{b^2}+3\bar{b}+3+\binom{z-2}{2} + 2(z-2)\\
&> 2\bar{b^2}+3\bar{b}+3+\binom{z-2}{2} \\&= \binom{(\bar{b}+2)+\bar{b}}{2}+\binom{z-(\bar{b}+2)+\bar{b}}{2}+2\\
&> \binom{\bar{c}+\bar{b}}{2}+\binom{z-\bar{c}+\bar{b}}{2}
\end{align*}
where the last inequality follows from Lemma~\ref{lem:antiunimodal binomials} with 
\[
(\bar{b}+2) + \bar{b} \le \bar{c} + \bar{b} \le \dfrac{\bar{a}+\bar{c}}{2} + \bar{b} =  \dfrac{z}{2} + \bar{b}
\]
when we necessarily have $c \ge b + 2$. Therefore it suffices to compare $R$ with the values 
\[
2\bar{b}^2+3\bar{b}+2+\binom{z}{2},\  2\bar{b^2}+3\bar{b}+3+\binom{z-2}{2}, \text{ and } \binom{\bar{c}+\bar{b}}{2}+\binom{z-\bar{c}+\bar{b}}{2}
\]
to determine if we are in the third, fifth or the six case of~\ref{eq.final.six}. If we are in the third or fifth case, we have recovered the caterpillar. If we are in the sixth case, we obtain a quadratic equation with one variable $\bar{c}$:
\[
R = \binom{\bar{c}+\bar{b}}{2}+\binom{z-\bar{c}+\bar{b}}{2}.
\]
Solving this equation results in two (possibly equal) roots and they are the values for $\bar{a}$ and $\bar{c}$ respectively.\qedhere
\end{enumerate}
\end{proof}

With this finally done, we have proven our main theorem.

\begin{proof}[Proof of Theorem~\ref{thm:MainTheorem}]
The $b=1$ case is provided by \Cref{prop:b1 recovery}. The $m=\min\{a,b,c\}=2$ case is given by \Cref{prop:recovering_from_m2}, and $m\geq 3$ by \Cref{prop:recover from m3}.
\end{proof}

\section{Discussion and future directions}\label{sec:future-directions}
We showed $\ps_3(Z_T(\mathbf{x};t))$ distinguishes some infinite families of trees, but this will not hold in general, even for caterpillars. In SageMath, one can compute the values of $\ps_3(Z_{C_{(3, 4, 2, 3)}})$ and $\ps_3(Z_{C_{(4, 2, 3, 3)}})$ to be 
\begin{align*}
   f&= t^2(q^{19} + q^{17} + q^{14} + q^{10} + q^7 + q^5)\\
    &\quad +t^3(2 q^{19} + 18 q^{18} + 71 q^{17} + 204 q^{16} + 438 q^{15} + 749 q^{14} + 1023 q^{13} + 1128 q^{12} + 1023 q^{11} \\
    &\quad + 749 q^{10} + 438 q^{9} + 204 q^{8} + 71 q^{7} + 18 q^{6} + 2 q^{5})
\end{align*}
However, we have been unable to find two nonisomorphic caterpillars with the same $\ps_4(Z_T(\mathbf{x};t))$.
\begin{conj}\label{conj:ps_4Z}
  $\ps_4(Z_T(\mathbf{x};t))$ distinguishes caterpillars.
\end{conj}
What is more, it seems this is likely to hold at $t=1$.
\begin{conj}\label{conj:ps_4X}
  $\ps_4(X_T(\mathbf{x}))$ distinguishes caterpillars.
\end{conj}

We have verified the above conjectures on caterpillars with at most $25$ vertices, as well as caterpillars with 5 internal vertices with 30, 35, 40, 45 and 50 total vertices. In fact, we have been unable to find two nonisomorphic trees with fewer than 26 vertices that have the same $\ps_4(X_T(\mathbf{x}))$. We also record $t$-versions of the Loehr--Warrington and Bajo--Beck--Vindas--Melendes conjecture. These are strictly weaker than their counterparts. 

\begin{conj}
    For each $n\geq 2$, $\ps_n(Z_T(\mathbf{x};t))$ distinguishes trees with at most $n$ vertices.
\end{conj}

\begin{conj}
    $\ps(Z_T(\mathbf{x};t))$ distinguishes trees.
\end{conj}  

\subsection*{Acknowledgments} This material is based on work supported by the National Science Foundation under Grant Number
DMS-1916439 while the authors were participating in the Mathematics Research Communities (MRC) 2024
Summer Conference at Beaver Hollow Conference Center in Java Center, New York. We extend our thanks
to the American Mathematical Society for their support. A. Bingham was partially supported by the grant DST/INT/RUS/RSF/P-41/2021 from the Department of Science \& Technology, Govt. of India, and FONDECYT-ANID grant 3250472. P. Commins was partially supported by an NSF graduate research fellowship and was in residence at the Simons Laufer Mathematical Sciences Institute for a portion of this work.  R. Orellana was partially supported by NSF grant DMS-2153998. 
This material is also based upon work supported by the National Science Foundation, while A. Bingham, P. Commins, R. Orellana and J. Pan were in residence at the ICERM semester program ``Categorification and Computation in Algebraic Combinatorics'' in Fall 2025. We thank Robert Angarone and Matthias Beck for helpful conversations. We also thank Jonathan Buck for sharing code.

\subsection*{Statement on AI use} AI was used to generate first drafts of the pseudocode for Algorithms \ref{alg:recover-abc} and \ref{alg:midends} and the TikZ code for \Cref{tab:ps2-class-example}. The code from Jonathan Buck was generated by AI models, which we used to check our Conjectures~\ref{conj:ps_4Z} and \ref{conj:ps_4X} for higher orders. All other code was produced by the authors, including the code used to make the conjectures.  We did not use any AI for exploration, brainstorming, computation, mathematical reasoning, proving, writing, proofreading, or in any other capacity.


\appendix
\section{Counting elongated bistar colorings}\label{sec:Append.A}
Here we prove some auxiliary lemmas that are used in \Cref{sec:bistars}. Recall that for the elongated bistar $(a,1^r,b)$, we have the leaf sets $A$, $B$, and spine vertices $v_A$ and $v_B$. We will compute certain coefficients $[q^st^3]\ps_3(Z_{(a,1^r,b)})$ by counting the contributions from each set of colorings where $v_A$ and $v_B$ are assigned fixed colors from among $\{\mb{0,1,2}\}$. These \textbf{endpoint colorings} will be denoted, for example, $\langle \mb{2},\mb{0}\rangle$ to indicate that $v_A$ is colored $\mb{2}$ and $v_B$ is colored $\mb{0}$.

\begin{lemma}\label{lem:r_by_2_plus_3}
    Let $(a,1^r,b)$ be an elongated bistar with $a\geq b\geq 2$. Then for all $r\geq 0$ even, we have 
    \[
    [q^{\floor{r/2}+3}t^3]\ps_3(Z_{(a,1^{r},b)})=\left({\floor{{r}/{2}}}+2\right)(\delta_{b=2}+\delta_{a=2})+r+2,
    \]
    and if $r$ is odd we have 
\[[q^{\floor{r/2}+3}t^3]\ps_3(Z_{(a,1^{r},b)})=\lfloor r/2\rfloor +2.\]
\end{lemma}
\begin{proof}
First we handle the case with $r$ even. We proceed by examining the contribution from each possible endpoint coloring, of which there are 9.
  \begin{itemize}
        \item $\langle \mathbf{0},\mathbf{0}\rangle$. Note that this endpoint coloring is only possible if $r>0$. We need to color all of the vertices from $A$ and $B$ with at least a $\mb{1}$, and a valid $r$-spine coloring has minimum weight $\lfloor r/2\rfloor+2$ for $r>0$.  
        But then since $\ab{A}\geq\ab{B}\geq 1$, any coloring will have total weight at least $r/2 + 4$, so this set of colorings contributes nothing to the coefficient.  
        \item $\langle \mathbf{1},\mathbf{0}\rangle$. An $r$-spine coloring has minimum weight $\lfloor r/2\rfloor$ by using only $\mb{0}$'s and $\mb{1}$'s, and all the leaves in $B$ must be colored at least $\mb{1}$. However if $\ab{B}>1$, then any proper coloring which includes a $\mb{2}$ would have weight at least $r/2+4$, so we only have a contribution if $b=2$. Supposing $b=2$, if the $r$-spine is assigned the minimum weight coloring then the leaf in $B$ must be colored $\mb{2}$ since if it were colored $\mb{1}$ it would be impossible to assign $\mb{2}$ to a leaf in $A$ without exceeding the total weight. If $r>0$,  we can also do an $r$-spine coloring of weight $\lfloor r/2\rfloor+1$ in one of $\lfloor r/2\rfloor$ ways (substituting a $\mb{2}$ for a $\mb{1}$ on the $r$-spine), and then color the leaf in $B$ as $\mb{1}$. Hence, these endpoint colorings contribute $(\lfloor r/2\rfloor +1)\delta_{b=2}$.
        \item $\langle \mathbf{2},\mathbf{0}\rangle$. A minimal $r$-spine coloring has weight $\lfloor r/2\rfloor$, so we can only produce a coloring of desired weight if $|B|=1$ in which case that leaf is assigned $\mb{1}$. Thus, the contribution is $\delta_{b=2}$.
        \item $\langle \mathbf{0},\mathbf{1}\rangle$. Symmetrizing the argument given for the case $\lan \mb{1},\mb{0}\ran$, we get a contribution of $(\lfloor r/2\rfloor+1)\delta_{a=2}$.
        \item $\langle \mathbf{0},\mathbf{2}\rangle$. Symmetrizing the argument given for the case $\lan \mb{2},\mb{0}\ran$, we get a contribution of $\delta_{a=2}$.
        \item $\langle \mathbf{1},\mathbf{1}\rangle$. Any $r$-spine coloring must use a $\mb{2}$. Choose where on the $r$-spine to put a $\mb{2}$. The rest of the coloring is uniquely determined as the pattern of alternating $\mb{0}$'s and $\mb{1}$'s, starting with $\mb{0}$. Thus, this contributes $r$.
        \item $\langle \mathbf{2},\mathbf{2}\rangle$. {Any $r$-spine coloring has weight at least $\lfloor r/2\rfloor$, so the total weight is already too large and} this contributes $0$.
        \item $\langle \mathbf{1},\mathbf{2}\rangle$. {The unique coloring which alternates $\mb{0}$'s and $\mb{1}$'s from left to right along the $r$-spine (that is, with the spine vertex adjacent to $v_A$ colored  $\mb{0}$) and colors all leaves $\mb{0}$ is the only 3-coloring with the appropriate weight, so} this contributes $1$.
        \item $\langle \mathbf{2},\mathbf{1}\rangle$. By symmetry with the previous case, this also contributes $1$.
    \end{itemize}

   Now consider the case where $r$ is odd. The arguments are identical to the above with the following exceptions.

   \begin{itemize}
    \item $\langle \mathbf{0},\mathbf{0}\rangle$. A valid $r$-spine coloring has minimum weight $\lfloor r/2\rfloor+1$. Since vertices in $A$ and $B$ need to be colored at least 1, and we need to use at least one 2, the total minimum weight is $(\lfloor r/2\rfloor+1) + 3$. Hence the contribution is 0.
       \item $\langle \mathbf{1},\mathbf{0}\rangle$. Now an $r$-spine coloring has minimum weight $\lfloor r/2\rfloor+2$ since a $\mb{2}$ must be used. This gives weight $\floor{r/2}+3$ before coloring the leaves, at least one of which must be non-$\mb{0}$, so there is no contribution.
       \item $\langle \mathbf{2},\mathbf{0}\rangle$. Now an $r$-spine coloring has minimum weight $\lfloor r/2\rfloor+1$ by alternating $\mb{0}$ and $\mb{1}$. But this gives weight $\floor{r/2}+3$ before coloring the leaves, at least one of which must be non-$\mb{0}$, so there is no contribution.
       \item $\langle \mathbf{0},\mathbf{1}\rangle$. Also no contribution, by symmetry with $\langle \mathbf{1},\mathbf{0}\rangle$.
       \item  $\langle \mathbf{0},\mathbf{2}\rangle$. Also no contribution, by symmetry with $\langle \mathbf{2},\mathbf{0}\rangle$.
     \item $\langle \mathbf{1},\mathbf{1}\rangle$. {The minimum $r$-spine coloring has weight $\floor{r/2}$ but uses no $\mb{2}$, thus yields weight $\floor{r/2}+2$ before coloring the leaves or placing a $\mb{2}$. Since we cannot use a $\mb{2}$ here without going overweight, we must use a $\mb{2}$ on the $r$-spine. However a $\mb{2}$ can only replace a $\mb{1}$ in the minimal $r$-spine coloring, so there are $\floor{r/2}$ choices. Such a coloring already has total weight $\lfloor r/2\rfloor+3$,} so this contributes $\floor{r/2}$.
   \end{itemize}
   Adding the contributions in each case gives the claimed result.
\end{proof}

\begin{lemma}\label{lem:elongated-bistar-c1-odd}
    Let $a \geq b \geq 2$ and let $r \geq 1$ be odd. Then, the coefficient $C_1 := [q^{\lfloor\frac{r}{2}\rfloor + 4}t^3] \ps_3(Z_{(a, 1^r, b)}(\mathbf{x}; t))$ is \[C_1 = 2(\bar{a}+\bar{b}+\lfloor r/2\rfloor)+\binom{\lfloor r/2+1\rfloor}{2}+1+\delta_{r\geq 2}+\delta_{a=b=2}(\lfloor r/2\rfloor+3+r+1)+\delta_{b=2}(r+1).\]

    As a result, if 
    \begin{equation}\label{eqn:deciding-b=2}
        C_1 = 2(\bar{a} + \bar{b}  + \lfloor r / 2\rfloor) + \binom{\lfloor r / 2 + 1\rfloor}{2} + \delta_{r \geq 2} + r + 2 + \delta_{a=2}(\lfloor r/2\rfloor+3+r+1)
    \end{equation}
     then $b = 2$. Otherwise, $b > 2.$
\end{lemma}
\begin{proof} 
  We investigate by considering the possible colorings of the endpoints $v_A$ and $v_B$. 
    \begin{itemize}
        \item $\langle \mathbf{0},\mathbf{0}\rangle$. Each of the $\overline{a} + \overline{b}$ vertices in $A \cup B$
        as well as $\lfloor r/2\rfloor + 1$ of the $r$-spine vertices must be colored non-$\mathbf{0}$, for a total of at least $\lfloor r/2\rfloor + a + b  -1 \geq \lfloor r/2\rfloor + 3$. Since at least one of these vertices must be colored $\bf 2$, the chromatic sum can only be $\lfloor r/2\rfloor + 4$ if $a = b = 2$. Further, in this case, we can pick any of the $\lfloor r/2\rfloor + 3$ non-$\bf 0$ vertices to be colored with a $\bf 2.$ Hence, in total,  this endpoint coloring contributes $\delta_{a=b=2}(\lfloor r/2\rfloor+3)$.
        \item $\langle \mathbf{1},\mathbf{0}\rangle$. In addition to $v_A$, vertices in $B$ must be colored non-$\bf 0$, as must $\lfloor r / 2\rfloor + 1$ of the $r$-spine vertices, for a total of at least $\lfloor r / 2\rfloor + b + 1 \geq \lfloor r / 2\rfloor + 3$ non-$\bf 0$ vertices. Since at least one of these must be colored $\bf2$, the coloring only contributes to the coefficient if $b = 2$. So, assume $b = 2$. Observe that there needs to be two adjacent non-$\bf 0$ vertices on the spine, so the single $\bf 2$ must be on the spine. It is straightforward to check there are $r$ ways to color the spine appropriately. In total, this endpoint coloring contributes $\delta_{b = 2}r.$ 
        \item $\langle \mathbf{2},\mathbf{0}\rangle$. Since the vertices in $B$ and $\lfloor r/ 2\rfloor + 1$ of the $r$-spine vertices must be non-$\bf 0$, the chromatic sum of any contributing coloring is at least $2 + \lfloor r / 2\rfloor + b.$ So, the only way to obtain the correct chromatic sum is if $b = 2$, where we color the single vertex in $B$ with $\bf1$ and the $r$-spine vertices $\bf 1, \bf 0, 1, \ldots, 0, 1$. Thus, the total contribution is $\delta_{b=2}.$
        \item $\langle \mathbf{0},\mathbf{1}\rangle$. Similar to the cases above, it is straightforward to check that the only contributing proper colorings must have $a = 2$ (in which case $b = 2$ too), with the single vertex in $A$ colored $\bf 1$ and the $r$-spine vertices colored with one $\bf 2$ and with $\lfloor r/2\rfloor$ $\bf 1'$s. There are $r$ ways to to distribute these colors along the inner spine vertices. Hence, the total contribution is $\delta_{a=b=2}r$.
        \item $\langle \mathbf{0},\mathbf{2}\rangle$. There must be at least $\overline{a} + \lfloor r / 2 \rfloor + 1$ vertices colored non-$\bf 0$ in addition to $v_B$. Hence, the only possible contributing colorings arise when $a = 2$ (in which case $b = 2$ too), and we color the single vertex in $A$ with $\bf 1$ and the $r$-spine vertices $\bf 1, 0, 1, 0, \ldots, 1$. 
        The total contribution is thus $\delta_{a=b=2}$.
        \item $\langle \mathbf{1},\mathbf{1}\rangle$. At least $\lfloor r / 2\rfloor$ of the $r$-spine vertices must be colored non-$\bf 0$, and hence at most two of them can be colored with a $\bf 2$. We divide into subcases based on whether zero, one, or two of the $r$-spine vertices are colored $\bf 2.$ If \textit{zero} are colored $\bf 2$, we have one way to color the $r$-spine vertices, and $\overline{a} + \overline{b}$ ways to place a $\bf2$ in $A \cup B$. If exactly  \textit{one} is colored $\bf 2$, see carefully that it must be at one of the ``odd'' $r$-spine vertices, of which we have $\lfloor r / 2\rfloor + 1$; the locations of the $\bf 1$'s are then uniquely determined. If \textit{two} are colored $\bf 2$, they must occur at ``even'' $r$-spine vertices, and so there are $\binom{\lfloor r / 2\rfloor}{2}$ ways to place them, and the locations of the remaining $\bf 1$'s are then again uniquely determined. Hence, in total, the contribution is $\bar{a}+\bar{b}+\lfloor r/2  \rfloor + 1 + \binom{\lfloor r/2\rfloor}{2} = \bar{a}+\bar{b}+\binom{\lfloor r/2+1\rfloor}{2}+1.$
        \item $\langle \mathbf{2},\mathbf{2}\rangle$ Since the minimum $r$-spine coloring can only contribute $\lfloor r/2\rfloor $, we can have at most one coloring. However if $r=1$, then this coloring would not have any $\bf{1}$'s in it, so it is invalid. Hence, the contribution is $\delta_{r\geq 2}.$
        \item $\langle \mathbf{1},\mathbf{2}\rangle$. The minimum $r$-spine coloring is weight $\lfloor r/2\rfloor$. Hence, accounting for the endpoints, we currently have weight $\lfloor r/2\rfloor+3$, so we can only use one more $\bf1$ which must go in $B$. Or, if we don't require the minimum $r$-spine coloring, we can swap one of the $\bf1$'s on the $r$-spine to a $\bf 2$ (for which there are $\lfloor r/2\rfloor$ choices) and color all the vertices in $B$ with $\bf 0$. Hence, we have a total contribution of $\bar{b}+\lfloor r/2\rfloor$.
        \item $\langle \mathbf{2},\mathbf{1}\rangle$. By a symmetric argument, this coloring contributes $\bar{a}+\lfloor r/2\rfloor$.
    \end{itemize}
    Summing the contributions together proves the claim.
    \end{proof}

The following lemma will be used to recover $a$ and $b$ in the case where $r$ is known to be odd. 

    \begin{lemma} \label{lem:recovering-a-b-r-odd}
    For $a\geq b>2$ and $r\geq 1$ odd, there is a function $F:\Z^2\to \Z$ such that
    \[[q^{\lfloor r/2\rfloor + 5}t^3]=\binom{\bar{a}}{2}+\binom{\bar{b}}{2}+F(\bar{a}+\bar{b},r)+(r+1)(\delta_{b=3}+\delta_{a=3}).\]
    Furthermore, if $r$ and $a+b$ are known, $a$ and $b$ are recoverable in this situation.
\end{lemma}
\begin{proof}
We argue again by considering each possible endpoint coloring.
  
    \begin{itemize}
        \item $\langle \mathbf{0},\mathbf{0}\rangle$: We need to color all of the leaves in $A \cup B$ with positive numbers, giving a weight of at least 4. The minimum $r$-spine coloring, created by alternating between \textbf{1} and \textbf{0}, gives a weight of $\lfloor r/2\rfloor + 1$ since the endpoints are both $\textbf{0}$ and $r$ is odd. Thus, the total weight is at least $\lfloor r/2\rfloor + 5$. However, the unique coloring of weight $\lfloor r/2\rfloor + 5$ is only a 2-coloring, and therefore the contribution is $0$.
        
        \item $\langle \mathbf{1},\mathbf{0}\rangle$: We need to color all of the leaves in $B$ with positive numbers contributing a weight of at least $\bar{b}$. The $r$-spine needs a coloring of weight at least $\lfloor r/2\rfloor+2$ since the endpoints are $\mb{1}$ and $\mb{0}$ and $r$ is odd. Note that this $r$-spine coloring must use a $\mb{2}$ somewhere. Then for $b>2$, we only have a valid coloring if $b=3$ in which case both leaves in $B$ receive color $\mb{1}$, and the coloring is determined by which spine vertex is assigned $\mb{2}$. The remaining spine vertices' colors are then determined in an alternating pattern of $\mb{0}$ and $\mb{1}$. Thus, this contributes $r\delta_{b=3}$.        
        \item $\langle \mathbf{2},\mathbf{0}\rangle$: A minimal weight $r$-spine coloring has weight $\lfloor r/2\rfloor +1$ and is unique. So we get a contribution only when $b=3$, yielding $\delta_{b=3}.$
        \item $\langle \mathbf{0},\mathbf{1}\rangle$: By a symmetric argument with the $\lan \mb{1},\mb{0}\ran$ case, we have $r\delta_{a=3}$.
        \item $\langle \mathbf{0},\mathbf{2}\rangle$: By a symmetric argument with the $\lan\mb{2},\mb{0}\ran$ case, we have $\delta_{a=3}$.
        \item $\langle \mathbf{1},\mathbf{1}\rangle$: We count by the weight of $r$-spine colorings. 
        
        A minimal $r$-spine coloring has weight $\floor{r/2}$ and is unique. With this coloring, we would need to choose one leaf in $A\cup B$ to color $\mb{2}$ and another leaf to color $\mb{1}$. But this is impossible since $v_A$ and $v_B$ are colored $\mb{1}$, so there is no contribution.

        There are $\floor{r/2}$ valid $r$-spine colorings with weight $\floor{r/2}+1$, each obtained by replacing a $\mb{1}$ by a $\mb{2}$ in the minimal weight coloring. Once this spine coloring is fixed, we choose a leaf in $A\cup B$ to color $\mb{2}$, so the total contribution here is $\floor{r/2}(\bar{a}+\bar{b})$.
        
        If we were to choose an $r$-spine coloring of weight $\floor{r/2}+2$, then there will be a remaining weight of 1 to put on the leaves. But this is impossible since $v_A$ and $v_B$ are colored $\mb{1}$, so no $r$-spine colorings contribute with this weight.

        Finally, we count $r$-spine colorings of weight $\floor{r/2}+3$. Such a coloring must use at least two $\mb{2}$'s. Note that any $r$-spine coloring must also have no more than three $\mb{2}$'s since any $r$-spine coloring must use at least $\floor{r/2}$ non-$\mb{0}$ colors, so more than three $\mb{2}$'s would give $r$-spine weight more than $\floor{r/2}+3$. Thus, we have two subcases here.
        
        If an $r$-spine coloring of weight $\floor{r/2}+3$ uses exactly two $\mb{2}$'s, then there must be $\floor{r/2}-1$ $\mb{1}$'s. Such a coloring is nearly determined by choosing any two non-adjacent $r$-spine vertices to color $\mb{2}$ and allowing the remaining vertices to alternate between $\mb{0}$ and $\mb{1}$. However, if there are an even number of spine vertices between the $\mb{2}$'s, there are two possible colorings of these in-between vertices. Since $r$ is odd, one may verify that there are $\floor{r/2}+2\binom{\floor{r/2}}{2}$ colorings with an odd number of spine vertices between those colored $\mb{2}$, and there are $2\binom{\floor{r/2}}{2}$ colorings with an even number of vertices in between. Thus we have a contribution of 
        \[\floor{r/2}+2\binom{\floor{r/2}}{2}+2(2\binom{\floor{r/2}}{2})=\floor{r/2}+6\binom{\floor{r/2}}{2}\]
        from these colorings.

        Now suppose the $r$-spine coloring uses three $\mb{2}$'s. Since we must use $\floor{r/2}-3$ $\mb{1}$'s, the coloring is determined by choosing three of the $\mb{1}$'s from the minimal weight coloring to change to $\mb{2}$'s. Thus the contribution here is $\binom{\floor{r/2}}{3}$.

        The total contribution from this endpoint coloring is then
        \[g_{11}(\bar{a}+\bar{b},r):=\floor{r/2}(\bar{a}+\bar{b})+\floor{r/2}+6\binom{\floor{r/2}}{2}+\binom{\floor{r/2}}{3}.
        \]

        \item $\langle \mathbf{2},\mathbf{2}\rangle$:  A minimal $r$-spine coloring again has weight $\floor{r/2}$ and is unique. Thus, with this spine coloring, we have a choice of leaf in $A$ or $B$ to color $\mb{1}$, contributing $\bar{a}+\bar{b}$.
        
        On the other hand, we may choose an $r$-spine coloring of weight $\floor{r/2}+1$. There is always one consisting of only $\mb{1}$'s and $\mb{0}$'s. When $r>3$, there are others obtained by replacing a $\mb{1}$ by a $\mb{2}$ in the $r$-spine coloring of weight $\floor{r/2}$. When $r=3$, such a coloring is invalid since it uses no $\mb{1}$. Thus the total contribution is \[g_{22}(\bar{a}+\bar{b},r):=\bar{a}+\bar{b}+1+\floor{r/2}-\delta_{r=3}.\]
        \item $\langle \mathbf{1},\mathbf{2}\rangle$: We count by $r$-spine weight again.
        
        A minimal $r$-spine coloring has weight $\lfloor r/2 \rfloor$ and is unique. We need an additional weight of 2 which is obtained by coloring leaves non-$\mb{0}$ in one of
        \[\binom{\bar{a}}{1}+\binom{\bar{b}}{2}\]
        ways.

        One may create a $r$-spine coloring of weight $\lfloor r/2\rfloor+1$ by using the unique minimal $r$-spine coloring and replacing a \textbf{1} with \textbf{2}, performed in one of $\lfloor r/2\rfloor$ ways. We then need an additional weight of 1, obtained in one of $\binom{\bar{b}}{1}$ ways.

        Finally, one may color the $r$-spine to give weight $\lfloor r/2\rfloor+2$. Such a coloring can use one or two $\mb{2}$'s on the $r$-spine. We consider these subcases separately.
        
        If the $r$-spine contains one $\mb{2}$, this $\mb{2}$ can be placed on the $r$-spine in any of $r-1$ positions (avoiding the vertex adjacent to $v_B$). The rest of the $r$-spine coloring is determined uniquely if there are an odd number of vertices between the $\mb{2}$ and $v_B$, and if there are an even number of vertices we have two possible colorings. As there are $\floor{r/2}$ vertices of the $r$-spine with this latter property, the contribution is 
        \[r-1+\floor{r/2}=3\floor{r/2}.\]

        If two $\mb{2}$'s appear, then there must be $\floor{r/2}-2$ $\mb{1}$'s on the $r$-spine. This means that such a coloring is obtained from the minimal weight coloring by choosing two $\mb{1}$'s to replace by $\mb{2}$'s. This contributes $\binom{\floor{r/2}}{2}$.
        
        Adding these subcases, we have a total contribution of
            \[\bar{a}+\binom{\bar{b}}{2}+\floor{r/2}\bar{b}+3\floor{r/2}+\binom{\floor{r/2}}{2}\]
            from this endpoint coloring.

        \item $\langle \mathbf{2},\mathbf{1}\rangle$: By symmetry with the $\lan \mb{1},\mb{2}\ran$ case, this contributes
        \[\bar{b}+\binom{\bar{a}}{2}+\floor{r/2}\bar{a}+3\floor{r/2}+\binom{\floor{r/2}}{2}.\]
       
    \end{itemize}

     Now set 
        \begin{align}\label{eqn:F-def}
F(s,r)&:=(\floor{r/2}+1)s+6\floor{r/2}+2\binom{\floor{r/2}}{2}+g_{11}(s,r)+g_{22}(s,r).\\
        &=2s + \lfloor r / 2 \rfloor \left(2s + 8\right) + 8 \binom{\lfloor r/ 2\rfloor}{2} + \binom{\lfloor r/ 2\rfloor}{3} +1 - \delta_{r = 3}.
        \end{align}
        Then, adding the contributions above we have that
        \[[q^{\lfloor r/2\rfloor + 5}t^3]=\binom{\bar{a}}{2}+\binom{\bar{b}}{2}+F(\bar{a}+\bar{b},r)+(r+1)(\delta_{b=3} +\delta_{a=3}),\]
        proving the first claim.
        
Next, given the values of $a+b$ and $r$, using $F$ one computes the quantity
\begin{equation}\label{eq:q_r_by_2_plus_5}
    [q^{\lfloor r/2\rfloor + 5}t^3]-F(\bar{a}+\bar{b},r)=\begin{cases}
            2+2(r+1)&a=b=3\\
            \binom{\bar{a}}{2} +\binom{{2}}{2}+r+1 &a>b=3\\ 
            \binom{\bar{a}}{2}+\binom{\bar{b}}{2}&a\geq b>3    
        \end{cases}
\end{equation}
Since the quantity ${a}+{b}$ is known, and we assume $a\geq b\geq 3$, we may determine that $a=b=3$ just from $a+b=6$. It remains to distinguish between the last two subcases for fixed $y:=\bar{a}+\bar{b}$. However, we can apply \Cref{lem:antiunimodal binomials} to observe that 
\[\binom{{y-2}}{2} +\binom{2}{2}+r+1>\binom{{y-2}}{2} +\binom{2}{2}>\binom{{y-\bar{b}}}{2}+\binom{\bar{b}}{2}\]
for  any $\bar{b}\geq 3$. Thus, by comparing the left-hand side of \eqref{eq:q_r_by_2_plus_5} to $\binom{a+b-4}{2}+1 = \binom{y - 2}{2} + \binom{2}{2}$, we can determine whether or not $b=3$. 

In case $b=3$, $a$ is immediately recoverable from $a+b$. Otherwise, the known value of $y$ and 
\[ \binom{\bar{a}}{2}+\binom{\bar{b}}{2}=\binom{y-\bar{b}}{2}+\binom{\bar{b}}{2}
\]
gives us a quadratic equation in $\bar{b}$ whose solutions will be $\bar{b}$ and $\bar{a}$.
\end{proof}

\section{Counting colorings for caterpillars with 3 internal vertices}\label{appendix:ps-qrt3}

\addtocontents{toc}{\protect\setcounter{tocdepth}{1}} 

The formulas used in Section~\ref{sec:3-cat} come from enumerating all possible proper colorings of a caterpillar of type $(a,b,c)$ and given weight using the three colors $\mb{0, 1, 2}$. Throughout this appendix, we will write $[q^rt^3]\psgn{(a,b,c)}{3}$ as just $[q^rt^3]$ for brevity. Recall $[q^rt^s]$ counts the number of proper colorings of the caterpillar $(a,b,c)$ of weight $r$ that use exactly $s$ colors.

Throughout the proofs in this appendix, when referring to the caterpillar $(a,b,c)$ we let $A, B, C$ denote the subset of leaves in the leaf components of order $a,b,c$, respectively. Thus, $|A| = a-1 \ge 1$, $|B| = b-1 \ge 0$, $|C| = c-1\ge 1$.
We will also refer to the colors assigned to the spine vertices in the caterpillar by $\lan \mathbf{k_a}, \mathbf{k_b}, \mathbf{k_c}\ran$ where $\mathbf{k_a}$ is the color assigned to the spine vertex attached to the leaves in $A$, etc. In general, there are $3\cdot 2+3!=12$ spine colorings:

\begin{multicols}{4}
    \begin{enumerate}
     \item $\mb{\lan 1,2,1\ran}$
        \item $\mb{\lan 2,1,2\ran}$
        
         \item $\mb{\lan 2,0,2\ran}$
           \item $\mb{\lan 0,2,0 \ran}$
                   \item $\mb{\lan 1,0,1\ran}$

        \item $\mb{\lan 0,1,0 \ran}$
       
        \item $\mb{\lan 1,0,2\ran}$
         \item $\mb{\lan 2,0,1\ran}$
 \item $\mb{\lan 1,2,0\ran}$
          \item $\mb{\lan 2,1,0\ran}$
        \item $\mb{\lan 0,2,1\ran}$
        \item $\mb{\lan 0,1,2\ran}$
    \end{enumerate}
\end{multicols}

These 12 cases need to be addressed to compute $[q^rt^3].$ We will make extensive use of the Kronecker delta function $\delta_P$, which takes value 1 if proposition $P$ holds and 0 otherwise. Finally, recall that we define $\Bar{a}=a-1$, $\Bar{b}=b-1$ and $\Bar{c}=c-1$, and assume that $a \geq c\geq 2$ and $b\geq 1$ unless otherwise stated. 

\subsection{Proof of Lemma~\ref{lem:coeff-of-q3-t3}}\label{appendix:coeff-of-q3-t3}
Lemma~\ref{lem:coeff-of-q3-t3} states
  \[ [q^3t^3]=\begin{cases}
  2&\text{if $b=1$}\\
   0 &\text{if $b>1$.}
   \end{cases}\]

\begin{proof}[Proof of Lemma~\ref{lem:coeff-of-q3-t3}]
    In order to obtain a chromatic sum of $3$ while using exactly three colors, we must use the colors \textbf{1} and \textbf{2} exactly once. Thus, spine colorings $\mb{\lan 1,2,1\ran}$,  $\mb{\lan 2,1,2\ran}$, $\mb{\lan 2,0,2\ran}$ and $\mb{\lan 1,0,1\ran}$ are not possible.  We examine the contribution of the remaining possible spine colorings to the coefficient $[q^3t^3]$ below.

    \begin{itemize}
          \item $\mb{\lan 0,2,0 \ran}$ and $\mb{\lan 0,1,0 \ran}$ contribute 0. Since all vertices in $A\cup C$ cannot be colored \textbf{0} and $|A \cup C| \geq 2$, these spine colorings produce a chromatic sum strictly greater than 3.
       
        \item $\mb{\lan 0,2,1\ran}$ and $\mb{\lan 0,1,2\ran}$ contribute 0. Since all vertices in $A$ cannot be colored \textbf{0} and $|A| \geq 1$, this spine coloring produces a chromatic sum strictly greater than 3.
        \item $\mb{\lan 2,1,0\ran}$ and $\mb{\lan 1,2,0\ran}$ contribute 0 for a similar reason.
         \item $\mb{\lan 2,0,1\ran}$ and $\mb{\lan 1,0,2\ran}$ each contribute $\delta_{b=1}$. All leaves in $B$ must be colored \textbf{1} or \textbf{2}, and so these spine colorings only produce a chromatic sum of 3 if $B = \varnothing$, in which case $b=1$.
    \end{itemize}
      Thus, in general we have $[q^3t^3]=2\delta_{b=1}$, and so the claim follows. 
\end{proof}

\subsection{Proof of Lemma~\ref{cor:coeff-t3-q4-list}}\label{appendix:coeff-of-q4-t3}
Lemma~\ref{cor:coeff-t3-q4-list} states that if $[q^3t^3] = 0$ (meaning $b>1$), then 
     \begin{align*}
         [q^4t^3] = \begin{cases}
             11 & \text{ if } a = b = c = 2\\
             8 & \text{ if } a = c =  2, b \geq 3\\
             6 & \text{ if } b = c = 2, a \geq 3\\
             4 & \text{ if } b = 2, a \geq 3, c \geq 3\\
             3 & \text{ if } c = 2, a \geq 3, b \geq 3\\
             1 & \text{ if } a, b, c \geq 3.
         \end{cases}
     \end{align*}

\begin{proof}[Proof of Lemma~\ref{cor:coeff-t3-q4-list}]

It suffices to prove the general formula
\[
 [q^4t^3] = 1 + 2\delta_{c = 2} + 3\delta_{b = 2} + 5 \delta_{a = 2}\delta_{c = 2}.
 \]

    In order to obtain chromatic sum $4$ while using exactly three colors, we must use \textbf{2} exactly once and \textbf{1} exactly twice. Thus, spine colorings $\mb{\lan 2,1,2\ran}$ and $\mb{\lan 2,0,2\ran}$ are not possible.  We examine the contribution of the remaining possible spine colorings to the coefficient $[q^4t^3]$ below. Note that by assumption,  $|B| \geq 1$. 
    \begin{itemize}
        \item $\mb{\lan 1,2,1\ran}$ contributes $1$. There is only one way to color the leaves to contribute to this coefficient since all leaves must be colored \textbf{0}.
        
        \item $\mb{\lan 0,2,0\ran}$ contributes $\delta_{a = 2}\delta_{c = 2}$. Two leaves must be colored \textbf{1} and the other leaves must be colored \textbf{0}. Since $|A|, |C| \geq 1$, the only way such a coloring arises is if $|A| = |C| = 1$, and the two leaves in $A \cup C$ are colored \textbf{1}.
        
        \item $\mb{\lan 1,0,1\ran}$ contributes $\delta_{b = 2}$. One leaf must be colored \textbf{2} and the other leaves must be colored \textbf{0}. Such a coloring is possible only if $|B| = 1$ and the unique leaf in $B$ is colored \textbf{2}. 

        \item $\mb{\lan 0,1,0\ran}$ contributes $2 \delta_{a= 2}\delta_{c = 2}$. We can only have two leaves not colored \textbf{0}. Since $|A \cup C| \geq 2$, the only way such a coloring arises is if $|A \cup C| = 2$. The coefficient of $2$ comes from the two ways to assign the leaves in $A\cup C$ the colors \textbf{1}, \textbf{2}.
    \end{itemize}
    The remaining spine colorings use three distinct colors, which means exactly one leaf must be colored \textbf{1} and the other leaves must be colored \textbf{0}. Such a coloring is possible only if there is a unique leaf in the leaf component containing the spine vertex colored \textbf{0}. Therefore, $\mb{\lan 1,0,2\ran}$ and $\mb{\lan 2,0,1\ran}$ each contribute $\delta_{b = 2}$, $\mb{\lan 1,2,0\ran}$ and $\mb{\lan 2,1,0\ran}$ each contribute $\delta_{c = 2}$, and $\mb{\lan 0,2,1\ran}$ and  $\mb{\lan 0,1,2\ran}$ each contribute $\delta_{a = 2}\delta_{c = 2}$. 
\end{proof}

\subsection{Proof of Proposition~\ref{prop:unified_q5t3}}\label{appendix:coeff-of-q5-t3} Proposition~\ref{prop:unified_q5t3} states
\begin{align*}
    [q^5t^3] = (\bar{b}+1) &+ \left( 2(\bar{a}+\bar{c}) + \binom{\bar{a}}{2} + \binom{\bar{c}}{2}\right)\delta_{b=1} + (2\bar{a}+2\bar{c}+3)\delta_{b=2} + 4\delta_{b=3} \\
    &+3(\bar{b}+2)\delta_{a=2} + (\bar{a}+\bar{b}+2)\delta_{c=2} + 4\delta_{a=3}\delta_{c=2} + 2\delta_{a=3} + 2\delta_{c=3}.
\end{align*}

\begin{proof}[Proof of Proposition~\ref{prop:unified_q5t3}] 
In order to obtain a weight $5$ proper coloring while using exactly three colors, we can have $5 = \mb{2} + \mb{2} + \mb{1}$ or $5 = \mb{2} + \mb{1} + \mb{1} + \mb{1}$. We examine the
contribution of each possible spine coloring to the coefficient $[q^5t^3]$ below.

\begin{itemize}
    \item $\mathbf{\langle 1,2,1 \rangle}$ contributes $\bar{b}$. We must create a $\mb{2} + \mb{1} + \mb{1} + \mb{1}$ coloring. We choose one leaf from $B$ to color \textbf{1}.

    \item $\mathbf{\langle 2,1,2 \rangle}$ contributes 1. We must create a $\mb{2} + \mb{2} + \mb{1}$ coloring. All of the leaves must be colored \textbf{0}. 

    \item $\mathbf{\langle 2,0,2 \rangle}$ contributes $(\bar{a}+\bar{c})\delta_{b=1} + \delta_{b=2}$. We must create a $\mb{2} + \mb{2} + \mb{1}$ coloring, and no leaf in $B$ can be colored \textbf{0}. If $B$ is empty, then we select one leaf from $A \cup C$ to color \textbf{1}, giving $(\bar{a}+\bar{c})\delta_{b=1}$. If $|B|=1$, then we must have the leaf in $B$ colored \textbf{1}, giving $\delta_{b=2}$.

    \item $\mathbf{\langle 0, 2, 0\rangle}$ contributes $\bar{b}\delta_{a=2} + \delta_{a=3}\delta_{c=2} + 2\delta_{a=2}$. To create a $\mb{2} + \mb{1} + \mb{1} + \mb{1}$ coloring with no leaves in $A\cup C$ colored \textbf{0}, all of the leaves in $A\cup C$ must be colored \textbf{1}. If $|A| = |C| = 1$, then we select a leaf from $B$ to color \textbf{1} as well, giving $\bar{b}\delta_{a=2}$. If $|A| =2$ and $|C| = 1$, then the three leaves in $A\cup C$ are the only leaves colored \textbf{1}, giving $\delta_{a=3}\delta_{c=2}$. To create a $\mb{2} + \mb{2} + \mb{1}$ coloring with no leaves in $A \cup C$ colored \textbf{0}, we must have $|A| = |C| = 1$. Hence there are two choices for which leaf in $A\cup C$ is colored \textbf{2}, giving $2\delta_{a=2}$.

    \item $\mathbf{\langle 1, 0, 1\rangle}$ contributes $(\bar{a}+\bar{c})\delta_{b=2} + 2\delta_{b=3}$. We must create a $\mb{2} + \mb{1} + \mb{1} + \mb{1}$ coloring, and no leaf in $B$ can be colored \textbf{0}. If $B$ is empty, this coloring is impossible. If $|B| = 1$, then the leaf in $B$ must be colored \textbf{1} and we can select one leaf from $A\cup C$ to color \textbf{2}, giving $(\bar{a}+\bar{c})\delta_{b=2}$. If $|B| = 2$, then we have two choices for which leaf is colored \textbf{2}, giving $2\delta_{b=3}$.

    \item $\mathbf{\langle 0, 1, 0\rangle}$ contributes $\bar{b}\delta_{a=2} + 3\delta_{a=3}\delta_{c=2} + \delta_{a=2}$. To create a $\mb{2} + \mb{1} + \mb{1} + \mb{1}$ coloring with no leaves in $A\cup C$ colored \textbf{0}, we can either color the unique leaves in $A$ and in $C$ with \textbf{1} and select a leaf in $B$ to color \textbf{2} (giving $\bar{b}\delta_{a=2}$), or we select one of the three leaves in $A\cup C$ to color \textbf{2} (giving $3\delta_{a=3}\delta_{c=2}$). To create a $\mb{2} + \mb{2} + \mb{1}$ coloring with no leaves in $A\cup C$ colored \textbf{0}, we must color the unique leaves in $A$ and in $C$ with \textbf{2}, giving $\delta_{a=2}$. 
    \end{itemize}

     The remaining spine colorings use three distinct colors, which means to complete a $\mb{2} + \mb{1} + \mb{1} + \mb{1}$ coloring we need to select two vertices to color \textbf{1}. To complete a $\mb{2} + \mb{2} + \mb{1}$ coloring we need to select one vertex to color \textbf{2}. This creates symmetry in the following arguments, with the few differences coming from the assumption that $a\geq c \geq 2$ whereas $b\geq 1$. 
\begin{itemize}
    \item $\mathbf{\langle 1, 0, 2\rangle}$ contributes $\left(\binom{\bar{c}}{2} + \bar{a} \right)\delta_{b=1} + (\bar{c}+1)\delta_{b=2} + \delta_{b=3}$. To create a $\mb{2} + \mb{1} + \mb{1} + \mb{1}$ coloring with no leaves in $B$ colored \textbf{0}, we must have all leaves in $B$ colored \textbf{1} and select the appropriate number of leaves from $C$ to color \textbf{1}, giving $\binom{\bar{c}}{2}\delta_{b=1} + \bar{c}\delta_{b=2} + \delta_{b=3}$. To create a $\mb{2} + \mb{2} + \mb{1}$ coloring with no leaves in $B$ colored \textbf{0}, we either select one leaf from $A$ to color \textbf{2} if $B$ is empty, or the unique leaf in $B$ must be colored \textbf{2}, giving $\bar{a}\delta_{b=1} + \delta_{b=2}$.

    \item $\mathbf{\langle 2, 0, 1\rangle}$ contributes $\left(\binom{\bar{a}}{2} + \bar{c} \right)\delta_{b=1} + (\bar{a}+1)\delta_{b=2} + \delta_{b=3}$, following a similar reason as $\mathbf{\langle 1, 0, 2\rangle}$ by permuting $a$ and $c$.

    \item $\mathbf{\langle 1, 2, 0\rangle}$ contributes $(\bar  {b} + 1)\delta_{c=2} + \delta_{c=3}$. To create a $\mb{2} + \mb{1} + \mb{1} + \mb{1}$ coloring with no leaves in $C$ colored \textbf{0}, we must have all leaves in $C$ colored \textbf{1} and select the appropriate number of leaves from $B$ to color \textbf{1}, giving $\bar{b}\delta_{c=2} + \delta_{c=3}$. To create a $\mb{2} + \mb{2} + \mb{1}$ coloring with no leaves in $C$ colored \textbf{0}, we must have the unique leaf in $C$ colored \textbf{2}, giving $\delta_{c=2}$. 

    \item $\mathbf{\langle 2, 1, 0\rangle}$ contributes $(\bar{a} + 1)\delta_{c=2} + \delta_{c=3}$. To create a $\mb{2} + \mb{1} + \mb{1} + \mb{1}$ coloring with no leaves in $C$ colored \textbf{0}, we must have all leaves in $C$ colored \textbf{1} and select the appropriate number of leaves from $A$ to color \textbf{1}, giving $\bar{a}\delta_{c=2} + \delta_{c=3}$. To create a $\mb{2} + \mb{2} + \mb{1}$ coloring with no leaves in $C$ colored \textbf{0}, we must have the unique leaf in $C$ colored \textbf{2}, giving $\delta_{c=2}$. 

    \item $\mathbf{\langle 0, 2, 1\rangle}$ contributes $(\bar{b}+1)\delta_{a=2} + \delta_{a=3}$, following a similar reason as $\mathbf{\langle 1, 2, 0\rangle}$ by permuting $a$ and $c$.

    \item $\mathbf{\langle 0, 1, 2\rangle}$ contributes $2\delta_{a=2} + \delta_{a=3}$, following a similar reason as $\mathbf{\langle 2, 1, 0\rangle}$ by permuting $a$ and $c$.
\end{itemize}
    
\end{proof}

\subsection{Proof of Proposition~\ref{prop:coeff-of-q6t3}} \label{appendix:coeff-of-q6-t3} 
Proposition~\ref{prop:coeff-of-q6t3} states that if $b \geq 2$, then 
\begin{align*}
      [q^6t^3] =  &\left({\bar{b} \choose 2} + 2(\bar{a}+\bar{c})\right) + 
        \left(4(\bar{a}+\bar{c}) + {\bar{a} \choose 2} + {\bar{c} \choose 2}\right)\delta_{b = 2}\\ + 
        &\left(2(\bar{a}+ \bar{c}) + 6\right)\delta_{b = 3} +
         5\delta_{b = 4} + \left(2{\bar{b} \choose 2} + 6\bar{b} +2\right)\delta_{a = 2}\delta_{c = 2} +\left(3\bar{b} + 11\right)\delta_{a = 3}\delta_{c = 2}\\
         + &\left(\bar{b} + 11\right)\delta_{a = 3}\delta_{c = 3}  
        + 5\delta_{a = 4}\delta_{c = 2} + 
        2\delta_{a = 4}+ 
        \left(2(\bar{a}+\bar{b}) + {\bar{a} \choose 2} + {\bar{b} \choose 2}\right)\delta_{c = 2} \\
        &+ (\bar{a}+\bar{b} + 4)\delta_{c = 3} + 2 \delta_{c = 4}.
    \end{align*}

    \begin{proof}[Proof of Proposition~\ref{prop:coeff-of-q6t3}]
    In order to obtain a weight $6$ proper coloring while using exactly three colors, we can have $6 = \mb{2} + \mb{2} + \mb{1}+ \mb{1}$ or $6 = \mb{2} + \mb{1} + \mb{1}+ \mb{1}+\mb{1}$. We examine the
contribution of each possible spine coloring to the coefficient $[q^6t^3]$ below.

    \begin{itemize}
        \item $\mb{\lan 1,2,1\ran}$ contributes $\binom{\bar{b}}{2}+{\bar{a}+\bar{c}}$. To create a $\mb{2} + \mb{2} + \mb{1}+ \mb{1}$ coloring, we choose one leaf from $A \cup C$  to color \textbf{2}. To create a $\mb{2} + \mb{1} + \mb{1} + \mb{1}+ \mb{1}$ coloring, we can choose two leaves from $B$ to color \textbf{1}.

        \item $\mb{\lan 2,1,2\ran}$ contributes $\bar{a}+\bar{c}$. We must create a $\mb{2} + \mb{2} + \mb{1}+ \mb{1}$ coloring. We choose one leaf from $A \cup C$ to color \textbf{1}.

        \item $\mb{\lan 2,0,2\ran}$ contributes $(\bar{a}+\bar{c})\delta_{b=2}+\delta_{b=3}$. We must create a $\mb{2} + \mb{2} + \mb{1}+ \mb{1}$ coloring, and have no leaf in $B$ colored \textbf{0}. If $|B| = 1$, then we color the unique leaf in $B$ with \textbf{1} and select a leaf in $A \cup C$ to color \textbf{1}, giving $(\bar{a}+\bar{c})\delta_{b=2}$. If $|B| = 2$, we color the two leaves in $B$ with \textbf{1}, giving $\delta_{b=3}$. 
        
        \item $\mb{\lan 0,2,0 \ran}$ contributes $(\binom{\bar{b}}{2}+2\bar{b})\delta_{a=2}+(\bar{b}+3)\delta_{a=3}\delta_{c=2}+\delta_{a=3}\delta_{c=3}+\delta_{a=4}\delta_{c=2}$. 

        To create a $\mb{2} + \mb{2} + \mb{1}+ \mb{1}$ coloring with no leaf in $A \cup C$ colored \textbf{0}, we must have $2 \leq|A \cup C| \leq 3$. If $|A \cup C| = 2$, then we choose one leaf in $A\cup C$ to color \textbf{2}, we color the remaining leaf in $A \cup C$ with \textbf{1}, and we choose one leaf in $B$ to color \textbf{1}, giving $2\bar{b}\delta_{a=2}$. If $|A| = 2$ and $|C|=1$, then we choose one leaf in $A \cup C$ to color \textbf{2} and color the remaining leaves in $A \cup C$ with \textbf{1}, giving $3\delta_{a=3}\delta_{c=2}$.

        To create a $\mb{2} + \mb{1} + \mb{1} + \mb{1}+ \mb{1}$ coloring with all of the leaves in $A\cup C$ colored \textbf{1}, we must have $2 \leq |A \cup C | \leq 4$. If $|A \cup C| = 2$, then we choose $2$ leaves in $B$ to color \textbf{1}, giving $\binom{\bar{b}}{2}\delta_{a=2}$. If $|A \cup C| = 3$, then we choose one leaf in $B$ to color \textbf{1}, giving $\bar{b}\delta_{a=3}\delta_{c=2}$. If $|A \cup C| = 4$, then we color all leaves in $B$ with \textbf{0}, giving $\delta_{a=4}\delta_{c=2} + \delta_{a=3}\delta_{c=3}$.
        
       \item $\mb{\lan 1,0,1\ran}$ contributes $(\bar{a}+\bar{c})\delta_{b=2}+(\bar{a}+\bar{c}+1)\delta_{b=3}+3\delta_{b=4}$.

       To create a $\mb{2} + \mb{2} + \mb{1}+ \mb{1}$ coloring with all of the leaves in $B$ colored \textbf{2}, we must have $1 \leq |B| \leq 2$. If $|B| = 1$, we choose one leaf from $A\cup C$ to color $2$, giving $(\bar{a}+\bar{c})\delta_{b=2}$. If $|B| = 2$, then we color all of the leaves in $B$ with \textbf{2}, giving $\delta_{b=3}$. 

       To create a $\mb{2} + \mb{1} + \mb{1} + \mb{1}+ \mb{1}$ coloring with no leaf in $B$ colored \textbf{0} and no leaf in $A \cup C$ colored \textbf{1}, we can either select a leaf in $A \cup C$ to color \textbf{2} (giving $(\bar{a}+\bar{c})\delta_{b=3}$) or we select a leaf in $B$ to color \textbf{2} (giving $3\delta_{b=4}$).
    
        \item $\mb{\lan 0,1,0 \ran}$ contributes $2\bar{b}\delta_{a=2}+(\bar{ b}+3)\delta_{a=3}\delta_{c=2}+4\delta_{a=3}\delta_{c=3}+4\delta_{a=4}\delta_{c=2}$. 

        To create a $\mb{2} + \mb{2} + \mb{1}+ \mb{1}$ coloring with no leaves in $A \cup C$ colored \textbf{0}, we must have $2 \leq |A\cup C|\leq 3$. If $|A \cup C| = 2$, then we select one leaf in $B$ to color \textbf{2} and one leaf in $A \cup C$ to color \textbf{2}, giving $2\bar{b}\delta_{a=2}$. If $|A \cup C| = 3$, then we choose one leaf in $A \cup C$ to color \textbf{1}, giving $3\delta_{a=3}\delta_{c=2}$. 

        To create a $\mb{2} + \mb{1} + \mb{1} + \mb{1}+ \mb{1}$ coloring with no leaves in $A \cup C$ colored \textbf{0} and no leaves in $B$ colored \textbf{1}, we must have $3 \leq |A \cup C| \leq 4$. If $|A\cup C| = 3$, then we select one leaf from $B$ to color \textbf{2}, giving $\bar{b}\delta_{a=3}\delta_{c=2}$. If $|A\cup C| = 4$, then we choose one leaf in $A \cup C$ to color \textbf{2}, giving $4\delta_{a=4}\delta_{c=2} + 4\delta_{a=3}\delta_{c=3}$.
\end{itemize}

        The remaining spine colorings use three distinct colors. This means that one of the sets of leaves ($A$, $B$, or $C$) cannot have any leaves colored \textbf{0}, which restricts the number of leaves in that set. The arguments for the following contributions are analogous to the arguments made in the first case $\mb{\lan 1,0,2\ran}$.
\begin{itemize}
   
        \item $\mb{\lan 1,0,2\ran}$ contributes $(\bar{a}+\bar{c}+\binom{\bar{c}}{2})\delta_{b=2}+(\bar{c}+2)\delta_{b=3}+\delta_{b=4}$. 

        To create a $\mb{2} + \mb{2} + \mb{1}+ \mb{1}$ coloring with no leaves in $B$ colored \textbf{0}, we must have $1 \leq |B| \leq 2$. If $|B| = 1$, then we can color the leaf in $B$ with \textbf{1} (or \textbf{2}). Then we must choose a leaf in $A$ to color \textbf{2} (or choose a leaf in $C$ to color \textbf{1}), giving $\bar{a}\delta_{b=2} + \bar{c}\delta_{b=2}$. The case when $|B| = 2$ gives $2\delta_{b=3}$.
        
        To create a $\mb{2} + \mb{1} + \mb{1} + \mb{1}+ \mb{1}$ coloring with no leaves in $B$ colored \textbf{0} and no leaves in $A$ colored \textbf{1}, we must color all of the leaves in $B$ with \textbf{1} and select the appropriate number of leaves in $C$ to color \textbf{1}, giving $\binom{\bar{c}}{2}\delta_{b=2} + \bar{c}\delta_{b=3} + \delta_{b=4}$.

        \item $\mb{\lan 2,0,1\ran}$ contributes $(\bar{a}+\bar{c}+\binom{\bar{a}}{2})\delta_{b=2}+(\bar{a}+2)\delta_{b=3}+\delta_{b=4}$. Identical to the $\mb{\lan 1,0,2 \ran}$ case after permuting $a$ and $c$.

        \item $\mb{\lan 1,2,0\ran}$ contributes $(\bar{a}+\bar{b}+\binom{\bar{b}}{2})\delta_{c=2}+(\bar{b}+2)\delta_{c=3}+\delta_{c=4}$.        Identical to the $\mb{\lan 1,0,2\ran}$ case after permuting $b$ and $c$.

        \item $\mb{\lan 2,1,0\ran}$ contributes $(\bar{a}+\bar{b}+\binom{\bar{a}}{2})\delta_{c=2}+(\bar{a}+2)\delta_{c=3}+\delta_{c=4}$. Identical to the $\mb{\lan 1,2,0\ran}$ case after permuting $a$ and $b$.

        \item $\mb{\lan 0,2,1\ran}$ contributes $(\bar{b}+1+\binom{\bar{b}}{2})\delta_{a=2}+(\bar{b}+2)\delta_{a=3}+\delta_{a=4}$.         Identical to the $\mb{\lan 1,2,0 \ran}$ case after permuting $a$ and $c$ and recalling that $a=2$ implies $c=2$.

           \item $\mb{\lan 0,1,2\ran}$ contributes $(\bar{b}+1)\delta_{a=2}+(\bar{c}+2)\delta_{a=3}+\delta_{a=4}$.           Identical to the $\mb{\lan 1, 0, 2 \ran}$ case after permuting $a$ and $b$ and recalling that $a=2$ implies $c=2$.

\end{itemize}
  Adding the contributions gives the claimed formula.
 \end{proof}

    \subsection{Proof of Proposition~\ref{prop:q_m_plus_2}} \label{proof.q_m_plus_2}
    The proof of Proposition~\ref{prop:q_m_plus_2} requires the coefficient $ [q^{m+2}t^3]$ when $a,b,c \geq m$ for some $m \geq 3$. We prove the formula 
    \begin{align*}
       [q^{m+2}t^3] & =   2\delta_{a=m} + 2\delta_{c=m} + (\bar{b}+2)\delta_{b=m} \\
       & \quad + \sum_{j=0}^{\frac{m-2}{2}}\binom{\bar{a}+\bar{c}}{j}\binom{\bar{b}}{m-2-2j} + \sum_{j=0}^{\frac{m-3}{2}}\binom{\bar{b}}{j}\binom{\bar{a}+\bar{c}}{m-3-2j}.
    \end{align*}

\begin{proof}

We examine the contribution of each possible spine coloring to the coefficient $[q^{m+2}t^3]$ below. Note that $m-1 \leq |A|, |B|, |C|$ and that we must use exactly three colors in each coloring. 

\begin{itemize}
     \item $\mb{\lan 1,2,1\ran}$ contributes $\sum_{j=0}^{\frac{m-2}{2}}\binom{\bar{a}+\bar{c}}{j}\binom{\bar{b}}{m-2-2j}$. The sum of the colors on the leaves must be $m-2$. Pick $0 \leq j \leq \lfloor \frac{m-2}{2} \rfloor$ leaves in $A \cup C$ to color \textbf{2}, then pick $m-2-2j$ leaves in $B$ to color \textbf{1}. Furthermore, at least one leaf in $A\cup C$ is guaranteed to be colored \textbf{0} because $\lfloor \frac{m-2}{2} \rfloor < 2(m-1) \leq |A\cup C|$.

        \item $\mb{\lan 2,1,2\ran}$ contributes $\sum_{j=0}^{\frac{m-3}{2}}\binom{\bar{b}}{j}\binom{\bar{a}+\bar{c}}{m-3-2j}$. The sum of the colors on the leaves must be $m-3$. Pick $0 \leq j \leq \lfloor \frac{m-3}{2} \rfloor$ leaves in $B$ to color \textbf{2}, then pick $m-3-2j$ leaves in $A\cup C$ to color \textbf{1}. Furthermore, at least one leaf in $B$ is guaranteed to be colored \textbf{0} because $\lfloor \frac{m-3}{2} \rfloor < m-1 \leq |B|$.
        
         \item $\mb{\lan 2,0,2\ran}$ contributes 0. The sum of the colors on the leaves must be $m-2$ and none of the leaves in $B$ can be colored \textbf{0}. Since $B$ contains at least $m-1$ leaves, we cannot create a proper coloring. 

           \item $\mb{\lan 0,2,0 \ran}$ contributes 0. The sum of the colors on the leaves must be $m$ and none of the leaves in $A\cup C$ can be colored \textbf{0}. Since $A\cup C$ contains at least $2(m-1)$ leaves and $m \geq 3$, we cannot create a proper coloring.

                   \item $\mb{\lan 1,0,1\ran}$ contributes $\bar{b}\delta_{b=m}$. The sum of the colors on the leaves must be $m$ and none of the leaves in $B$ can be colored \textbf{0}. Thus, $m-1 \leq |B| \leq m$. If $|B| = m-1$, we select one leaf in $B$ to color \textbf{2} and color the remaining leaves in $B$ with \textbf{1}, giving $\bar{b}\delta_{b=m}$. If $|B| = m$, then we need to color all leaves in $B$ with \textbf{1}, which produces a coloring with only two colors, a contradiction.  

        \item $\mb{\lan 0,1,0 \ran}$ contributes $0$. The sum of the colors on the leaves must be $m+1$ and none of the leaves in $A \cup C$ can be colored \textbf{0}. Thus, $2(m-1) \leq |A \cup C| \leq m+1$. Since $m \geq 3$, the only option is when $m=3$ and $|A \cup C| = 4$. However, if $|A \cup C| = 4$, then we need to color all of the leaves in $A \cup C$ with \textbf{1}. This produces a coloring with only two colors, a contradiction.

    \end{itemize}

    The spine colorings with three unique colors
    \[
    \mb{\lan 1,0,2\ran}, \mb{\lan 2,0,1\ran}, \mb{\lan 1,2,0\ran}, \mb{\lan 2,1,0\ran},\mb{\lan 0,2,1\ran}, \mb{\lan 0,1,2\ran}
    \]
each require that the sum of the colors on the leaves must be $m-1$ and that the leaves in some $D \in \{A, B, C\}$ cannot be colored \textbf{0}. Thus, $m-1 \leq |D| \leq m-1$, which means that all leaves in $D$ must be colored \textbf{1}, giving $\delta_{|D| = m-1}$. In total, these six cases contribute $2\delta_{a=m} + 2\delta_{b=m} + 2\delta_{c=m}$. Adding up the contributions from these twelve cases gives the claimed formula.
\end{proof}

 \subsection{Proof of Proposition~\ref{prop:q_m_plus_4}} \label{proof.q_m_plus_4}

    Proposition~\ref{prop:q_m_plus_4} examines the case where $b$ is the unique minimum value among $(a,b,c)$ and states that if $a \geq c > b \geq 3$, then 

    \[
    [q^{b+4}t^3] = G(\bar{a}+\bar{c},\bar{b})+\psi(a,c;b)+ \rho(a,c;b),
    \]
    where 

\begin{align}\label{eq:Gqm4_append}
    G(\bar{a}+\bar{c},\bar{b}) &= \binom{\bar{b}}{3} +\bar{b}\left(\bar{a} + \bar{c}\right)  + (\bar{a} + \bar{b}+\bar{c})
    +\sum_{j=1}^{\frac{\bar{b}+1}{2}}\binom{\bar{a}+\bar{c}}{j}\binom{\bar{b}}{\bar{b}+1-2j}+\sum_{j=0}^{\frac{\bar{b}}{2}}\binom{\bar{b}}{j}\binom{\bar{a}+\bar{c}}{\bar{b}-2j}
\end{align}

\begin{align}\label{eq:psiqm4_append}
\psi(a,c;b)=
   \begin{cases}
  8b-2  &    a = c = b + 1\\
  4b+2  &    a = b + 2, c = b + 1\\
  a+3b-2 &  a > b + 2, c = b + 1\\
   4    &   a = c = b + 2\\
   2    &   a > b + 2 = c\\
   0    &   a,c > b+2,
\end{cases}
\end{align}
and 
\begin{align}\label{eq:rhoqm4_append}
     \rho(a,c;b)= \bar{a}+\bar{c}+\binom{\bar{a}+\bar{b}}{2}+\binom{\bar{c}+\bar{b}}{2}.\end{align}

\begin{proof}[Proof of Proposition~\ref{prop:q_m_plus_4}]

We examine the contribution of each possible spine coloring to the coefficient $[q^{b+4}t^3]$ below. We must use exactly three colors in each coloring. 

\begin{itemize}
    \item $\mb{\langle 1,2,1\rangle}$ contributes  $\sum_{j=1}^{\frac{\bar{b}+1}{2}}\binom{\bar{a}+\bar{c}}{j}\binom{\bar{b}}{\bar{b}+1-2j}$ in $G(\bar{a}+\bar{c}, \bar{b})$. The sum of the colors on the leaves must be $b$. We first pick $0 \leq j \leq b/2$ leaves in $A \cup C$ to color \textbf{2}, and then pick $b-2j$ leaves in $B$ to color \textbf{1}. The remaining vertices are colored \textbf{0}, and in all of these colorings at least one vertex is colored \textbf{0}. Note that $\binom{\bar{b}}{b} = 0$, hence the summation may start at $j=1$.

    \item $\mb{\langle 2,1,2\rangle}$ contributes $\sum_{j=0}^{\frac{\bar{b}}{2}}\binom{\bar{b}}{j}\binom{\bar{a}+\bar{c}}{\bar{b}-2j}$ in $G(\bar{a}+\bar{c}, \bar{b})$. The sum of colors on the leaves must be $b-1$. We first pick $0 \leq j \leq \bar{b}/2$ leaves in $B$ to color \textbf{2}, and then pick $b-1 - 2j$ leaves in $A \cup C$ to color \textbf{1}. The remaining vertices are colored \textbf{0}, and in all of these colorings at least one vertex is colored \textbf{0}.

    \item $\mb{\langle 2,0,2\rangle}$ contributes $\bar{b} + \bar{a}+ \bar{c}$ in $G(\bar{a}+ \bar{c},\bar{b})$. The sum of the colors on the leaves must be $b$, and all leaves in $B$ must be colored either \textbf{1} or \textbf{2}. The leaves in $B$ can have at most one vertex colored \textbf{2}, since then the sum of colors on the leaves in $B$ would be $2\cdot 1+ 1 \cdot(b-2) = b$. This type of proper coloring can arise in $\bar{b}$ ways. If none of the leaves in $B$ are colored \textbf{2}, then the sum of colors on the leaves in $B$ is $b-1$ and we must pick a leaf in $A \cup C$ to color \textbf{1}. This type of proper coloring can arise in $\bar{a} + \bar{c}$ ways.

    \item $\mb{\langle 0,2,0\rangle}$ contributes 0. The sum of the colors on the leaves must be $b+2$, and all leaves in $A \cup C$ must be colored either \textbf{1} or \textbf{2}. Thus, $|A \cup C| \leq b+2$. However, since $b \leq |C| \leq |A|$, we must have $2b \leq |A \cup C| < b+3$, which is never satisfied for $b \geq 3$.

    \item $\mb{\langle 1,0,1\rangle}$ contributes $\bar{b}(\bar{a} + \bar{c}) + \binom{\bar{b}}{3}$ in $G(\bar{a}+ \bar{c},\bar{b})$. The sum of the colors on the leaves must be $b+2$, and all leaves in $B$ must be colored either \textbf{1} or \textbf{2}. The leaves in $B$ can have at most three vertices colored \textbf{2}, since then the sum of colors on the leaves in $B$ would be $2\cdot 3 + 1 \cdot(b-4) = b+2$. This type of proper coloring can arise in $\binom{\bar{b}}{3}$ ways. If there is one leaf in $B$ colored \textbf{2}, then the sum of colors on the leaves in $B$ would be $2 \cdot 1 + 1\cdot (b-2) = b$. Then there are $\bar{a} + \bar{c}$ ways to pick a leaf in $A \cup C$ to color \textbf{2} and $\bar{b}$ ways to pick a leaf in $B$ to color \textbf{2}, giving $\bar{b}(\bar{a} + \bar{c}$). 
    However, if there are two leaves or zero leaves in $B$ colored \textbf{2}, we would need the colors on the leaves in $A \cup C$ to sum to an odd number, which is impossible since all the leaves in $A \cup C$ must be colored either \textbf{2} or \textbf{0}. 
    
    \item $\mb{\langle 0,1,0\rangle}$ 
    contributes 0. The sum of the colors on the leaves must be $b+3$, and all leaves in $A \cup C$ must be colored either \textbf{1} or \textbf{2}. Thus, $|A \cup C| < b+3$ since for a proper coloring using exactly three colors, we must have at least one leaf in $A \cup C$ colored \textbf{2}. However, since $b \leq |C| \leq |A|$, we must have $ 2b \leq |A \cup C| < b+3$, which is never satisfied for $b \geq 3$. 

\end{itemize}
The remaining six spine colorings contribute either to $\rho(a,c;b)$ or $\psi(a,c;b)$. In the following cases, the sum of the colors on the leaves must be $b+1$. 
\begin{itemize}
    
    \item $\mb{\langle 1,0,2\rangle}$ contributes $\bar{a} + \binom{\bar{b} + \bar{c}}{2}$ in $\rho(a,c; b)$.
All of the leaves in $B$ must be colored \textbf{1} or \textbf{2}. If all of the leaves in $B$ are colored \textbf{1}, then we must either pick one leaf from $A$ to color \textbf{2} or pick two leaves from $C$ to color \textbf{1}, giving $\bar{a} + \binom{\bar{c}}{2}$. If one leaf in $B$ is colored \textbf{2} and the remaining $b-2$ leaves in $B$ are colored \textbf{1}, then we must pick one leaf from $C$ to color \textbf{1}, giving $\bar{b}\bar{c}$. If two leaves in $B$ are colored \textbf{2} and the remaining $b-3$ leaves in $B$ are colored \textbf{1}, then we must have all leaves in $A \cup C$ colored \textbf{0}, giving $\binom{\bar{b}}{2}$. Note that $\binom{\bar{c}}{2} + \bar{b}\bar{c} + \binom{\bar{b}}{2} = \binom{\bar{c} + \bar{b}}{2}$.

    \item $\mb{\langle 2,0,1\rangle}$ contributes $\bar{c} + \binom{\bar{a} + \bar{b}}{2}$ in $\rho(a,c; b)$. The arguments are identical to the $\mb{\langle 1,0,2\rangle}$ case after permuting $a$ and $c$.

    \item $\mb{\langle 1,2,0\rangle}$ contributes $(\bar{b}+\bar{c})\delta_{c=b+1} + \delta_{c=b+2}$ in $\psi(a,c;b)$. Since all leaves in $C$ must be colored \textbf{1} or \textbf{2} and $b \leq |C|$, we must have $b \leq |C| \leq b+1$. If $|C| = b$ and all leaves in $C$ are colored \textbf{1}, then we must pick one leaf from $B$ to color \textbf{2}, giving $\bar{b}$. If $|C| = b$ and one leaf in $C$ is colored \textbf{2} and the remaining $b-1$ leaves are colored \textbf{1}, then all leaves in $A \cup B$ must be colored \textbf{0}, giving $\bar{c}$. If $|C| = b+1$, then all leaves in $C$ must be colored \textbf{1}, giving $1$.

    \item $\mb{\langle 2,1,0\rangle}$ contributes $(\bar{a}+\bar{c})\delta_{c=b+1} + \delta_{c=b+2}$ in $\psi(a,c;b)$. The arguments are identical to the $\mb{\langle 1,2,0\rangle}$  case after permuting $a$ and $b$.

    \item $\mb{\langle 0,2,1\rangle}$ contributes $(\bar{b}+\bar{a})\delta_{a=b+1} + \delta_{a=b+2}$ in $\psi(a,c;b)$. The arguments are identical to the $\mb{\langle 1,2,0\rangle}$  case after permuting $a$ and $c$. 
    
    \item $\mb{\langle 0,1,2\rangle}$ contributes $(\bar{c}+\bar{a})\delta_{a=b+1} + \delta_{a=b+2}$ in $\psi(a,c;b)$. The arguments are identical to the $\mb{\langle 2,1,0\rangle}$ case after permuting $a$ and $c$. 
\end{itemize}

These last four spine colorings give 
\[
\psi(a,c;b) = (2\bar{a} + \bar{b} + \bar{c})\delta_{a = b+1} + (\bar{a} + \bar{b} + 2\bar{c})\delta_{c = b+1} + 2\delta_{a=b+2} + 2 \delta_{c=b+2},
\]
which is equivalent to the six cases presented in Equation~\eqref{eq:psiqm4_append}.

\end{proof}

\bibliographystyle{plain}
\bibliography{q_refs.bib}

\end{document}